\documentclass{amsart}

\usepackage{amssymb,amsmath,amsthm,latexsym,booktabs}

\theoremstyle{definition}
\newtheorem{definition}{Definition}[section]
\newtheorem{remark}[definition]{Remark}
\newtheorem{example}[definition]{Example}
\newtheorem{exercise}[definition]{Exercise}
\newtheorem{openproblem}[definition]{Open Problem}

\theoremstyle{plain}
\newtheorem{lemma}[definition]{Lemma}
\newtheorem{proposition}[definition]{Proposition}
\newtheorem{theorem}[definition]{Theorem}
\newtheorem{corollary}[definition]{Corollary}
\newtheorem{conjecture}[definition]{Conjecture}

\begin{document}

\title[Gr\"obner Bases and Universal Envelopes]
{Free Associative Algebras,
Noncommutative Gr\"obner Bases, and
Universal Associative Envelopes for Nonassociative Structures}

\author{Murray R. Bremner}

\address{Department of Mathematics and Statistics, 
University of Saskatchewan, Canada}

\email{bremner@math.usask.ca}

\urladdr{math.usask.ca/~bremner}

\begin{abstract}
These are the lecture notes from my short course of the same title 
at the CIMPA Research School on 
\emph{Associative and Nonassociative Algebras and Dialgebras: 
Theory and Algorithms - In Honour of Jean-Louis Loday (1946--2012)}, 
held at CIMAT, Guanajuato, Mexico, February 17 to March 2, 2013.
The underlying motivation is to apply the theory of noncommutative Gr\"obner bases 
in free associative algebras to the construction of 
universal associative envelopes for nonassociative structures 
defined by multilinear operations.
Trilinear operations were classified by the author and Peresi in 2007.
In her Ph.D. thesis of 2012, Elgendy studied the universal associative envelopes
of nonassociative triple systems obtained by applying these
trilinear operations to the 2-dimensional simple associative triple system.
In these notes I use computer algebra 
to extend some aspects of her work to the 4-dimensional and 6-dimensional 
simple associative triple systems.
\end{abstract}

\maketitle

\allowdisplaybreaks


\section{Introduction}

The primary goal of these lecture notes is to apply the theory of noncommutative 
Gr\"obner bases in free associative algebras to the construction of universal 
associative envelopes for nonassociative structures defined by multilinear operations.
Throughout I will take an algorithmic approach, developing just enough theory to motivate 
the computational methods.  
Some of the easier proofs and examples are left as exercises for the reader.
Along the way, I will mention a number of open research problems.
I begin by recalling the basic definitions of the most familiar examples of 
nonassociative structures: finite dimensional Lie and Jordan algebras and their 
universal associative enveloping algebras.
Unless otherwise indicated, I will work over an arbitrary field $F$.

\subsection{Lie algebras}

Lie algebras are defined by the polynomial identities of degree $\le 3$ satisfied by the Lie bracket
$[x,y] = xy - yx$ in every associative algebra, namely anticommutativity and the Jacobi identity:
  \[
  [x,x] \equiv 0, 
  \qquad
  [[x,y],z] + [[y,z],x] + [[z,x],y] \equiv 0.
  \]
Every polynomial identity satisfied by the Lie bracket in every associative algebra 
is a consequence of these two identities;
see Corollary \ref{corollarylie}.

\begin{definition}
Let $A$ be an associative algebra with product denoted $xy$.
We write $A^-$ for the Lie algebra which has the same underlying vector space as $A$, 
but the original associative operation is replaced by the Lie bracket $[x,y] = xy - yx$.
Let $L$ be a Lie algebra over $F$.
If $L$ is isomorphic to a subalgebra of $A^-$ then we call $A$ an \textbf{associative envelope} for $L$.
\end{definition}

\begin{example}
Let $L = \mathfrak{sl}_n(F)$ be the special linear Lie algebra of all $n \times n$ matrices of trace 0 over $F$.
Then clearly $L$ is a subalgebra of $A^-$ where $A = M_n(F)$ is the associative algebra of all $n \times n$ matrices.
\end{example}

\begin{definition} \label{definitionUL}
The \textbf{universal associative envelope} $U(L)$ of the Lie algebra $L$ is the unital associative algebra 
satisfying the following \textbf{universal property}, which implies that $U(L)$ is unique up to isomorphism:
  \begin{itemize}
  \item
There is a morphism of Lie algebras $\alpha\colon L \to U(L)^-$ such that
for any unital associative algebra $A$ and any morphism of Lie algebras $\beta\colon L \to A^-$,
there is a unique morphism of associative algebras $\gamma\colon U \to A$ 
satisfying $\beta = \gamma \circ \alpha$.
  \end{itemize}
In the terminology of category theory, this says that the functor 
sending a Lie algebra $L$ to its universal associative envelope $U(L)$ is the left adjoint of the functor
sending an associative algebra $A$ to the Lie algebra $A^-$.
\end{definition}

\begin{lemma} \label{lemmaUL}
The subset $\alpha(L)$ generates $U(L)$.
If $A$ is an associative envelope for $L$, and $A$ is generated by the subset $L$,
then $A$ is isomorphic to a quotient of $U(L)$; 
that is, $A \approx U(L) / I$ for some ideal $I$.
\end{lemma}

\begin{proof}
Exercise.
\end{proof}

We will see later that $U(L)$ is always infinite dimensional, and that
the map $\alpha$ is always injective, so that $L$ is isomorphic to a subalgebra of $U(L)^-$.
These are corollaries of the PBW theorem (Theorem \ref{pbwtheorem}) that we will
prove using the theory of noncommutative Gr\"obner bases.

\begin{example}
Let $L$ be the $n$-dimensional Lie algebra with basis $\{ x_1, \dots, x_n \}$
and trivial commutation relations $[ x_i, x_j ] = 0$ for all $i, j$.
Then $U(L) \approx F[ x_1, \dots, x_n ]$, the algebra of commutative associative polynomials in $n$ variables over $F$.
\end{example}

\subsection{Jordan algebras}

Assume that $\mathrm{char}\,F \ne 2$.
Jordan algebras are defined by the polynomial identities of degree $\le 4$ satisfied by the 
Jordan product $x \circ y = \frac12 ( xy + yx )$ in every associative algebra,
commutativity and the Jordan identity:
  \[
  x \circ y \equiv y \circ x,
  \qquad
  ( ( x \circ x ) \circ y ) \circ x \equiv ( x \circ x ) \circ ( y \circ x ).
  \]
In contrast to Lie algebras, 
there exist further identities satisfied by the Jordan product in every associative algebra
which are not consequences of these two identities.
The simplest such identities were discovered almost 50 years ago;
they have degree 8 and are called the Glennie identities \cite{Glennie1966}.

\begin{definition} \label{definitionUJ}
Let $A$ be an associative algebra with product denoted $xy$.
We write $A^+$ for the Jordan algebra which has the same underlying vector space as $A$, 
but the original associative operation is replaced by the Jordan product $x \circ y = \frac12 ( xy + yx )$.
Let $J$ be a Jordan algebra over $F$.
If $J$ is isomorphic to a subalgebra of $A^+$ then we call $A$ an \textbf{associative envelope} for $J$.
\end{definition}

\begin{example}
Let $S_n(F)$ be the Jordan algebra of symmetric $n \times n$ matrices with entries in $F$, 
and let $A = M_n(F)$ be the associative algebra of all $n \times n$ matrices.
\end{example}

\begin{exercise}
Modify Definition \ref{definitionUL} to define universal associative envelopes for Jordan algebras.
State and prove the analogue of Lemma \ref{lemmaUL} for Jordan algebras.
\end{exercise}

If $J$ is finite dimensional, then so is its universal associative envelope $U(J)$.
On the other hand, the natural map from $J$ to $U(J)$ may not be injective; 
hence, strictly speaking, the universal associative envelope $U(J)$
may not be an associative envelope in the sense of Definition \ref{definitionUJ}.

\begin{example}
Let $J$ be the $n$-dimensional Jordan algebra with basis $\{ x_1, \dots, x_n \}$
and trivial products $x_i \circ x_j = 0$ for all $i, j$.
Then $U(J) \approx \Lambda( x_1, \dots, x_n )$, the exterior (Grassmann) algebra on $n$ generators over $F$,
and so $\dim U(J) = 2^n$.
\end{example}

We have the following definition, which has no analogue for Lie algebras.

\begin{definition}
If a Jordan algebra $J$ has an associative envelope then we call $J$ a \textbf{special} Jordan algebra.
Otherwise, we call $J$ an \textbf{exceptional} Jordan algebra.
\end{definition} 

\begin{example}
The vector space $H_3(\mathbb{O})$ of $3 \times 3$ Hermitian matrices over the 8-dimensional 
division algebra $\mathbb{O}$ of real octonions is closed under the Jordan product 
and is a 27-dimensional exceptional Jordan algebra.
\end{example}


\section{Free Associative Algebras}

These lecture notes on the theory of noncommutative Gr\"obner bases follow closely the exposition by 
de Graaf \cite[\S\S 6.1-6.2]{deGraaf2000}.
The most famous paper on this topic is by Bergman \cite{Bergman1978}, 
but similar results were published a little earlier by Bokut \cite{Bokut1976}.
Bokut's approach was based on Shirshov's work on Lie algebras \cite{Shirshov1962}.
(Shirshov's papers have appeared recently in English translation \cite{Shirshov2009}.)
For further references, including current research directions, see Section \ref{bibliographicalremarks}.

\begin{definition}
Let $X = \{ x_1, x_2, \dots, x_n, \dots \}$ be an \textbf{alphabet}:
a set of indeterminates (sometimes called \textbf{letters}), finite or countably infinite.
We impose a total order on $X$ by setting $x_i \prec x_j$ if and only if $i < j$.
We write $X^\ast$ for the set of \textbf{words} 
(also called \textbf{monomials}) $w = x_{i_1} x_{i_2} \cdots x_{i_k}$ where
$x_{i_1}, x_{i_2}, \dots, x_{i_k} \in X$ and $k \ge 0$.
(If $k = 0$ then we have the \textbf{empty word} denoted $w = 1$.)
The \textbf{degree} of a word $w = x_{i_1} x_{i_2} \cdots x_{i_k}$ is the number of letters it contains,
counting repetitions:
$\deg(w) = k$.
We define \textbf{concatenation} on $X^\ast$ by $(u,v) \mapsto uv$ for any $u, v \in X^\ast$;
this associative operation makes $X^\ast$ into the \textbf{free monoid} generated by $X$.
\end{definition}

\begin{example} \label{noncommutative}
If $X = \{ a \}$ has only one element, then $X^\ast = \{ \, a^k \mid k \ge 0 \, \}$
is the set of all non-negative powers of $a$.  
The multiplication on $X^\ast$ is given by $a^i a^j = a^{i+j}$, so $X^\ast$ is commutative.  
If $X$ has two or more elements, then $X^\ast$ is noncommutative.
For example, if $X = \{ a, b \}$ then there are $2^k$ distinct words of degree $k$ for all $k \ge 0$:
  \[
  \begin{array}{ll}
  k = 0\colon & 1
  \\
  k = 1\colon & a, \; b
  \\
  k = 2\colon & a^2, \; ab, \; ba, \; b^2
  \\ 
  k = 3\colon & a^3, \; a^2b, \; aba, \; ab^2, \; ba^2, \; bab, \; b^2a, \; b^3
  \\
  k = 4\colon & a^4, \; a^3b, \; a^2ba, \; a^2b^2, \; aba^2, \; abab, \; ab^2a, \; ab^3,
  \\
  & ba^3, \; ba^2b, \; baba, \; bab^2, \; b^2a^2, \; b^2ab, \; b^3a, \; b^4
  \end{array}
  \]
\end{example}

\begin{definition}
A nonempty word $u \in X^\ast$ is a \textbf{subword} (also called a \textbf{factor} or a \textbf{divisor}) 
of $w \in X^\ast$ if $w = v_1 u v_2$ for some $v_1, v_2 \in X^\ast$.
If $v_1 = 1$ then $u$ is a \textbf{left} subword of $w$; if $v_2 = 1$ then $u$ is a \textbf{right} subword of $w$.
We say that $u$ is a \textbf{proper} subword of $w$ if $u \ne w$.
\end{definition}

\begin{definition} \label{deforder}
The total order on $X$ extends to a total order on $X^\ast$, called the \textbf{deglex}
(degree lexicographical) order, as follows:
If $u, w \in X^\ast$ then $u \prec w$ (we say $u$ \textbf{precedes} $w$) if and only if either
  \begin{enumerate}
  \item[(i)]
  $\deg(u) < \deg(w)$, or
  \item[(ii)]
  $\deg(u) = \deg(w)$ where $u = v x_i u'$ and $w = v x_j w'$ 
  for some $v, u', w' \in X^\ast$ and $x_i, x_j \in X$ with $x_i < x_j$.
  \end{enumerate}
In condition (ii) we find the common left subword $v$ of highest degree, and then compare the next letters
$x_i$ and $x_j$ using the total order on $X$.
We write $u \preceq v$ when $u \prec v$ or $u = v$.
We often write $v \succ u$ to mean $u \prec v$.
\end{definition}

\begin{example}
Let $X = \{ a, b \}$ with $a \prec b$.  We list the words in $X^\ast$ of degree $\le 3$ in deglex order;
this is the same order as in Example \ref{noncommutative}:
  \[
  1 \prec a \prec b \prec a^2 \prec ab \prec ba \prec b^2 \prec
  a^3 \prec
  a^2 b \prec
  aba \prec
  a b^2 \prec
  b a^2 \prec
  b a b \prec
  b^2 a \prec
  b^3.
  \]
\end{example}

\begin{exercise}
Let $X = \{ a, b, c \}$ with $a \prec b \prec c$.  List the words in $X^\ast$ of degree $\le 3$ in deglex order.
Do the same with $c \prec b \prec a$.
\end{exercise}

\begin{definition}
A total order on $X^\ast$ is \textbf{multiplicative} if 
for all $u, v, w \in X^\ast$ with $u \prec v$ we have $uw \prec vw$ and $wu \prec wv$.
(More concisely, we could require the single condition that 
$w_1 u w_2 \prec w_1 v w_2$ for all $u, v, w_1, w_2 \in X^\ast$.)
\end{definition}

\begin{definition}
A total order on $X^\ast$ satisfies the \textbf{descending chain condition} (\textbf{DCC}) if
whenever $w_1, w_2, \dots, w_n, \dots \in X^\ast$ with $w_1 \succeq w_2 \succeq \cdots \succeq w_n \succeq \cdots$ 
then for some $n$ we have $w_n = w_{n+1} = \cdots$; that is,
there do not exist infinite strictly decreasing sequences.
Equivalently, for any $w \in X^\ast$ the set $\{ v \in X^\ast \mid v \prec w \}$ is finite.
The DCC allows us to use induction on $X^\ast$ with respect to the total order.
\end{definition}

\begin{lemma}
The total order $\prec$ on $X^\ast$ from Definition \ref{deforder} is multiplicative
and satisfies the descending chain condition.
\end{lemma}

\begin{proof}
Exercise.
\end{proof}

\begin{definition}
We write $F\langle X \rangle$ for the vector space with basis $X^\ast$ over $F$.
Concatenation in $X^\ast$ extends bilinearly to $F\langle X \rangle$:
  \[
  \Big( \sum_i a_i u_i \Big) \Big( \sum_j b_j v_j \Big)
  =
  \sum_{i,j} a_i b_j u_i v_j
  \quad
  ( a_i, b_j \in F; \, u_i, v_j \in X^\ast ).
  \]
This multiplication makes $F\langle X \rangle$ into the \textbf{free associative algebra} generated by $X$ over $F$.
This is a \textbf{unital} algebra, since the empty word acts as the unit element.
Elements of $F\langle X \rangle$ are linear combinations of monomials in $X^\ast$, 
and we refer to them as \textbf{noncommutative polynomials} in 
the variables $X$ with coefficients in $F$. 
(Here \emph{noncommutative} means \emph{not necessarily commutative}.)
\end{definition}

\begin{example}
If $X = \{ a \}$ has only one element, then $F\langle X \rangle$ is the same as $F[a]$,
the familiar algebra of commutative associative polynomials in one variable. 
If $X$ has two or more elements, then $F\langle X \rangle$ and $F[X]$ do not coincide:
$F[X]$ is commutative but $F\langle X \rangle$ is noncommutative.
\end{example}

\begin{definition}
Consider a nonzero element $f \in F\langle X \rangle$.  
We write 
  \[
  f = \sum_{i \in \mathcal{I}} a_i u_i \quad ( a_i \in F; \, u_i \in X^\ast ),
  \]
where $\mathcal{I}$ is a nonempty finite index set and $a_i \ne 0$ for all $i \in \mathcal{I}$.
The \textbf{support} of $f$ is the set of all monomials occurring in $f$:
  \[
  \mathrm{support}(f) = \{ \, u_i \mid i \in \mathcal{I} \, \}.
  \]
(If $f = 0$ then by convention its support is the empty set $\emptyset$.)
For nonzero $f \in F\langle X \rangle$, the support is a nonempty finite subset of $X^\ast$;
the greatest element of $\mathrm{support}(f)$ with respect to the total order $\prec$ on $X^\ast$ is 
the \textbf{leading monomial} of $f$, denoted $LM(f)$.
The coefficient of $LM(f)$ is the \textbf{leading coefficient} of $f$,
denoted $lc(f)$.
We say that $f$ is \textbf{monic} if $lc(f) = 1$.
For any subset $S \subseteq F\langle X \rangle$, we write 
  \[
  LM(S) = \{ \, LM(f) \mid f \in S \, \}.
  \]
\end{definition}

\begin{example}
For $X = \{ a, b, c \}$ and $cab - bca + da - cb + a^2 \in F\langle X \rangle$ we have 
  \[
  \mathrm{support}(f) = \{ \, a^2, \, cb, \, da, \, bca, \, cab \},
  \qquad
  LM(f) = cab, 
  \qquad
  lc(f) = 1.
  \]
\end{example}

\begin{definition} \label{definitionstandardform}
The \textbf{standard form} of a nonzero element $f \in F\langle X \rangle$ 
consists of $f$ divided by $lc(f)$ with the monomials in reverse deglex order.
Thus the standard form is monic and the leading monomial occurs in the first (leftmost) position.
The polynomial $f$ in the previous example is in standard form.
\end{definition}


\section{Universal Associative Envelopes of Lie and Jordan Algebras}

We use the concepts of the previous section to construct
the universal associative envelopes of Lie and Jordan algebras.

\begin{definition}
Every associative algebra $A$ is isomorphic to a quotient $F\langle X \rangle / I$ 
for some set $X$ and some ideal $I \subseteq F\langle X \rangle$.
If $I$ is generated by the subset $G \subset I$ then the pair $(X,G)$
is a \textbf{presentation} of $A$ by \textbf{generators} and \textbf{relations}.
\end{definition} 

\subsection{Lie algebras}

Let $L$ be a Lie algebra of finite dimension $d$ over $F$ with basis $X = \{ \, x_1, \dots, x_d \, \}$.
The structure constants $c_{ij}^k \in F$ are given by the equations
  \[
  [ x_i, x_j ] = \sum_{k=1}^d c_{ij}^k x_k
  \qquad
  ( 1 \le i, j \le d ).
  \]
Let $F\langle X \rangle$ be the free associative algebra generated by $X$.
(By a slight abuse of notation, we regard the basis elements of $L$ as formal variables, 
but this should not cause confusion.)
Let $I$ be the ideal in $F\langle X \rangle$ generated by the $d(d{-}1)/2$ elements
  \[
  x_i x_j - x_j x_i - \sum_{k=1}^d c_{ij}^k x_k  
  \qquad
  ( 1 \le j < i \le d ).
  \]
The quotient algebra $U(L) = F\langle X \rangle / I$ is the universal associative envelope of $L$.

\begin{example} \label{Usl2example}
We consider the Lie algebra $\mathfrak{sl}_2(F)$ of $2 \times 2$ matrices of trace 0 over a field $F$ of characteristic 0.
We use the following notation for basis elements:
  \[
  h = E_{11} - E_{22} = \left[ \begin{array}{rr} 1 & 0 \\ 0 & -1 \end{array} \right],
  \quad
  e = E_{12} = \left[ \begin{array}{rr} 0 & 1 \\ 0 & 0 \end{array} \right],
  \quad
  f = E_{21} = \left[ \begin{array}{rr} 0 & 0 \\ 1 & 0 \end{array} \right].
  \]
The structure constants are given by these equations:
  \[
  [h,e] = 2e, \qquad [h,f] = -2f, \qquad [e,f] = h.
  \]
From these equations we obtain the following set of generators $G$ for the ideal $I$:
  \[
  he - eh - 2e, 
  \qquad
  hf - fh + 2f, 
  \qquad
  ef - fe - h.
  \]
The universal associative envelope of $\mathfrak{sl}_2(F)$ is 
the quotient $U(\mathfrak{sl}_2(F)) = F\langle h, e, f \rangle / I$.
\end{example}

\subsection{Jordan algebras}

If $J$ is a Jordan algebra with structure constants
  \[
  x_i \circ x_j = \sum_{k=1}^d c_{ij}^k x_k  
  \qquad
  ( 1 \le i, j \le d ),
  \]
then we consider the ideal $I$ generated by the $d(d{+}1)/2$ elements
  \[
  \tfrac12 ( x_i x_j + x_j x_i ) - \sum_{k=1}^d c_{ij}^k x_k  
  \qquad
  ( 1 \le j \le i \le d ),
  \]
and $U(J) = F\langle X \rangle / I$ is the universal associative envelope of $J$.

\begin{example} \label{jordansymmetric}
We consider the Jordan algebra $S_2(F)$ of symmetric $2 \times 2$ matrices over a field $F$ of characteristic 0.
We use the following notation for basis elements:
  \[
  a = E_{11} = \begin{bmatrix} 1 & 0 \\ 0 & 0 \end{bmatrix},
  \quad
  b = E_{22} = \begin{bmatrix} 0 & 0 \\ 0 & 1 \end{bmatrix},
  \quad
  c = E_{12} + E_{21} = \begin{bmatrix} 0 & 1 \\ 1 & 0 \end{bmatrix}.
  \]
The structure constants are given by these equations:
  \[
  a \circ a = 2a, 
  \quad
  a \circ b = 0,
  \quad
  a \circ c = c,
  \quad
  b \circ b = 2b,
  \quad
  b \circ c = c,
  \quad
  c \circ c = 2a + 2b.
  \]
From these equations we obtain the following set $G$ of generators for the ideal $I$:
  \[ 
  a^2 - a,
  \quad  
  ba + ab,
  \quad    
  ca + ac - c,
  \quad    
  b^2 - b,
  \quad    
  cb + bc - c,
  \quad    
  c^2 - b - a.
  \]
The universal associative envelope of $S_2(F)$ is the quotient $U(S_2(F)) = F\langle a, b, c \rangle / I$.
\end{example}


\section{Normal Forms of Noncommutative Polynomials}

To understand the structure of the quotient algebra $F\langle X \rangle / I$, 
we need to find a basis for $F\langle X \rangle / I$ and express the product of any two basis elements
as a linear combination of basis elements.
This can be achieved easily if we can construct a \textit{Gr\"obner basis} for the ideal $I$:
a set of generators (not a linear basis) for $I$ with special properties
which will be explained in detail in this section and the next.

\subsection{Normal forms modulo an ideal}

A basis for $F\langle X \rangle / I$ is a subset $B$ of $F\langle X \rangle$ consisting of coset representatives:
the elements $b + I$ for $b \in B$ are linearly independent in $F\langle X \rangle / I$ and 
span $F\langle X \rangle / I$.
Equivalently, $B$ is a basis for a complement $C(I)$ to $I$ in $F\langle X \rangle$, meaning that
$F\langle X \rangle = I \oplus C(I)$, the direct sum of subspaces. 

\begin{lemma}
Assume that $I$ is an ideal in $F\langle X \rangle$, and that $B$ is a subset of $F\langle X \rangle$.
Then the set $\{ \, b + I \mid b \in B\, \}$ is a basis of the quotient $F\langle X \rangle / I$
if and only if
$B$ is a basis for a complement of $I$ in $F\langle X \rangle$;
that is, 
the elements of $B$ are linearly independent in $F\langle X \rangle$ and
$F\langle X \rangle = I \oplus \mathrm{span}(B)$.
\end{lemma}

\begin{proof}
Exercise.
\end{proof}

\begin{definition} \label{normalwordsmoduloI}
Let $I$ be an ideal in $F\langle X \rangle$.
The set $N(I)$ of \textbf{normal words} modulo $I$ is the subset of $X^\ast$ 
consisting of all monomials which are \emph{not} leading monomials of elements of $I$:
  \[
  N(I) = \{ \, w \in X^\ast \mid w \notin LM(I) \, \}.
  \]
The \textbf{complement} to $I$ in $F\langle X \rangle$ is the subspace $C(I) \subseteq F\langle X \rangle$ 
with basis $N(I)$.
\end{definition}

\begin{proposition} \label{propositionIC(I)}
We have $F\langle X \rangle = I \oplus C(I)$.
\end{proposition}

\begin{proof}
We follow de Graaf \cite[Proposition 6.1.1]{deGraaf2000} but fill in some details.
The proof consists for the most part of writing out the details in 
the division algorithm for noncommutative polynomials.

First, we prove that $I \,\cap\, C(I) = \{0\}$.
Assume that $f \in I$ and $f \in C(I)$.
If $f \ne 0$ then since $f \in I$, its leading monomial $LM(f)$ belongs to $LM(I)$;
but since $f \in C(I)$, its leading monomial belongs to $N(I)$, and hence does not belong to $LM(I)$.
This contradiction implies that $f = 0$.

Second, we prove that any $f \in F\langle X \rangle$ can be written as $f = g + h$ where $g \in I$ and $h \in C(I)$.
This is clear for $f = 0$ (take $g = h = 0$), so we assume that $f \ne 0$.
We use induction on leading monomials with respect to the total order $\prec$ on $X^\ast$.

For the basis of the induction, assume that $LM(f) = 1$ (the empty word).
Then $f = \alpha \in F \setminus \{0\}$.
If $I = F\langle X \rangle$ then $N(I) = \emptyset$ and $C(I) = \{0\}$; 
we have $f = \alpha + 0$ where $\alpha \in I$ and $0 \in C(I)$.
If $I \ne F\langle X \rangle$ then $1 \notin LM(I)$ so $1 \in N(I)$;
we have $f = 0 + \alpha$ where $0 \in I$ and $\alpha \in C(I)$.

Since $X^\ast$ satisfies the DCC, we may now assume the claim for all $f_0 \in F\langle X \rangle$ 
with $LM(f_0) \prec LM(f)$.
This is the inductive hypothesis, which depends on the fact that
only finitely many elements of $X^\ast$ precede $LM(f)$.
We have $f = \alpha LM(f) + f_0$ where $\alpha = lc(f) \in F$, and either $f_0 = 0$ or $LM(f_0) \prec LM(f)$.

If $f_0 = 0$ then $f = \alpha LM(f)$; 
if $LM(f) \in I$ then $f = \alpha LM(f) + 0 \in I + C(I)$, 
and if $LM(f) \notin I$ then $LM(f) \in N(I)$ and $f = 0 + \alpha LM(f) \in I + C(I)$.

If $f_0 \neq 0$ then $LM(f_0) \prec LM(f)$, and by induction we have $f_0 = g_0 + h_0$ where
$g_0 \in I$ and $h_0 \in C(I)$.
We now have two cases: $LM(f) \in N(I)$ and $LM(f) \notin N(I)$.
If $LM(f) \in N(I)$ then
  \[
  f = \alpha LM(f) + ( g_0 + h_0 ) = g_0 + \big( \, \alpha LM(f) + h_0 \, \big) \in I + C(I).
  \]
If $LM(f) \notin N(I)$ then by definition of $N(I)$ we have $LM(f) = LM(k)$ for some $k \in I \setminus \{0\}$.
(We cannot assume that $LM(f) \in I$.
This raises an important issue: we are non-constructively choosing an element $k \in I$ 
which has the same leading monomial as the element $f$.
Finding an algorithm to construct such an element $k$ is one of the main goals of 
the theory of noncommutative Gr\"obner bases.)

Write $k = \beta LM(k) + k_0$ where $\beta = lc(k) \in F \setminus \{0\}$, and 
either $k_0 = 0$ or $LM(k_0) \prec LM(k) = LM(f)$.
Then
  \begin{align*}
  f - \frac{\alpha}{\beta} k
  &=
  \Big( \alpha LM(f) + ( g_0 + h_0 ) \Big)
  -
  \frac{\alpha}{\beta} \Big( \beta LM(k) + k_0 \Big)
  \\
  &=
  \alpha LM(f) + g_0 + h_0 - \alpha LM(k) - \frac{\alpha}{\beta} k_0
  \\
  &=
  g_0 + h_0 - \frac{\alpha}{\beta} k_0
  \quad
  \text{since $LM(f) = LM(k)$}.
  \end{align*}
If $k_0 = 0$ then
  \[
  f = \Big( \frac{\alpha}{\beta} k + g_0 \Big) + h_0 \in I + C(I).
  \]
If $k_0 \ne 0$ then by induction $k_0 = \ell_0 + m_0$ where $\ell_0 \in I$ and $m_0 \in C(I)$.
We have
  \begin{align*}
  f 
  &= 
  \frac{\alpha}{\beta} k + g_0 + h_0 - \frac{\alpha}{\beta} k_0
  \\
  &=
  \frac{\alpha}{\beta} k + g_0 + h_0 - \frac{\alpha}{\beta} \big( \ell_0 + m_0 \big)
  \\
  &=
  \Big( \frac{\alpha}{\beta} k + g_0 - \frac{\alpha}{\beta} \ell_0 \Big)
  +
  \Big( h_0 - \frac{\alpha}{\beta} m_0 \Big).
  \end{align*}
The first three terms belong to $I$, and the last two terms belong to $C(I)$.
\end{proof}

\begin{corollary} \label{cornf}
Let $I$ be an ideal in $F\langle X \rangle$.
Then every element $f \in F\langle X \rangle$ 
has a unique decomposition $f = g + h$ where $g \in I$ and $h \in C(I)$.
\end{corollary}

\begin{proof}
This follows immediately from the definition of direct sum.
\end{proof}

\begin{definition}
For any element $f \in F\langle X \rangle$ and any ideal $I \subseteq F\langle X \rangle$, 
the element $h \in C(I)$ which is uniquely determined by Corollary \ref{cornf} 
is called the \textbf{normal form} of $f$ modulo $I$, and is denoted $N\!F_I(f)$
or $N\!F(f)$ if $I$ is understood.
\end{definition}

\begin{lemma} \label{lemmaquotientstructure}
Let $I \subseteq F\langle X \rangle$ be an ideal.
Define a product $f \cdot g$ on $C(I)$ as follows:
For any $f, g \in C(I)$ set $f \cdot g = N\!F_I(fg)$.
Then the algebra consisting of the vector space $C(I)$ with the product $f \cdot g$ is 
isomorphic to the quotient algebra $F\langle X \rangle / I$.
\end{lemma}

\begin{proof}
Exercise.
\end{proof}

Lemma \ref{lemmaquotientstructure} shows how to find a basis and structure constants
for $F\langle X \rangle / I$.
But this depends on being able to determine the basis $N(I)$ of the complement $C(I)$, 
and to calculate the normal form $N\!F_I(f)$ for every element $f \in F\langle X \rangle$.

\subsection{Computing normal forms}

Our next task is to find an algorithm for which the input is an element $f \in F\langle X \rangle$
and an ideal $I \subseteq F\langle X \rangle$ given by a set $G$ of generators,
and the output is the normal form $N\!F_I(f)$.
We present an algorithm for computing the normal form $N\!F(f,G)$ of $f$
with respect to the set $G$.
Unfortunately, the output of this algorithm depends on the set $G$;
that is, if $G_1$ and $G_2$ are two generating sets for the same ideal $I$,
then we may have $N\!F(f,G_1) \ne N\!F(f,G_2)$.
Furthermore, even for one set $G$, the output may depend on the choice of reductions performed 
at each step of the algorithm; see Example \ref{examplenormalform} below.
Therefore in general the output is not the normal form of $f$ modulo $I$.
The important property of a Gr\"obner basis is that 
if $G$ is a Gr\"obner basis for $I$ then $N\!F(f,G) = N\!F_I(f)$.

\begin{definition} \label{definitionnormalformG}
Let $f$ be an element of $F\langle X \rangle$ and let $G$ be a finite subset of $F\langle X \rangle$.
We say that $f$ is in \textbf{normal form} with respect to $G$ if the following condition holds:
  \begin{itemize}
  \item
  For every generator $g \in G$ and every monomial $w \in \mathrm{support}(f)$, 
  the leading monomial $LM(g)$ is not a subword of $w$.
  \end{itemize}
\end{definition}

We first give an informal description of the algorithm for computing the normal form of $f$ with respect to $G$.
This algorithm is similar to the calculation in the proof of Proposition \ref{propositionIC(I)};
it is a division algorithm for noncommutative polynomials.
We may assume without loss of generality that the elements of $G$ are monic.

Consider the set $LM(G)$ of leading monomials of the elements of $G$.
For each $v \in LM(G)$ and $w \in \mathrm{support}(f)$ we can easily determine if $v$ is a subword of $w$.
If this never occurs, then $f$ is in normal form with respect to $G$, 
and the algorithm terminates.
Otherwise, $w = u_1 v u_2$ for some $u_1, u_2 \in X^\ast$, and $f$ contains the term
$\alpha w$ for some $\alpha \in F \setminus \{0\}$.
There exists $g \in G$ with $LM(g) = v$; we replace $f$ by 
  \[
  f_2 = f - \alpha u_1 g u_2.
  \]
This reduction step eliminates from $f$ the term $\alpha w$. 
Repeating this procedure, 
we obtain a sequence $f_1 = f, f_2, f_3, \dots, f_n, \dots$ of elements of $F\langle X \rangle$;
this sequence converges since $X^\ast$ satisfies the DCC.
This algorithm is given in pseudocode in Figure \ref{normalformmoduloG}.

\begin{figure}[ht]
\hrule\hrule\hrule \bigskip
\begin{enumerate}
\item[]
\textbf{NormalForm$(f,G)$}
\smallskip
\item[]
\textbf{Input}:
An element $f \in F\langle X \rangle$ and a finite monic subset $G \subset F\langle X \rangle$.
\item[]
\textbf{Output}:
The normal form of $f$ with respect to $G$.
\smallskip
\item
Set $n \leftarrow 0$, $f_0 \leftarrow 0$, $f_1 \leftarrow f$.
\item
While $f_n \ne f_{n+1}$ do:
  \begin{enumerate}
  \item
  Set $n \leftarrow n+1$.
  \item
  If $w = u_1 v u_2$ for some $v \in LM(G$) and $w \in \mathrm{support}(f_n)$ then 
  \[
  \qquad 
  \text{set} \; f_{n+1} \leftarrow f_n - \alpha u_1 g u_2 \; \text{where} \; v = LM(g)
  \]  
  \item[]
  else set $f_{n+1} \leftarrow f_n$.
  \end{enumerate}
\item
Return $f_n$.
\end{enumerate}
\caption{Algorithm for a normal form of $f$ with respect to $G$}
\label{normalformmoduloG}
\bigskip \hrule\hrule\hrule
\end{figure}

\begin{lemma} \label{lemmanormalformG}
For the algorithm of Figure \ref{normalformmoduloG}, 
we have 
  \[
  LM(f_1) \succeq LM(f_2) \succeq LM(f_3) \succeq \cdots \succeq LM(f_n) \succeq \cdots,
  \]
and so $LM(f_n) = LM(f_{n+1}) = \cdots$ for some $n \ge 1$. 
Hence the algorithm terminates, and its output $f_n$ is a normal form of $f$ with respect to $G$.

Furthermore, $f_n + I = f + I$ in $F\langle X \rangle / I$; that is,
$f_n$ is congruent to $f$ modulo the ideal $I$ generated by $G$.
\end{lemma}

\begin{proof}
Exercise.
\end{proof}

A normal form of $f$ with respect to $G$ is not uniquely determined 
by the algorithm of Figure \ref{normalformmoduloG}: the output depends on the
choices made of $v$ and $w$ in step (2)(b).
In particular, it follows that the output of the algorithm does not necessarily equal $N\!F_I(f)$, 
which is uniquely determined by Corollary \ref{cornf}.

\begin{example} \label{examplenormalform}
Let $X = \{ a, b, c \}$ and let $I \subset F\langle X \rangle$ be the ideal generated by
  \[
  G = \{ \;\; a^2 - a, \;\; ba + ab, \;\; b^2 - b, \;\; ca + ac - c, \;\; cb + bc - c, \;\; c^2 - b - a \;\; \}.
  \]
(We have seen this set before in Example \ref{jordansymmetric}.) 
For convenience, we write each generator in standard form, 
and the generators are sorted in deglex order of their leading monomials.
We compute the normal form of $f_1 = c^2 b$ with respect to $G$ in two different ways,
and obtain two different answers.
We will see in Example \ref{examplegrobnerbasis}
that $N\!F_I( c^2 b ) = b$, so neither of these two calculations produces the desired result.

(1) 
Starting with $g_6 = c^2 - b - a$ we obtain 
  \[
  f_2 = f_1 - g_6 b = c^2 b - ( c^2 b - b^2 - ab ) = b^2 + ab.
  \]
Next using $g_3 = b^2 - b$ we obtain
  \[
  f_3 = f_2 - g_3 = b^2 + ab - ( b^2 - b ) = ab + b.
  \]
No further reductions are possible; the algorithm terminates with output $ab + b$.

(2)
Starting with $g_5 = cb + bc - c$ we obtain
  \[
  f_2 = f_1 - c g_5 = c^2 b - ( c^2b + cbc - c^2 ) = - cbc + c^2.
  \]
Next using $g_5$ again we obtain
  \[
  f_3 = f_2 + g_5 c = - cbc + c^2 + ( cbc + bc^2 - c^2 ) = bc^2.
  \]
Using $g_6 = c^2 - b - a$ gives
  \[
  f_4 = f_3 - b g_6 = bc^2 - ( bc^2 - b^2 - ba ) = b^2 + ba.
  \]
Using $g_3 = b^2 - b$ gives
  \[
  f_5 = f_4 - g_3 = b^2 + ba - ( b^2 - b ) = ba + b.
  \]
Finally, using $g_2 = ba + ab$ we obtain
  \[
  f_6 = f_5 - g_2 = ba + b - ( ba + ab ) = - ab + b.
  \]
No further reductions are possible; the algorithm terminates with output $- ab + b$.
\end{example}


\section{Gr\"obner Bases for Ideals in $F\langle X \rangle$}

If the set $G$ of generators of the ideal $I$ has a certain special property, 
stated in the next definition,
then the output of the algorithm of Figure \ref{normalformmoduloG}
is uniquely determined, and equals the normal form of $f$ modulo $I$. 

\begin{definition} \label{definitiongrobner}
Let $X$ be a finite set and let $G$ be a set of generators for the ideal $I$ 
in the free associative algebra $F\langle X \rangle$.
We say that $G$ is a \textbf{Gr\"obner basis} for $I$ if the following condition holds:
  \begin{itemize}
  \item
For every nonzero element $f \in I$ there is a generator $g \in G$ such that $LM(g)$ is a subword of $LM(f)$.
  \end{itemize}
In other words, the leading monomial of every nonzero element of the ideal 
contains a subword equal to the leading monomial of some generator of the ideal.
\end{definition}

\begin{remark}
A Gr\"obner basis is not a basis in the sense of linear algebra: 
it is not a basis for $I$ as a vector space over $F$, but rather a set of generators for $I$.
In this context, \emph{basis} means \emph{set of generators}.
Unfortunately, this misleading terminology is so well-established that we have no
choice but to accept it.
\end{remark}

The next theorem shows why Gr\"obner bases are so important.
Recall that the set $N(I)$ of all normal words
modulo $I$ is the complement of $LM(I)$ in $X^\ast$: the set of all words
which are \emph{not} leading monomials of elements of $I$.
If we have a Gr\"obner basis for $I$, then we can easily compute $N(I)$ using part (a) of the next theorem, 
and we can easily compute $N\!F_I(f)$ for all $f \in F\langle X \rangle$ using part (b):

\begin{theorem} \label{theoremgrobner}
If $G$ is a Gr\"obner basis for the ideal $I \subseteq F\langle X \rangle$ then:
  \begin{enumerate}
  \item[(a)]
  $N(I) = \{ \, w \in X^\ast \mid \text{for all $g \in G$, $LM(g)$ is not a subword of $w$} \, \}$.
  \item[(b)]
  For all $f \in F\langle X \rangle$ we have $N\!F_I(f)= N\!F(f,G)$:
  the normal form of $f$ modulo $I$ equals the normal form of $f$ with respect to $G$. 
  \end{enumerate}
\end{theorem}

\begin{proof}
Part (a) follows immediately from Definitions \ref{normalwordsmoduloI} and \ref{definitiongrobner}.
For part (b), consider $f \in F\langle X \rangle$ and let $h = N\!F(f,G)$
be the normal form of $f$ with respect to $G$ computed by the algorithm of Figure \ref{normalformmoduloG}.
For any $w \in \mathrm{support}(h)$, since $h \in I$ and $G$ is a Gr\"obner basis for $I$,
we know by Definition \ref{definitionnormalformG}
that for all $g \in G$, $LM(g)$ is not a subword of $w$.
Part (a) of the theorem now shows that $w \in N(I)$;
since this holds for all $w \in \mathrm{support}(h)$, we have $h \in C(I)$.
By the last statement of Lemma \ref{lemmanormalformG} we know that $f - h \in I$.
Clearly $f = (f-h) + h \in I \oplus C(I)$, and hence the uniqueness of the decomposition in 
Corollary \ref{cornf} implies that $h = N\!F_I(f)$.
\end{proof}

Theorem \ref{theoremgrobner} is a beautiful result, but we still have the following problem:
  \begin{itemize}
  \item
Find an algorithm for which the input is a set $G$ of generators for the ideal $I \subseteq F\langle X \rangle$, 
and for which the output is a Gr\"obner basis of $I$.
  \end{itemize}
This requires defining overlaps and compositions 
for two generators $g_1, g_2 \in G$ (Definition \ref{definitioncomposition}),
and proving the Composition (Diamond) Lemma (Lemma \ref{diamondlemma}).

\begin{definition} \label{definitionselfreduced}
Let $X$ be a finite set and let $G$ be a finite subset of $F\langle X \rangle$.
We say that $G$ is \textbf{self-reduced} if the following two conditions hold:
  \begin{enumerate}
  \item[(1)]
  Every $g \in G$ is in normal form with respect to $G \setminus \{g\}$. 
  \item[(2)]
  Every $g \in G$ is in standard form; in particular, $lc(g) = 1$.
  \end{enumerate}
\end{definition}

\begin{remark} \label{remarkselfreduced}
Condition (1) in Definition \ref{definitionselfreduced} is stronger than the condition given by de Graaf
\cite[Definition 6.1.5]{deGraaf2000}, which requires only that 
for all $g \in G$ and for all $h \in G \setminus \{g\}$, $LM(h)$ is not a subword of $LM(g)$.
The definition of de Graaf is analogous to the row-echelon form of a matrix, whereas 
our definition is analogous to the reduced row-echelon form
(and is therefore somewhat more canonical).
\end{remark}

\begin{exercise}
Referring to Remark \ref{remarkselfreduced}, explain the analogy between row-echelon forms of matrices and 
self-reduced sets of noncommutative polynomials in $F\langle X \rangle$.  
(Consider finite sets of homogeneous polynomials of degree 1.)
\end{exercise}

By calling the algorithm of Figure \ref{normalformmoduloG} repeatedly, we can create an algorithm 
for which the input is a finite subset $G \subset F\langle X \rangle$ generating an ideal $I \subseteq F\langle X \rangle$ 
and the output is a self-reduced set which generates the same ideal.
A naive approach would compute the set $\{ \, N\!F( \, g, \, G\setminus\{g\} \, ) \mid g \in G \, \}$.
However, this set may not generate the same ideal, and it may not be self-reduced;
so we have to be careful.

\begin{example}
Let $X = \{ a, b, c \}$ with $a \prec b \prec c$, and let $G = \{ c-a, c-b \}$.
Then $G$ is not self-reduced; computing the normal form of each element with respect to the other
gives $c-a-(c-b) = b-a$ and $c-b-(c-a) = -b+a$ (with standard form $b-a$).
Clearly the set $\{ b-a \}$ does not generate the same ideal as $G$.
\end{example}

\begin{example}
Let $X = \{ a, b, c, d \}$ with $a \prec b \prec c \prec d$, and consider the set 
  \[
  G = \{ \, d-a, \, d-b, \, d-c \, \},
  \]
which is not self-reduced.
One way to compute the normal form of each element with respect to the others
is as follows, replacing each result by its standard form:
  \[
  d-a - (d-b) = b-a,
  \quad
  d-b - (d-c) = c-b,
  \quad
  d-c - (d-a) = a-c \rightarrow c-a.
  \] 
Clearly the set $\{ \, b-a, \, c-b, \, c-a \, \}$ is not self-reduced.
\end{example}

\begin{exercise}
Using the algorithm of Figure \ref{normalformmoduloG}, 
compose an algorithm whose input is a finite subset $G \subset F\langle X \rangle$
generating an ideal $I \subseteq F\langle X \rangle$ and whose output is a self-reduced set generating the same ideal.
Hint: Avoid the problems illustrated by the last two examples by sorting $G$ using deglex order of leading monomials.
\end{exercise}

\begin{definition} \label{definitioncomposition}
Consider two nonzero elements $g_1, g_2 \in F\langle X \rangle$ in standard form;
we allow $g_1 = g_2$.
Set $w_1 = LM(g_1)$ and $w_2 = LM(g_2)$.
Assume that 
  \begin{enumerate}
  \item[(1)]
  $w_1$ is not a proper subword of $w_2$, and $w_2$ is not a proper subword of $w_1$
  (we say ``proper'' because we allow $g_1 = g_2$).
  \end{enumerate}
(Condition (1) is satisfied if $g_1, g_2$ belong to a self-reduced set.)
Assume also that 
  \begin{enumerate}
  \item[(2)]
  for some words $u_1, u_2, v \in X^\ast$ with $v \ne 1$ we have $w_1 = u_1 v$ and $w_2 = v u_2$
  (condition (1) implies that $u_1 \ne 1$ and $u_2 \ne 1$).
  \end{enumerate}
In this case, we call $v$ an \textbf{overlap} between $w_1$ and $w_2$,
and we have $w_1 u_2 = u_1 w_2$, where
$u_1$ is a proper right subword of $w_1$, and $u_2$ is a proper left subword of $w_2$:
  \[
  w_1 u_2 = u_1 v u_2 = u_1 w_2.
  \]
The element $g_1 u_2 - u_1 g_2$
is called a \textbf{composition} of $g_1$ and $g_2$;
the common term, a scalar multiple of $u_1 v u_2$, cancels, since both $g_1$ and $g_2$ are monic.
(In the theory of commutative Gr\"obner bases, compositions are often called $S$-polynomials.)
\end{definition}

\begin{example}
Consider the following two words in $X^\ast$ where $X = \{ a, b, c \}$:
  \[
  w_1 = a^2bcba,
  \qquad
  w_2 = bacba^2.
  \]
These words have the following overlaps:
  \begin{itemize}
  \item
  $w_1$ has a self-overlap: $w_1 = u_1 v = v u_2$ for $u_1 = a^2bcb$, $v = a$, $u_2 = abcba$.
  \item
  $w_1$ and $w_2$ overlap: $w_1 = u_1 v$, $w_2 = v u_2$ for $u_1 = a^2bc$, $v = ba$, $u_2 = cba^2$.
  \item
  $w_2$ and $w_1$ have overlaps of length 1 and length 2:
    \begin{itemize}
    \item[$\centerdot$]
    $w_2 = u_2 v$, $w_1 = v u_1$ for $u_2 = bacba$, $v = a$, $u_1 = abcba$.
    \item[$\centerdot$]
    $w_2 = u_2 v$, $w_1 = v u_1$ for $u_2 = bacb$, $v = a^2$, $u_1 = bcba$.    
    \end{itemize}
  \end{itemize}
\end{example}

\begin{example} \label{examplecomposition}
Consider the last two generators from Example \ref{jordansymmetric}:
  \[
  g_5 = cb + bc - c, \qquad g_6 = c^2 - b - a.
  \]
There is a composition of $g_6$ and $g_5$ corresponding to
  \[
  w_6 = c^2, \quad w_5 = cb, \quad u_6 = c, \quad u_5 = b, \quad v = c.
  \]
We obtain
  \begin{align*}
  g_6 u_5 - u_6 g_5
  &=
  ( c^2 - b - a ) b - c ( cb + bc - c )
  =
  c^2 b - b^2 - ab - c^2b - cbc + c^2
  \\
  &=
  - b^2 - ab - cbc + c^2
  \xrightarrow{\mathrm{\; sf \;}}
  cbc - c^2 + b^2 + ab,
  \end{align*}
where the arrow denotes replacing the polynomial by its standard form.  
\end{example}

\begin{remark} \label{remarkadds}
The motivation for considering compositions is as follows.
Suppose that $s = g_1 u_2 - u_1 g_2$ is a composition of $g_1$ and $g_2$,
and that the normal form of $s$ with respect to $G$ is nonzero.
Then $N\!F( s, G )$ is an element of the ideal $I$ whose 
leading monomial is not divisible by any element of $G$.
If we replace $G$ by $G \cup \{ N\!F( s, G ) \}$, then we are one step closer to having a Gr\"obner basis for $I$. 
\end{remark}


\section{The Composition (Diamond) Lemma} \label{sectioncompositionlemma}

This lemma is fundamental to the theory of Gr\"obner bases,
and leads to an algorithm for constructing a Gr\"obner basis for an ideal from a given set of generators for the ideal;
the basic idea underlying this algorithm was given in Remark \ref{remarkadds}.

The origin of the name Diamond Lemma is roughly as follows; see also \cite{Bergman1978,Newman1942}.
We have an element $f \in F\langle X \rangle$, and we want to compute its normal form with respect to
a finite subset $G \subset F\langle X \rangle$.
At every step in the computation, there may be many different choices of reduction:
many leading monomials of elements of $G$ may occur as subwords of many monomials in $f$.
We want to be sure that whatever sequence of reductions we perform, the final result will be the same.
This condition is called the ``resolution of ambiguities'', 
and is illustrated by this ``diamond'':
  \[
  \begin{array}{ccccc}
  & & g_0 = f = h_0 & & \\
  & \swarrow & & \searrow \\
  g_i & & & & h_j \\
  & \searrow & & \swarrow \\
  & & g_m = h_n & &
  \end{array}
  \]

\begin{definition}
Let $G = \{ g_1, \dots, g_n \}$ be a set of generators for the ideal $I \subseteq F\langle X \rangle$.
For any word $w \in X^\ast$ we define $I(G,w)$ to be the subspace of $I$
spanned by the elements of the form $ugv$ where $g \in G$, $u, v \in X^\ast$, and
$LM(ugv) \prec w$:
  \[
  I(G,w) 
  = 
  \Big\{ \, 
  \sum_{i=1}^n \alpha_i u_i g_i v _i 
  \,\Big|\,
  \alpha_i \in F; \,
  u_i, v_i \in X^\ast; \,
  LM(u_ig_iv_i) \prec w 
  \, \Big\}.
  \]
Thus $I(G,w)$ is the subspace of $I$, relative to the set $G$ of generators, 
consisting of the elements all of whose monomials precede $w$ in the total order on $X^\ast$.
\end{definition}

\begin{lemma} \label{diamondlemma}
\emph{\textbf{Composition (Diamond) Lemma}}.
Let $G$ be a monic self-reduced set generating the ideal $I \subseteq F\langle X \rangle$.
Then these conditions are equivalent:
  \begin{enumerate}
  \item[(1)]
  $G$ is a Gr\"obner basis for $I$.
  \item[(2)]
  For every pair of generators $g, h \in G$, if $LM(g) u = v LM(h)$ for some $u, v \in X^\ast$,
  then $gu - vh \in I(G,t)$ where $t = LM(g) u = v LM(h)$.
  \end{enumerate}
\end{lemma}

\begin{remark}
Condition (2) implies that every composition $gu - vh$ of the elements of $G$ is a linear combination of
elements of the form $u_i g_i v_i$ where $g_i \in G$ and $u_i, v_i \in X^\ast$ with  
$u_i LM(g_i) v_i \prec LM(g) u = v LM(h)$.
The crucial point here is that we are only allowed to use elements of the form $u_i g_i v_i$.
\end{remark}

\begin{proof} (of Lemma \ref{diamondlemma}) \,
We follow closely the proof by de Graaf \cite[Theorem 6.1.6]{deGraaf2000}.

$(1) {\implies} (2)$:
Assume that $G$ is a Gr\"obner basis.
For $g, h \in G$ let $f = gu - vh$ where $ LM(g) u = v LM(h)$ for some $u, v \in X^\ast$.
Clearly $f \in I$ and so $N\!F_I(f) = 0$.
For $t = LM(g) u = v LM(h)$ we have $LM(f) \prec t$ since the leading terms of 
$LM(g) u$ and $v LM(h)$ cancel.
When we apply the algorithm of Figure \ref{normalformmoduloG}
to compute $N\!F(f,G)$, we repeatedly subtract terms of the form 
  \[
  \alpha u_1 k u_2
  \quad 
  ( \alpha \in F; \, k \in G; \, u_1, u_2 \in X^\ast; \, LM( u_1 k u_2 ) \prec t ).
  \]
Clearly all these terms belong to $I$ and hence to $I(G,t)$.
Since $G$ is a Gr\"obner basis, we have $N\!F(f,G) = N\!F_I(f) = 0$.
It follows that $f$ is a sum of terms in $I(G,t)$, and hence $f \in I(G,t)$.

$(2) {\implies} (1)$:
We assume condition (2) and prove that $G$ is a Gr\"obner basis for $I$.
Let $f \in I$ be arbitrary; we have
  \begin{equation} 
  \label{sumf}
  f = \sum_{i=1}^n \alpha_i u_i g_i v_i \quad ( \, \alpha_i \in F; \, u_i, v_i \in X^\ast; \, g_i \in G \, ).
  \end{equation}
We need to show that $LM(g)$ is a subword of $LM(f)$ for some $g \in G$.
We write 
  \[
  s_i = LM( u_i g_i v_i ).
  \] 
Renumbering the generators in $G$ if necessary, we may assume that
  \begin{equation}
  \label{ell}
  s_1 = \cdots = s_\ell \succ s_{\ell+1} \succeq \cdots \succeq s_n.
  \end{equation}
Thus $\ell$ is the number of equal highest monomials in deglex order; 
the remaining monomials strictly precede these highest monomials; 
and we sort the remaining monomials in weak reverse deglex order.

If $\ell = 1$ then $s_1 \succ s_2$ and so $LM(f) = u_1 s_1 v_1 = u_1 LM(g_1) v_1$ as required.

We now assume $\ell \ge 2$.
In this case we can rewrite equation \eqref{sumf} as follows:
  \begin{equation}
  \label{rewritef}
  f 
  = 
  \alpha_1 ( u_1 g_1 v_1 - u_2 g_2 v_2 ) 
  + 
  ( \alpha_1 + \alpha_2 ) u_2 g_2 v_2
  +
  \sum_{i=3}^n \alpha_i u_i g_i v_i.
  \end{equation}
Since $\ell \ge 2$, we have 
  \begin{equation} 
  \label{equalmonomials}
  u_1 LM(g_1) v_1 = u_2 LM(g_2) v_2.
  \end{equation}
If $u_1 = u_2$ then $LM(g_1) v_1 = LM(g_2) v_2$. 
Hence either $LM(g_1)$ is a left subword of $LM(g_2)$, or $LM(g_2)$ is a left subword of $LM(g_1)$.
But this contradicts the assumption that $G$ is self-reduced.
Hence $u_1 \ne u_2$, and so either $u_1$ is a proper left subword of $u_2$,
or $u_2$ is a proper left subword of $u_1$.

Assume that $u_1$ is a proper left subword of $u_2$;
a similar argument applies when $u_2$ is a proper left subword of $u_1$.
We have $u_2 = u_1 u'_2$ where $u'_2 \ne 1$.
Then 
  \[
  u_1 LM(g_1) v_1 = u_1 u'_2 LM(g_2) v_2
  \quad
  \text{and so}
  \quad
  LM(g_1) v_1 = u'_2 LM(g_2) v_2.
  \]
If $v_1$ is a right subword of $v_2$ then 
$LM(g_2)$ is a subword of $LM(g_1)$, again contradicting the assumption that $G$ is self-reduced.
Hence $v_2$ is a right subword of $v_1$, giving $v_1 = v'_1 v_2$ where $v'_1 \ne 1$.
Then 
  \[
  LM(g_1) v'_1 v_2 = u'_2 LM(g_2) v_2
  \quad
  \text{and so}
  \quad
  LM(g_1) v'_1 = u'_2 LM(g_2).
  \]
By the assumption that condition (2) holds, it follows that
  \[
  g_1 v'_1 - u'_2 g_2 \in I(G,s)
  \;
  \text{where $s = LM(g_1) v'_1 = u'_2 LM(g_2)$}.
  \]
Therefore
  \[
  u_1 ( g_1 v'_1 - u'_2 g_2 ) v_2
  =
  u_1 g_1 v'_1 v_2 - u_1 u'_2 g_2 v_2
  =
  u_1 g_1 v_1 - u_2 g_2 v_2.  
  \]
But $u_1 LM(g_1) v_1 = u_2 LM(g_2) v_2$ (since $\ell \ge 2$)
and so cancellation gives
  \[
  u_1 g_1 v_1 - u_2 g_2 v_2 \in I(G,t),
  \;
  \text{where $t = u_1 LM(g_1) v_1$}.
  \]
It follows that we can rewrite equation \eqref{sumf} to obtain an expression of the same form
where either 
  \begin{enumerate}
  \item[i)]
  the new value of $LM( u_1 g_1 v_1 )$ is lower in deglex order
  (this happens when $\ell = 2$ and $\alpha_1 + \alpha_2 = 0$),
  or
  \item[ii)]
  the number $\ell$, defined by the order relations \eqref{ell}, has decreased. 
  \end{enumerate}
Since the total order on $X^\ast$ satisfies the descending chain condition, 
after a finite number of steps we obtain an expression for $f$ of the form \eqref{sumf} where $\ell = 1$,
and then again
$LM(f) = u_1 s_1 v_1 = u_1 LM(g_1) v_1$ as required.
\end{proof}

\begin{lemma} \label{lemmaformain}
Consider two elements $g, h \in G$ in standard form, and let $s \in X^\ast$ be an arbitrary monomial.
Set $u = s LM(h)$, $v = LM(g) s$ and $t = LM(g) s LM(h)$, so that $LM(g) u = v LM(h) = t$.
Then we have $gu -vh \in I(G,t)$.
\end{lemma}

\begin{proof}
Separate the leading monomials of $g$ and $h$:
  \[
  g = LM(g) + g_0,
  \qquad
  h = LM(h) + h_0,
  \]
where either $g_0 = 0$ or $LM(g_0) \prec LM(g)$,
and either $h_0 = 0$ or $LM(h_0) \prec LM(h)$.
We calculate as follows:
  \begin{align*}
  gu - vh
  &=
  \big( LM(g) + g_0 \big) s LM(h) - LM(g) s \big( LM(h) + h_0 \big)
  \\
  &=
  g_0 s LM(h) - LM(g) s h_0
  \\
  &=
  g_0 s ( h - h_0 ) - ( g - g_0 ) s h_0 
  \\
  &=
  g_0 s h - g s h_0.   
  \end{align*}
Then clearly $gu - vh = ( g_0 s ) h - g ( s h_0 ) \in I(G,t)$ where $t = LM(g) s LM(h)$.
\end{proof}

\begin{theorem} \emph{\textbf{Main Theorem.}} \label{maintheorem}
Suppose that $G$ is a monic self-reduced set of generators for the ideal $I \subseteq F\langle X \rangle$.
Then these two conditions are equivalent:
  \begin{enumerate}
  \item[(1)]
  $G$ is a Gr\"obner basis for $I$.
  \item[(2)]
  For every composition $f$ of the generators in $G$, the normal form of $f$ with respect to $G$ is zero:
  $N\!F(f,G) = 0$.
  \end{enumerate}
\end{theorem}

\begin{proof}
$(1) {\implies} (2)$:
Let $G$ be a Gr\"obner basis for $I$, and let $f = g_1 u_2 - u_1 g_2$ be a composition of $g_1, g_2 \in G$
where $u_1, u_2 \in X^\ast$.
Clearly $f \in I$, and hence by the definition of Gr\"obner basis, for some $g \in G$ 
the leading monomial $LM(g)$ is a subword of $LM(f)$; say $LM(f) = v_1 LM(g) v_2$.
If we define
  \[
  f_1 = f - \alpha v_1 g v_2
  \; \text{where} \; 
  \alpha = lc(f),
  \]
where the subtracted element belongs to $I$,
then either $f_1 = 0$ or $LM(f_1) \prec LM(f)$.
Repeating this argument, and using the DCC on $X^\ast$, we obtain $N\!F(f,G) = 0$
after a finite number of steps.

$(2) {\implies} (1)$:
Suppose that $f = g_1 u_2 - u_1 g_2$ is a composition of $g_1, g_2 \in G$ where $u_1, u_2 \in X^\ast$, 
and set $t = LM(g_1) u_2 = u_1 LM(g_2)$.
Assume that $N\!F(f,G) = 0$.
Definition \ref{definitioncomposition} implies that $u_2 \ne LM(g_2)$ and $u_1 \ne LM(g_1)$.

If $u_2$ is longer than $LM(g_2)$ then also $u_1$ is longer than $LM(g_1)$, 
and hence by Lemma \ref{lemmaformain} we have $f \in I(G,t)$.

If $u_2$ is shorter than $LM(g_2)$ then $u_1$ is shorter than $LM(g_1)$.
Since $N\!F(f,G) = 0$ by assumption, 
the algorithm of Figure \ref{normalformmoduloG} outputs zero after a finite number of steps.
But during each iteration of the loop in step (2) of that algorithm, we set
  \[
  f_{n+1} \leftarrow f_n - \alpha u_1 g u_2,
  \]
where $LM( u_1 g u_2 ) = LM( f_n ) \preceq LM( f ) \prec t$.
Thus $f$ is a linear combination of terms $u_1 g u_2$
which strictly precede $t$ in deglex order, showing that $f \in I(G,t)$.

In both cases we have $f \in I(G,t)$, and now Lemma \ref{diamondlemma}
completes the proof.
\end{proof}

\begin{remark}
Theorem \ref{maintheorem} suggests the Gr\"obner basis algorithm in Figure \ref{grobnerbasisalgorithm}
for which the input is a set $G$ generating the ideal $I \subseteq F\langle X \rangle$
and for which the output (assuming that the algorithm terminates) is a Gr\"obner basis for $I$.
\end{remark}

\begin{figure}[ht]
\hrule\hrule\hrule \bigskip
\begin{enumerate}
\item[]
\textbf{GrobnerBasis$(G)$}
\smallskip
\item[]
\textbf{Input}:
A finite subset $G \subset F\langle X \rangle$ generating an ideal $I \subseteq F\langle X \rangle$.
\item[]
\textbf{Output}:
If step (2) terminates, the output is a Gr\"obner basis of $I$.
\smallskip
\item
Set $\texttt{newcompositions} \leftarrow \texttt{true}$.
\item
While $\texttt{newcompositions}$ do:
  \begin{enumerate}
  \item
  Convert the elements of $G$ to standard form.
  \item
  Sort $G$ by deglex order of leading monomials: $G = \{ g_1, \dots, g_n \}$.
  \item
  Convert $G$ to a self-reduced set:
    \begin{itemize}
    \item
    Set $\texttt{selfreduced} \leftarrow \texttt{false}$.
    \item
    While not $\texttt{selfreduced}$ do:
      \begin{enumerate}
      \item
      Set $\texttt{selfreduced} \leftarrow \texttt{true}$.
      \item
      Set $H \leftarrow \{ \, \}$ (empty set).
      \item
      For $i = 1, \dots, n$ do:
        \begin{itemize}
        \item
        Set $H \leftarrow H \cup \{ \, N\!F( \, g_i, \, \{ g_1, \dots, g_{i-1} \} \, ) \, \}$.
        \end{itemize}
      \item
      Convert the elements of $H$ to standard form.
      \item
      Sort $H$ by deglex order of leading monomials.
      \item
      If $G \ne H$ then set $\texttt{selfreduced} \leftarrow \texttt{false}$.
      \item
      Set $G \leftarrow H$.
      \end{enumerate}
    \end{itemize} 
  \item
  Set $\texttt{compositions} \leftarrow \{ \, \}$ (empty set).
  \item
  Set $\texttt{newcompositions} \leftarrow \texttt{false}$.
  \item 
  For $g \in G$ do for $h \in G$ do:
    \begin{itemize}
    \item
    If $LM(g)$ and $LM(h)$ have an overlap $w$ then:
      \begin{enumerate}
      \item
      Define $u, v$ by $LM(g) = vw$ and $LM(h) = wu$.
      \item
      Set $s \leftarrow gu - vh$ (the composition of $g$ and $h$).
      \item
      Replace $s$ by its standard form.
      \item
      Set $t \leftarrow N\!F( s, G )$.
      \item
      Replace $t$ by its standard form.
      \item
      If $t \ne 0$ and $t \notin \texttt{compositions}$ then 
        \begin{itemize}
        \item[$\ast$]
        Set $\texttt{newcompositions} \leftarrow \texttt{true}$.
        \item[$\ast$]
        Set $\texttt{compositions} \leftarrow \texttt{compositions} \cup \{ t \}$.
        \end{itemize}
      \end{enumerate}
    \end{itemize}
  \end{enumerate}
\item
Return $G$.
\end{enumerate}
\caption{Computing a Gr\"obner basis of the ideal $I$ generated by $G$}
\label{grobnerbasisalgorithm}
\bigskip \hrule\hrule\hrule
\end{figure}

\begin{exercise}
(a)
Write a complete formal proof by induction (with basis and inductive hypothesis)
of the statement
``repeating this argument, and using the DCC on $X^\ast$, we obtain $N\!F(f,G) = 0$
after a finite number of steps''
from part $(1) {\implies} (2)$ of the proof of Theorem \ref{maintheorem}.

(b)
Write a complete formal proof by induction (with basis and inductive hypothesis)
of the statement 
``$f$ is a linear combination of terms $u_1 g u_2$ which strictly precede $t$ in deglex order''
from part $(2) {\implies} (1)$ of the proof of Theorem \ref{maintheorem}.
\end{exercise}

\begin{remark}
A different approach to the Composition (Diamond) Lemma,
emphasizing Shirshov's point of view which was developed by the Novosibirsk school of algebra,
can be found in the works of Bokut and his co-authors.
See in particular, Bokut \cite{Bokut1976}, Bokut and Kukin \cite[Chapter 1]{BokutKukin1994}, 
Bokut and Shum \cite{BokutShum2005}, Bokut and Chen \cite{BokutChen2007}.
See also Mikhalev and Zolotykh \cite{Mikhalev1998}.
\end{remark}

\begin{example} \label{examplegrobnerbasis}
We compute a Gr\"obner basis for the ideal appearing in the construction 
of the universal associative envelope of the Jordan algebra $S_2(F)$ of symmetric $2 \times 2$ matrices.
Let $X = \{ a, b, c \}$ and let $I$ be the ideal in $F\langle X \rangle$ generated by
the self-reduced set $G$ from Example \ref{jordansymmetric}:
  \begin{equation}
  \label{setG}
  \Big\{ \quad
  \begin{array}{lll}
  g_1 = a^2 - a, &\quad g_2 = ba + ab, &\quad g_3 = b^2 - b, 
  \\
  g_4 = ca + ac - c, &\quad g_5 = cb + bc - c, &\quad g_6 = c^2 - b - a.
  \end{array}
  \end{equation}
The first iteration of the algorithm produces 10 compositions
(including 3 self-compositions);
after putting them in standard form, 
denoted $p \xrightarrow{\mathrm{\; sf \;}} q$, we obtain
  \begin{alignat*}{2}
  g_1 a - a g_1 \; &\xrightarrow{\mathrm{\; sf \;}} \; 0,
  &\quad
  g_2 a - b g_1 \; &\xrightarrow{\mathrm{\; sf \;}} \; s_1 = aba + ba,
  \\
  g_3 a - b g_2 \; &\xrightarrow{\mathrm{\; sf \;}} \; s_2 = bab + ba,
  &\quad
  g_3 b - b g_3 \; &\xrightarrow{\mathrm{\; sf \;}} \; 0,
  \\
  g_4 a - c g_1 \; &\xrightarrow{\mathrm{\; sf \;}} \; s_3 = aca,
  &\quad
  g_5 a - c g_2 \; &\xrightarrow{\mathrm{\; sf \;}} \; s_4 = cab - bca + ca,
  \\
  g_5 b - c g_3 \; &\xrightarrow{\mathrm{\; sf \;}} \; s_5 = bcb,
  &\quad
  g_6 a - c g_4 \; &\xrightarrow{\mathrm{\; sf \;}} \; s_6 = cac - c^2 + ba + a^2,
  \\
  g_6 b - c g_5 \; &\xrightarrow{\mathrm{\; sf \;}} \; s_7 = cbc - c^2 + b^2 + ab,
  &\quad
  g_6 c - c g_6 \; &\xrightarrow{\mathrm{\; sf \;}} \; s_8 = cb + ca - bc - ac.
  \end{alignat*}
Computing normal forms of these compositions with respect to the set $G$
using the algorithm of Figure \ref{normalformmoduloG}, 
we obtain only two distinct nonzero results:
  \begin{align*}
  &
  s_1 - a g_2 + g_1 b - g_2 = -2 ab \; \xrightarrow{\mathrm{\; sf \;}} \; ab,
  \\ 
  &
  s_2 - g_2 b + a g_3 - g_2 = -2 ab \; \xrightarrow{\mathrm{\; sf \;}} \; ab,
  \\ 
  &
  s_3 - a g_4 + g_1 c = 0 \; \xrightarrow{\mathrm{\; sf \;}} \; 0,
  \\ 
  &
  s_4 - g_4 b + b g_4 - g_2 c + a g_5 - g_5 - g_4 = - 2 bc - 2 ac + 2 c \; \xrightarrow{\mathrm{\; sf \;}} \; bc + ac - c,
  \\ 
  &
  s_5 - b g_5 + g_3 c = 0 \; \xrightarrow{\mathrm{\; sf \;}} \; 0,
  \\ 
  &
  s_6 - g_4 c + a g_6 - g_2 = -2 ab \; \xrightarrow{\mathrm{\; sf \;}} \; ab, 
  \\ 
  &
  s_7 - g_5 c + b g_6 + g_2 = 2 ab \; \xrightarrow{\mathrm{\; sf \;}} \; ab,
  \\ 
  &
  s_8 - g_5 - g_4 = - 2bc - 2ac + 2c \; \xrightarrow{\mathrm{\; sf \;}} \; bc + ac - c.
  \end{align*}
So we define
  \begin{equation}
  \label{2new}
  t_1 = ab, \qquad t_2 = bc + ac - c.
  \end{equation}
We include the new generators \eqref{2new} in the original set \eqref{setG}, obtaining
a new set $H$ of generators for the ideal $I$:
  \[
  \left\{ \quad
  \begin{array}{lll}
  g_1 = a^2 - a, &\quad 
  t_1 = ab, &\quad
  g_2 = ba + ab,
  \\
  g_3 = b^2 - b, &\quad
  t_2 = bc + ac - c, &\quad
  g_4 = ca + ac - c,
  \\
  g_5 = cb + bc - c, &\quad 
  g_6 = c^2 - b - a.
  \end{array}
  \right.
  \]
For each element $h \in H$, we compute its normal form with respect to the elements which precede it
in the total order on $H$ (deglex order of leading monomials).
In this simple example, all we need to do is to replace $g_2$ by $g_2 - t_1 = ba$:
  \[
  \left\{ \quad
  \begin{array}{lll}
  g_1 = a^2 - a, &\quad 
  t_1 = ab, &\quad
  g'_2 = ba,
  \\
  g_3 = b^2 - b, &\quad
  t_2 = bc + ac - c, &\quad
  g_4 = ca + ac - c,
  \\
  g_5 = cb + bc - c, &\quad 
  g_6 = c^2 - b - a.
  \end{array}
  \right.
  \]
We now verify that this set is a Gr\"obner basis: 
all compositions of these generators have normal form 0 with respect to this set.
\end{example}

\begin{remark} \label{remarknormalform}
Using the Gr\"obner basis of Example \ref{examplegrobnerbasis}, 
it is easy to compute the normal form of any element of $F\langle X \rangle$ 
using Theorem \ref{theoremgrobner}(b).
In particular, we can compute the normal form of the element $f = c^2 b$ from Example \ref{examplenormalform}:
  \[
  f_1 - g_6 b - g_3 - t_1 = c^2 b - ( c^2 - b - a ) b - ( b^2 - b ) - ab = b.
  \] 
In this way, using Lemma \ref{lemmaquotientstructure}, we can calculate
the structure constants of the universal associative envelope $U(S_2(F))$.
\end{remark}

\begin{exercise}
Let $J = S_2(F)$ be the Jordan algebra of symmetric $2 \times 2$ matrices.
A basis for $U(J) = F\langle a,b,c \rangle / I$ consists of the cosets of the monomials
which do not any leading monomial from the Gr\"obner basis
(Example \ref{examplegrobnerbasis}).
  \begin{enumerate}
  \item[(a)]
  Write down the (finite) set of basis monomials for $U(J)$.
  \item[(b)]
  Using Lemma \ref{lemmaquotientstructure}, 
  compute the structure constants for $U(J)$:
  express products of basis monomials as linear combinations of basis monomials.
  \item[(c)]
  Determine explicitly the structure of the associative algebra $U(J)$.
  (Use the algorithms in my survey paper \cite{Bremner2011} on the Wedderburn decomposition.)
  \end{enumerate}
\end{exercise}

\begin{example} \label{aba-ba}
Here is an example, from de Graaf \cite[page 226]{deGraaf2000},
of a generating set $G$ for which the algorithm of Figure \ref{grobnerbasisalgorithm} never terminates.
This also shows why we must consider self-compositions of generators. 
Let $X = \{ a, b \}$ and define 
  \[
  G_0 = \{ g_1 = aba - ba \}.
  \]
The first iteration of the algorithm produces one composition of $g_1$ with itself:
  \[
  g_1 ba - ab g_1 
  = 
  ( aba - ba ) ba - ab ( aba - ba ) 
  = 
  -baba + ab^2a 
  \xrightarrow{\mathrm{\; sf \;}} 
  baba - ab^2a.
  \]
Computing the normal form of this composition with respect to $G_0$ gives
  \[
  ( baba - ab^2a ) - b ( aba - ba ) = -ab^2a + b^2a
  \xrightarrow{\mathrm{\; sf \;}} 
  ab^2a - b^2a.
  \]
Including this with $g_1$ gives a new generating set,
which is already self-reduced:
  \[
  G_1 = \{ \, g_1 = aba - ba, \, g_2 = ab^2a - b^2a \, \}.
  \]
The second iteration produces three compositions:
  \begin{align*}
  g_1 b^2a - ab g_2
  &= 
  ( aba - ba ) b^2a - ab ( ab^2a - b^2a ) 
  = 
  -bab^2a + ab^3a 
  \\
  &
  \xrightarrow{\mathrm{\; sf \;}} 
  bab^2a - ab^3a,
  \\
  g_2 ba - ab^2 g_1 
  &= 
  ( ab^2a - b^2a ) ba - ab^2 ( aba - ba ) 
  = 
  -b^2aba + ab^3a 
  \\
  &
  \xrightarrow{\mathrm{\; sf \;}} 
  b^2aba - ab^3a,
  \\
  g_2 b^2a - ab^2 g_2
  &= 
  ( ab^2a - b^2a ) b^2a - ab^2 ( ab^2a - b^2a ) 
  = 
  -b^2ab^2a + ab^4a 
  \\
  &
  \xrightarrow{\mathrm{\; sf \;}} 
  b^2ab^2a - ab^4a.
  \end{align*}  
Computing the normal forms of these compositions with respect to $G_1$ gives  
  \begin{align*}
  ( bab^2a - ab^3a ) - b ( ab^2a - b^2a ) = -ab^3a + b^3a \xrightarrow{\mathrm{\; sf \;}} ab^3a - b^3a,
  \\
  ( b^2aba - ab^3a ) - b^2 ( aba - ba ) = -ab^3a + b^3a \xrightarrow{\mathrm{\; sf \;}} ab^3a - b^3a,
  \\
  ( b^2ab^2a - ab^4a ) - b^2 ( ab^2a - b^2a ) = -ab^4a + b^4a \xrightarrow{\mathrm{\; sf \;}} ab^4a - b^4a.  
  \end{align*}
Including these with $g_1, g_2$ gives a new generating set,
which is already self-reduced:
  \[
  G_2 = \{ \, g_1 = aba - ba, \, g_2 = ab^2a - b^2a \, g_3 = ab^3a - b^3a, \, g_4 = ab^4a - b^4a \, \}.
  \]  
It is now easy to verify that the algorithm never terminates; see Exercise \ref{exercise.aba-ba}.
\end{example}

\begin{exercise} \label{exercise.aba-ba}
(a) Work out in detail the next iteration for Example \ref{aba-ba}.

\noindent
(b) State and prove a conjecture for the elements of the set $G_n$ obtained 
at the end of the $n$-th iteration of the Gr\"obner basis algorithm in Example \ref{aba-ba}.
\end{exercise}

\begin{example} \label{CCexample}
Here is another (much more complicated) example in which self-compositions play an essential role.
We set $X = \{ a, b \}$ and consider the following two elements of $F\langle a, b \rangle$
which clearly form a self-reduced set:
  \[
  g_1 = aba - a^2b - a,
  \qquad
  g_2 = bab - ab^2 - b.
  \]
(1)
The first iteration of the Gr\"obner basis algorithm produces three compositions:
  \begin{align*}
  s_1
  =
  g_1 ba - ab g_1 
  &=
  ( aba - a^2b - a ) ba - ab ( aba - a^2b - a )
  \\
  &=
  ababa - a^2b^2a - aba - ababa + aba^2b + aba
  \\
  &=
  aba^2b - a^2b^2a,
  \\
  s_2 
  =
  g_2 a - b g_1
  &=
  ( bab - ab^2 - b ) a - b ( aba - a^2b - a )
  \\
  &=
  baba - ab^2a - ba - baba + ba^2b + ba
  \\
  &=
  ba^2b - ab^2a,
  \\
  s_3
  =
  g_2 ab - ba g_2
  &=
  ( bab - ab^2 - b ) ab - ba ( bab - ab^2 - b )
  \\
  &=
  babab - ab^2ab - bab - babab + ba^2b^2 + bab
  \\
  &=
  ba^2b^2 - ab^2ab. 
  \end{align*} 
We compute the normal form of each composition with respect to $\{ g_1, g_2 \}$:
  \begin{align*}
  s_1 - g_1 ab - a^2 g_2
  &=
  aba^2b - a^2b^2a - ( aba - a^2b - a ) ab - a^2 ( bab - ab^2 - b )
  \\
  &=
  aba^2b - a^2b^2a - aba^2b + a^2bab + a^2b - a^2bab + a^3b^2 + a^2b
  \\
  &=
  - a^2b^2a + a^3b^2 + 2 a^2b
  \xrightarrow{\mathrm{\; sf \;}}
  a^2b^2a - a^3b^2 - 2 a^2b
  =
  h_1,
  \\
  s_2 
  &=
  h_2,
  \\
  s_3 + ab g_2 + g_1 b^2
  &=
  ba^2b^2 - ab^2ab + ab ( bab - ab^2 - b ) + ( aba - a^2b - a ) b^2
  \\
  &=
  ba^2b^2 - ab^2ab + ab^2ab - abab^2 - ab^2 + abab^2 - a^2b^3 - ab^2
  \\
  &=
  ba^2b^2 - a^2b^3 - 2ab^2
  =
  h_3.
  \end{align*}
We combine these compositions with the original generators and sort them:
  \begin{align*}
  &
  g_1 = aba - a^2b - a,
  \qquad
  g_2 = bab - ab^2 - b,
  \qquad
  h_2 = ba^2b - ab^2a,
  \\
  &
  h_1 = a^2b^2a - a^3b^2 - 2 a^2b,
  \qquad
  h_3 = ba^2b^2 - a^2b^3 - 2 ab^2.
  \end{align*}
Self-reducing this set eliminates $h_3$ since $h_3 - h_2 b - ab g_2 - g_1 b^2 = 0$.

\noindent
(2)
The second iteration produces five compositions with these normal forms:
  \begin{align*}
  &
  h_4 = ba^3b - ab^2a^2 + ba^2,
  \quad 
  h_5 = ba^3b^2 - a^2b^3a,
  \quad 
  h_6 = a^3b^3a - a^4b^3 -3 a^3b^2,
  \\
  & 
  h_7 = ba^4b^2 - ab^2a^3b +2 ba^3b,
  \quad 
  h_8 = a^4b^4a - a^5b^4 +2 a^3b^3a -6 a^4b^3 -6 a^3b^2.
  \end{align*} 
Combining these compositions with $g_1, g_2, h_2, h_1$ and self-reducing the resulting set
eliminates $h_5$ and replaces $h_7$ and $h_8$ with these elements:
  \[
  h'_7 
  =
  ba^4b^2 - a^2b^3a^2 +2 ab^2a^2 -2 ba^2,
  \qquad
  h'_8
  = 
  a^4b^4a - a^5b^4 -4 a^4b^3.
  \] 
(3)
The third iteration of the algorithm produces 18 compositions:
  \begin{align*}
  &
  ba^4b - ab^2a^3 +2 ba^3,
  \\
  &
  ba^5b^2 - a^2b^3a^3 +2 ba^4b +2 ab^2a^3 -2 ba^3,
  \\
  &
  ba^5b^2 - ab^2a^4b +3 ba^4b,
  \\
  &
  ba^6b^2 - ab^2a^5b +4 ba^5b,
  \\
  &
  ba^5b^3 - a^3b^4a^2 +3 a^2b^3a^2 -6 ab^2a^2 +6 ba^2,
  \\
  &
  ba^6b^3 - a^2b^3a^4b +4 ba^5b^2 +2 ab^2a^4b,
  \\
  & 
  ba^6b^3 - ab^2a^5b^2 +4 ba^5b^2,
  \\
  &
  ba^5b^4 - a^4b^5a,
  \\
  &
  a^5b^5a - a^6b^5 -5 a^5b^4,
  \\
  &
  ba^7b^3 - a^2b^3a^5b +6 ba^6b^2 +2 ab^2a^5b +4 ba^5b,
  \\
  &
  ba^7b^3 - ab^2a^6b^2 +5 ba^6b^2,
  \\
  &
  a^5b^5a^2 - a^6b^5a -5 a^6b^4 -20 a^5b^3,
  \\
  &
  a^6b^6a - a^7b^6 +6 a^5b^5a -12 a^6b^5 -30 a^5b^4,
  \\
  &
  a^6b^6a - a^7b^6 +8 a^5b^5a -14 a^6b^5 -40 a^5b^4,
  \\
  &
  ba^8b^4 - a^2b^3a^6b^2 +8 ba^7b^3 +2 ab^2a^6b^2 +10 ba^6b^2,
  \\
  &  
  a^6b^6a^2 - a^7b^6a +4 a^5b^5a^2 -10 a^6b^5a -20 a^6b^4 -80 a^5b^3,
  \\
  &
  a^7b^7a - a^8b^7 +12 a^6b^6a -19 a^7b^6 +36 a^5b^5a -108 a^6b^5 -180 a^5b^4,
  \\
  &
  a^8b^8a - a^9b^8 +12 a^7b^7a -20 a^8b^7 +36 a^6b^6a -120 a^7b^6 +24 a^5b^5a 
  \\
  &\qquad
  -240 a^6b^5 -120 a^5b^4.
  \end{align*} 
At this point it seems clear that the algorithm will never terminate!
\end{example}

\begin{exercise} \label{CCexercise}
Referring to Example \ref{CCexample}:
  \begin{enumerate}
  \item[(a)]
Verify the statements about the second and third iterations.
  \item[(b)] 
Prove that the algorithm does not terminate.
  \item[(c)]
Determine a closed form for the generators at the end of the $n$-th iteration.
  \end{enumerate}
\end{exercise}

\begin{remark}
A rich source of examples of the behavior of the Gr\"obner basis algorithm
comes from the construction of universal associative envelopes for nonassociative triple systems 
obtained from the trilinear operations classified by the author and Peresi \cite{BremnerPeresi2007}.
A detailed study of the simplest non-trivial examples of this construction 
appears in the Ph.D. thesis of Elgendy \cite{ElgendyThesis};
see also her forthcoming paper \cite{Elgendy2013}.
Similar examples are discussed in \S\ref{trilinearexamples} of these lecture notes.
\end{remark}


\section{Application: The PBW Theorem}
\label{sectionPBW}

We now present the beautiful combinatorial proof of the Poincar\'e-Birkhoff-Witt (PBW) Theorem
discovered by Bokut \cite{Bokut1976} and independently by Bergman \cite{Bergman1978}.
We follow the exposition given by de Graaf \cite[Theorem 6.2.1]{deGraaf2000}.
The assumption that the Lie algebra is finite dimensional is not essential.

\begin{theorem} \label{pbwtheorem} \emph{\textbf{PBW Theorem.}}
If $L$ is a finite dimensional Lie algebra over a field $F$ with ordered basis $X = \{ x_1, \dots, x_n \}$,
then a basis of its universal associative envelope $U(L)$ consists of the monomials
  \[
  x_1^{e_1} \cdots x_n^{e_n}
  \quad
  ( e_1, \dots, e_n \ge 0 ).
  \]
It follows immediately that: 
  \begin{enumerate}
  \item[(i)]
$U(L)$ is infinite dimensional.
  \item[(ii)]
The natural map $L \to U(L)$ is injective.
  \item[(iii)]
$L$ is isomorphic to a subalgebra of $U(L)^-$.
  \end{enumerate}
\end{theorem} 
  
\begin{proof}
The structure constants of $L$ have the form
  \[
  [ x_i, x_j ] = \sum_{k=1}^n c_{ij}^k x_k \quad (c_{ij}^k \in F),
  \]
where $c_{ji}^k = - c_{ij}^k$ and $c_{ii}^k = 0$.
The universal associative envelope $U(L)$ is the quotient of the free associative algebra
$F\langle X \rangle$ by the ideal $I$ generated by the elements
  \[
  g_{ij}
  =
  x_i x_j - x_j x_i - [ x_i, x_j ] 
  =
  x_i x_j - x_j x_i - \sum_{k=1}^n c_{ij}^k x_k. 
  \]
By anticommutativity of the Lie bracket, we may assume that $i > j$, and hence 
$x_i x_j$ is the leading monomial of $g_{ij}$.
(If $i = j$ then $g_{ii} = 0$.)
So we set 
  \[
  G = \{ \, g_{ij} \mid 1 \le j < i \le n \, \}.
  \]
We will show that $G$ is a Gr\"obner basis for the ideal $I$.

Consider the leading monomials of two distinct generators,
  \[
  LM(g_{ij}) = x_i x_j \; (i > j), 
  \qquad
  LM( g_{\ell k} ) = x_\ell x_k \; (\ell > k).
  \]
The only possible compositions of these generators occur when either $j = \ell$ or $k = i$.
It suffices to assume $j = \ell$, so we consider $g_{ij}$ and $g_{jk}$ where $i > j > k$.
We have 
  \[
  LM( g_{ij} ) \, x_k = x_i x_j x_k = x_i \, LM( g_{jk} ),
  \]
which produces the composition
  \begin{align*}
  g_{ij} x_k - x_i g_{jk}
  &=
  \big( x_i x_j - x_j x_i - [ x_i, x_j ] \big) x_k
  -
  x_i \big( x_j x_k - x_k x_j - [ x_j, x_k ] \big)
  \\
  &=
  x_i x_j x_k - x_j x_i x_k - [ x_i, x_j ] x_k
  -
  x_i x_j x_k + x_i x_k x_j + x_i [ x_j, x_k ]
  \\
  &=
  - x_j x_i x_k - [ x_i, x_j ] x_k + x_i x_k x_j + x_i [ x_j, x_k ]
  \\
  &=
  x_i x_k x_j - x_j x_i x_k - [ x_i, x_j ] x_k + x_i [ x_j, x_k ],
  \end{align*}
which is in standard form.
(It is convenient to avoid explicit structure constants in this calculation;
recall that $[x_i,x_j]$ is a homogeneous polynomial of degree 1.)
To compute the normal form with respect to $G$, 
we first subtract $g_{ik} x_j$ and add $x_j g_{ik}$:
  \begin{align*}
  &
  x_i x_k x_j - x_j x_i x_k - [ x_i, x_j ] x_k + x_i [ x_j, x_k ]
  \\
  &\quad
  - \big( x_i x_k - x_k x_i - [ x_i, x_k ] \big) x_j
  + x_j \big( x_i x_k - x_k x_i - [ x_i, x_k ] \big)
  \\
  &=
  x_i x_k x_j - x_j x_i x_k - [ x_i, x_j ] x_k + x_i [ x_j, x_k ]
  \\
  &\quad
  - x_i x_k x_j + x_k x_i x_j + [ x_i, x_k ] x_j
  + x_j x_i x_k - x_j x_k x_i - x_j [ x_i, x_k ]
  \\
  &=
  - [ x_i, x_j ] x_k + x_i [ x_j, x_k ]
  + x_k x_i x_j + [ x_i, x_k ] x_j
  - x_j x_k x_i - x_j [ x_i, x_k ]
  \\
  &=
  - x_j x_k x_i + x_k x_i x_j  
  - [ x_i, x_j ] x_k + x_i [ x_j, x_k ] + [ x_i, x_k ] x_j - x_j [ x_i, x_k ].
  \end{align*}
We next add $g_{jk} x_i$ and subtract $x_k g_{ij}$:
  \begin{align*}
  &
  - x_j x_k x_i + x_k x_i x_j  
  - [ x_i, x_j ] x_k + x_i [ x_j, x_k ] + [ x_i, x_k ] x_j - x_j [ x_i, x_k ]
  \\
  &\quad
  + \big( x_j x_k - x_k x_j - [ x_j, x_k ] \big) x_i
  - x_k \big( x_i x_j - x_j x_i - [ x_i, x_j ] \big)
  \\
  &=
  - x_j x_k x_i + x_k x_i x_j  
  - [ x_i, x_j ] x_k + x_i [ x_j, x_k ] + [ x_i, x_k ] x_j - x_j [ x_i, x_k ]
  \\
  &\quad
  + x_j x_k x_i - x_k x_j x_i - [ x_j, x_k ] x_i
  - x_k x_i x_j + x_k x_j x_i + x_k [ x_i, x_j ]
  \\
  &= 
  - [ x_i, x_j ] x_k + x_i [ x_j, x_k ] + [ x_i, x_k ] x_j - x_j [ x_i, x_k ] - [ x_j, x_k ] x_i + x_k [ x_i, x_j ]
  \\
  &=
  x_i [ x_j, x_k ] - [ x_j, x_k ] x_i
  + x_j [ x_k, x_i ] - [ x_k, x_i ] x_j 
  + x_k [ x_i, x_j ] - [ x_i, x_j ] x_k.
  \end{align*}
We now observe that this last expression is equal to 
  \[
  [ x_i, [ x_j, x_k ] ] + [ x_j, [ x_k, x_i ] ] + [ x_k, [ x_i, x_j ] ],
  \]
which is zero by the Jacobi identity.
Thus every composition of the generators has normal form zero, and so $G$ is a Gr\"obner basis.

The leading monomials of the elements of this Gr\"obner basis have the form $x_i x_j$ where $i > j$.
A basis for $U(L)$ consists of all monomials $w$ which do not have any of these leading monomials as a 
subword. That is, if $w$ contains a subword $x_i x_j$ then $i \le j$.
It follows that the monomials in the statement of this Theorem form a basis for $U(L)$.
In particular, the monomials $x_1, \dots, x_n$ of degree 1 are linearly independent in $U(L)$,
and hence the natural map from $L$ to $U(L)$ is injective.
\end{proof}

\begin{corollary} \label{corollarylie}
Every polynomial identity satisfied by the Lie bracket in every associative algebra
is a consequence of anticommutativity and the Jacobi identity.
\end{corollary}

\begin{proof}
Suppose that $p( a_1, \dots, a_n ) \equiv 0$ is a polynomial identity which is not 
a consequence of anticommutativity and the Jacobi identity.
Then the Lie polynomial $p( a_1, \dots, a_n )$ is a nonzero element of the free Lie algebra $L$ 
generated by the variables $\{ a_1, \dots, a_n \}$.
Let $A$ be any associative algebra, and let $\epsilon\colon L \to A^-$ be any morphism of Lie algebras.
By definition of polynomial identity, we have $\epsilon(p) = 0$.
Take $A = U(L)$ and let $\epsilon$ be the injective map $L \to U(L)^-$ obtained from the PBW theorem.
Since $p \ne 0$ we have $\epsilon(p) \ne 0$, giving a contradiction.
\end{proof}

\begin{remark}
Lie algebras arose originally as tangent algebras of Lie groups.
Weakening the requirement of associativity in the definition of Lie groups gives rise to
various classes of nonassociative smooth loops, such as Moufang loops, Bol loops,
and monoassociative loops.
The corresponding tangent algebras are known respectively as Malcev algebras, Bol algebras, and BTQ algebras.
Universal nonassociative envelopes for Malcev and Bol algebras have been constructed
by P\'erez-Izquierdo and Shestakov \cite{Perez2005,PerezShestakov2004}.
This problem is still open for BTQ algebras,
but see my recent paper with Madariaga \cite{BremnerMadariaga2013}.
All of these tangent algebras are special cases of Akivis and Sabinin algebras;
for the universal nonassociative envelopes of these structures, 
see Shestakov and Umirbaev \cite{ShestakovUmirbaev2002} and P\'erez-Izquierdo \cite{Perez2007}.
\end{remark}

The PBW Theorem shows that for every Lie algebra $L$, the original set of generators
obtained from the structure constants is already a Gr\"obner basis.
The original generators in $I$ can be interpreted as rewriting rules in $U(L)$ as follows:
  \[
  x_i x_j - x_j x_i - \sum_{k=1}^n c_{ij}^k x_k \in I
  \quad \iff \quad
  x_i x_j = x_j x_i + \sum_{k=1}^n c_{ij}^k x_k \in U(L).
  \]
Repeated application of these rewriting rules allows us to work out explicit multiplication formulas
for monomials in $U(L)$.

\begin{exercise}
Let $L$ be the 2-dimensional solvable Lie algebra with basis $\{ a, b \}$ where $[a,b] = b$;
the other structure constants follow from anticommutativity.
The basis of $U(L)$ obtained from the PBW theorem consists of the monomials $a^i b^j$ for $i, j \ge 0$.
The ideal $I$ is generated by $ab - ba - b$, and so in $U(L)$ we have the relation $ba = ab - b$.
Use this and induction on the exponents to work out a formula for the product $( a^i b^j ) ( a^k b^\ell )$
as a linear combination of basis monomials.
\end{exercise}

\begin{exercise}
Let $L$ be the 3-dimensional nilpotent Lie algebra with basis $\{ a, b, c \}$ where 
$[a,b] = c$, $[a,c] = [b,c] = 0$.
The PBW basis of $U(L)$ consists of the monomials $a^i b^j c^k$ for $i, j, k \ge 0$.
In $U(L)$ we have $ba = ab - c$, $ac = ca$, $bc = cb$.
State and prove a formula for $( a^i b^j c^k ) ( a^\ell b^m c^n )$
as a linear combination of basis monomials.
\end{exercise}

\begin{exercise}
Let $L$ be the 3-dimensional simple Lie algebra $\mathfrak{sl}_2(F)$ with basis $\{ e, f, h \}$ 
where $[h,e] = 2e$, $[h,f] = -2f$, $[e,f] = h$.
The PBW basis of $U(L)$ consists of the monomials $f^i h^j e^k$ for $i, j, k \ge 0$.
In $U(L)$ we have 
  \[
  eh = he - 2e, 
  \qquad
  hf = fh - 2f, 
  \qquad 
  ef = fe + h.
  \]
State and prove a formula for $( f^i h^j e^k ) ( f^\ell h^m e^n )$
as a linear combination of basis monomials.
(This exercise is harder than the previous two.
Note that $\{h,e\}$ and $\{h,f\}$ span 2-dimensional solvable subalgebras.
See also Example \ref{Usl2example}.)
\end{exercise}


\section{Jordan Structures on $2 \times 2$ Matrices} \label{sectionFD}

In this section we study some examples of nonassociative structures whose universal associative envelopes
are finite dimensional.
The underlying vector space in all three examples is $M_2(F)$, 
the $2 \times 2$ matrices over a field $F$ of characteristic $\ne 2$.
We will use the following notation for the basis of matrix units:
  \[
  a = E_{11} = \begin{bmatrix} 1 & 0 \\ 0 & 0 \end{bmatrix},
  \;\;
  b = E_{12} = \begin{bmatrix} 0 & 1 \\ 0 & 0 \end{bmatrix},
  \;\;
  c = E_{21} = \begin{bmatrix} 0 & 0 \\ 1 & 0 \end{bmatrix},
  \;\;
  d = E_{22} = \begin{bmatrix} 0 & 0 \\ 0 & 1 \end{bmatrix}.
  \]

\subsection{The Jordan algebra of $2 \times 2$ matrices} \label{symmetric2x2}

We first make $M_2(F)$ into a Jordan algebra $J$ using the Jordan product $x \circ y = xy + yx$.
(For convenience we omit the scalar $\tfrac12$.)
The universal associative envelope $U(J)$ is isomorphic to $F\langle a,b,c,d \rangle / I$
where the ideal $I$ is generated by the following set of 10 elements, 
obtained from the structure constants of $J$; this set is already self-reduced:
  \begin{align*}
  &
  g_1 = a^2 - a, \qquad 
  g_2 = ba + ab - b, \qquad 
  g_3 = b^2, \qquad
  g_4 = ca + ac - c,
  \\
  &
  g_5 = cb + bc - d - a, \qquad
  g_6 = c^2, \qquad
  g_7 = da + ad, \qquad 
  g_8 = db + bd - b,
  \\
  & 
  g_9 = dc + cd - c, \qquad 
  g_{10} = d^2 - d.
  \end{align*}
We obtain three distinct nonzero compositions from the pairs $(g_5,g_2)$, $(g_5,g_3)$, $(g_6,g_5)$;
computing their normal forms with respect to the set of generators gives:
  \[
  s_1 = ad, \qquad 
  s_2 = bd - ab, \qquad 
  s_3 = cd - ac.
  \] 
Combining these three compositions with the original ten generators gives a new set of 13 generators;
self-reduction makes only minor changes 
  \begin{align*}
  &
  a^2 - a, \quad 
  ad, \quad 
  ba + ab - b, \quad 
  b^2, \quad 
  bd - ab, \quad 
  ca + ac - c, \quad
  cb + bc - d - a,
  \\
  &
  c^2, \quad 
  cd - ac, \quad 
  da, \quad 
  db + ab - b, \quad 
  dc + ac - c, \quad 
  d^2 - d.
  \end{align*} 
Every composition of these 13 generators has normal form zero, and so this set is a Gr\"obner basis.
There are only 9 monomials in $F\langle a,b,c,d \rangle$ which do not have the leading monomial of
one of the Gr\"obner basis elements as a subword:
  \[
  u_1 = 1, \;\;
  u_2 = a, \;\;
  u_3 = b, \;\;
  u_4 = c, \;\;
  u_5 = d, \;\;
  u_6 = ab, \;\;
  u_7 = ac, \;\;
  u_8 = bc, \;\;
  u_9 = abc.
  \]
The cosets of these monomials modulo $I$ form a basis for $U(J)$.
The multiplication table of $U(J)$ is displayed in Table \ref{UJ2table},
where $u_i$ is denoted by $i$ and dot indicates 0.

  \begin{table}
  \[
  \begin{array}{c|ccccccccc}
  &
  1 &  
  2 &  
  3 &  
  4 &  
  5 &  
  6 &  
  7 &  
  8 &  
  9  
  \\
  \midrule
  1 &
  1 &  
  2 &  
  3 &  
  4 &  
  5 &  
  6 &  
  7 &  
  8 &  
  9  
  \\
  2 &
  2 &  
  2 &  
  6 &  
  7 &  
  \cdot & 
  6 &  
  7 &  
  9 &  
  9  
  \\
  3 &
  3 &  
  3 {-} 6 &  
  \cdot & 
  8 &  
  6 &  
  \cdot & 
  8 {-} 9 &  
  \cdot & 
  \cdot 
  \\
  4 &  
  4 &  
  4 {-} 7 &  
  2 {+} 5 {-} 8 &  
  \cdot & 
  7 &  
  5 {-} 8 {+} 9 &  
  \cdot & 
  4 &  
  4 {-} 7  
  \\
  5 &
  5 &  
  \cdot & 
  3 {-} 6 &  
  4 {-} 7 &  
  5 &  
  \cdot & 
  \cdot & 
  8 {-} 9 &  
  \cdot 
  \\
  6 &
  6 &  
  \cdot & 
  \cdot & 
  9 &  
  6 &  
  \cdot & 
  \cdot & 
  \cdot & 
  \cdot 
  \\
  7 &
  7 &  
  \cdot & 
  2 {-} 9 &  
  \cdot & 
  7 &  
  \cdot & 
  \cdot & 
  7 &  
  \cdot 
  \\
  8 &
  8 &  
  9 &  
  3 &  
  \cdot & 
  8 {-} 9 &  
  6 &  
  \cdot & 
  8 &  
  9  
  \\
  9 &
  9 &  
  9 &  
  6 &  
  \cdot & 
  \cdot & 
  6 &  
  \cdot & 
  9 &  
  9 
  \\
  \midrule
  \end{array}
  \] 
  \caption{Structure constants for $U(J)$ where $J = M_2(F)^+$}
  \label{UJ2table}
  \end{table}

\begin{exercise}
(a) 
Verify the multiplication table for $U(J)$ by computing the normal form of each product 
of basis elements with respect to the Gr\"obner basis.

(b)
Use the algorithms in my survey paper \cite{Bremner2011} to compute the structure of $U(J)$.
Prove or disprove that $U(J) \approx F \oplus M_2(F) \oplus M_2(F)$.
Compare your results with the known representation theory of Jordan algebras; see Jacobson \cite{Jacobson}.
\end{exercise}

\subsection{The $2 \times 2$ matrices as a Jordan triple system}

This subsection and the next introduce the topic of multilinear operations, 
which will be discussed systematically in Section \ref{sectionmultilinear}.
We consider the vector space $T = M_2(F)$ with a trilinear operation, 
the Jordan triple product $\langle x, y, z \rangle = xyz + zyx$.
Working out the structure constants for this operation, we find that the ideal $I$ 
appearing in the definition of $U(T)$
is generated by the following self-reduced set of 40 elements:
  \begin{align*}
  &
  a^3 - a, \;\; 
  aba, \;\; 
  aca, \;\; 
  ada, \;\; 
  ba^2 + a^2b - b, \;\; 
  bab, \;\; 
  b^2a + ab^2, \;\; 
  b^3, \;\;
  bca + acb - a,
  \\
  &
  bcb - b, \;\; 
  bda + adb, \;\; 
  bdb, \;\; 
  ca^2 + a^2c - c, \;\; 
  cab + bac - d, \;\;
  cac, \;\; 
  cba + abc - a,
  \\
  &
  cb^2 + b^2c, \;\; 
  cbc - c, \;\; 
  c^2a + ac^2, \;\; 
  c^2b + bc^2, \;\; 
  c^3, \;\;
  cda + adc, \;\; 
  cdb + bdc - a, \;\; 
  cdc,
  \\
  & 
  da^2 + a^2d, \;\; 
  dab + bad, \;\; 
  dac + cad, \;\; 
  dad, \;\; 
  dba + abd - b, \;\; 
  db^2 + b^2d, 
  \\
  &
  dbc + cbd - d, \;\; 
  dbd, \;\; 
  dca + acd - c, \;\; 
  dcb + bcd - d, \;\; 
  dc^2 + c^2d, \;\; 
  dcd, \;\; 
  d^2a + ad^2,
  \\
  & 
  d^2b + bd^2 - b, \;\; 
  d^2c + cd^2 - c, \;\; 
  d^3 - d. 
  \end{align*}
These elements produce 36 distinct nonzero compositions:
  \begin{align*}
  &
  ad, \;\; 
  b^2, \;\; 
  bd - ab, \;\; 
  c^2, \;\; 
  cd - ac, \;\; 
  da, \;\; 
  db - ba, \;\; 
  dc - ca, \;\; 
  d^2 - cb - bc + a^2, \;\; 
  a^2d,
  \\
  & 
  ab^2, \;\; 
  abd - a^2b, \;\; 
  acb + abc - a, \;\; 
  ac^2, \;\; 
  acd - a^2c, \;\; 
  adb, \;\; 
  adc, \;\; 
  ad^2, \;\; 
  bad, \;\; 
  b^2c,
  \\
  & 
  b^2d, \;\; 
  bc^2, \;\; 
  bcd - bac, \;\; 
  bdc + acb - a, \;\; 
  bdc - abc, \;\; 
  bd^2 - b^2c - a^2b, \;\; 
  bd^2 - abd,
  \\
  & 
  bd^2 - a^2b, \;\; 
  cad, \;\; 
  cbd + bcd - a^2d - d, \;\; 
  cbd + bcd - d, \;\; 
  cbd + bac - d, \;\; 
  c^2d,
  \\
  & 
  cd^2 - acd, \;\; 
  cd^2 + bc^2 - a^2c, \;\; 
  cd^2 - a^2c.
  \end{align*} 
Taking the union of these two sets gives 76 generators, and self-reducing this set produces
a set with only 22 elements:
  \begin{align*}
  &
  ad, \;\; 
  b^2, \;\; 
  bd - ab, \;\; 
  c^2, \;\; 
  cd - ac, \;\; 
  da, \;\; 
  db - ba, \;\; 
  dc - ca, \;\; 
  d^2 - cb - bc + a^2,
  \\
  &
  a^3 - a, \;\; 
  aba, \;\; 
  aca, \;\; 
  acb + abc - a, \;\; 
  ba^2 + a^2b - b, \;\; 
  bab, \;\; 
  bca - abc, \;\; 
  bcb - b,
  \\
  & 
  ca^2 + a^2c - c, \;\; 
  cab + bac - d, \;\; 
  cac, \;\; 
  cba + abc - a, \;\; 
  cbc - c. 
  \end{align*}
All compositions of this new set have normal form zero, 
so we have a Gr\"obner basis.
There are only 17 monomials in $F\langle a,b,c,d \rangle$ which do not have the leading monomial
of one of these 22 generators as a subword, and the cosets of these monomials modulo $I$ form
a basis for the universal associative envelope $U(T)$:
  \[
  1, \;\;
  a, \;\;
  b, \;\;
  c, \;\;
  d, \;\;
  a^2, \;\;
  ab, \;\;
  ac, \;\;
  ba, \;\;
  bc, \;\;
  ca, \;\;
  cb, \;\;
  a^2b, \;\;
  a^2c, \;\;
  abc, \;\;
  bac, \;\;
  a^2bc.
  \] 
The multiplication table for $U(T)$ is an array of size $17 \times 17$.

\begin{exercise}
Use a computer algebra system to calculate the multiplication table for $U(T)$.
Compute the Wedderburn decomposition of $U(T)$. 
Prove or disprove that 
$U(T) \approx F \oplus M_2(F) \oplus M_2(F) \oplus M_2(F) \oplus M_2(F)$.
For the structure theory of Jordan triple systems, 
see Loos \cite{Loos1971} and Meyberg \cite{Meyberg1972}.
\end{exercise}

\subsection{The $2 \times 2$ matrices with the Jordan tetrad} \label{subsectiontetrad}

We consider the vector space $Q = M_2(F)$ with a quadrilinear operation, 
the Jordan tetrad 
  \[
  \{ w, x, y, z \} = wxyz + zyxw.
  \]
Working out the structure constants for this operation, we find that the ideal $I$ 
is generated by the self-reduced set of 136 elements displayed in Table \ref{JQS136}.
Remarkably, there are 2769 distinct nontrivial compositions of these 136 generators.
The most complicated normal form of these compositions is
  \[
  bcbcdcd + bcbc^2d^2 + dcd^2 + c^2bd + cbdc + cbcd + bdc^2 - adcd - c.
  \]
Combining the original 136 generators with the 2769 compositions produces a new 
generating set of 2905 elements.
After two iterations of self-reduction, this large set of generators collapses to 
the set 25 elements in Table \ref{JQS25} which form a Gr\"obner basis.
There are only 25 monomials in $F\langle a,b,c,d \rangle$ which do not have the leading monomial of
one of these 25 generators as a subword; the cosets of these monomials form a basis of the universal 
associative envelope $U(Q)$:
  \begin{align*}
  &
  1, \quad
  a, \quad
  b, \quad
  c, \quad
  d, \quad
  a^2, \quad
  ab, \quad
  ac, \quad
  ba, \quad
  bc, \quad
  ca, \quad
  cb, \quad
  a^3, \quad
  a^2b, \quad
  a^2c,
  \\
  &
  abc, \quad
  ba^2, \quad
  bac, \quad
  ca^2, \quad
  cab, \quad
  a^3b, \quad
  a^3c, \quad
  a^2bc, \quad
  ba^2c, \quad
  a^3bc.
  \end{align*}
The multiplication table of $U(Q)$ is an array of size $25 \times 25$.   

  \begin{table}
  \[
  \begin{array}{l}
  a^4 - a, \quad 
  aba^2 + a^2ba, \quad 
  ab^2a, \quad 
  aca^2 + a^2ca, \quad 
  acba + abca - a, \quad 
  ac^2a,
  \\
  ada^2 + a^2da, \quad
  adba + abda, \quad 
  adca + acda, \quad 
  ad^2a, \quad 
  ba^3 + a^3b - b, \quad 
  ba^2b,
  \\
  baba + abab, \quad
  baca + acab, \quad
  bada + adab, \quad 
  b^2a^2 + a^2b^2, \quad 
  b^2ab + bab^2,
  \\
  b^3a + ab^3, \quad 
  b^4, \quad 
  b^2ca + acb^2, \quad 
  b^2da + adb^2, \quad
  bca^2 + a^2cb - a,
  \\
  bcab + bacb - b, \quad 
  bcba + abcb - b, \quad 
  bcb^2 + b^2cb, \quad 
  bc^2a + ac^2b, \quad 
  bc^2b,
  \\
  bcda + adcb, \quad 
  bda^2 + a^2db, \quad 
  bdab + badb, \quad 
  bdba + abdb, \quad 
  bdb^2 + b^2db,
  \\
  bdca + acdb - a, \quad
  bdcb + bcdb - b, \quad 
  bd^2a + ad^2b, \quad 
  bd^2b, \quad 
  ca^3 + a^3c - c,
  \\
  ca^2b + ba^2c - d, \quad 
  ca^2c, \quad
  caba + abac, \quad 
  cab^2 + b^2ac, \quad 
  caca + acac,
  \\
  cacb + bcac, \quad 
  cada + adac, \quad 
  cadb + bdac, \quad
  cba^2 + a^2bc - a, \quad
  cbab + babc, 
  \\
  cbac + cabc - c, \quad 
  cb^2a + ab^2c, \quad 
  cb^3 + b^3c, \quad 
  cb^2c, \quad
  cbca + acbc - c, 
  \\
  cbcb + bcbc - d - a, \quad 
  cbda + adbc, \quad 
  cbdb + bdbc, \quad
  c^2a^2 + a^2c^2, \quad
  c^2ab + bac^2, 
  \\
  c^2ac + cac^2, \quad 
  c^2ba + abc^2, \quad 
  c^2b^2 + b^2c^2, \quad
  c^2bc + cbc^2, \quad 
  c^3a + ac^3, \quad
  c^3b + bc^3, 
  \\
  c^4, \quad 
  c^2da + adc^2, \quad 
  c^2db + bdc^2, \quad
  cda^2 + a^2dc, \quad 
  cdab + badc, \quad 
  cdac + cadc,
  \\
  cdba + abdc - a, \quad 
  cdb^2 + b^2dc, \quad
  cdbc + cbdc - c, \quad 
  cdca + acdc, \quad 
  cdcb + bcdc,
  \\
  cdc^2 + c^2dc, \quad 
  cd^2a + ad^2c, \quad
  cd^2b + bd^2c - a, \quad 
  cd^2c, \quad 
  da^3 + a^3d, \quad 
  da^2b + ba^2d,
  \\
  da^2c + ca^2d, \quad 
  da^2d, \quad
  daba + abad, \quad 
  dab^2 + b^2ad, \quad 
  dabc + cbad, \quad 
  daca + acad,
  \\
  dacb + bcad, \quad
  dac^2 + c^2ad, \quad 
  dada + adad, \quad 
  dadb + bdad, \quad 
  dadc + cdad,
  \\
  dba^2 + a^2bd - b, \quad
  dbab + babd, \quad 
  dbac + cabd - d, \quad 
  dbad + dabd, \quad 
  db^2a + ab^2d,
  \\
  db^3 + b^3d, \quad
  db^2c + cb^2d, \quad
  db^2d, \quad 
  dbca + acbd, \quad 
  dbcb + bcbd - b, \quad 
  dbc^2 + c^2bd,
  \\
  dbda + adbd, \quad 
  dbdb + bdbd, \quad 
  dbdc + cdbd, \quad 
  dca^2 + a^2cd - c, \quad 
  dcab + bacd - d,
  \\
  dcac + cacd, \quad 
  dcad + dacd, \quad 
  dcba + abcd, \quad 
  dcb^2 + b^2cd, \quad 
  dcbc + cbcd - c,
  \\
  dcbd + dbcd - d, \quad 
  dc^2a + ac^2d, \quad 
  dc^2b + bc^2d, \quad 
  dc^3 + c^3d, \quad 
  dc^2d, \quad
  dcda + adcd,
  \\
  dcdb + bdcd, \quad 
  dcdc + cdcd, \quad 
  d^2a^2 + a^2d^2, \quad 
  d^2ab + bad^2, \quad
  d^2ac + cad^2,
  \\
  d^2ad + dad^2, \quad
  d^2ba + abd^2 - b, \quad 
  d^2b^2 + b^2d^2, \quad 
  d^2bc + cbd^2 - d, \quad 
  d^2bd + dbd^2,
  \\
  d^2ca + acd^2 - c, \quad
  d^2cb + bcd^2 - d, \quad 
  d^2c^2 + c^2d^2, \quad 
  d^2cd + dcd^2, \quad
  d^3a + ad^3,
  \\
  d^3b + bd^3 - b, \quad
  d^3c + cd^3 - c, \quad 
  d^4 - d. 
  \end{array}
  \]
  \caption{The 136 generators of the ideal $I$ for the Jordan tetrad}
  \label{JQS136}
  \[
  \begin{array}{l}
  ad, \qquad 
  b^2, \qquad 
  bd - ab, \qquad 
  c^2, \qquad 
  cd - ac, \qquad 
  da, \qquad 
  db - ba, \qquad 
  dc - ca,
  \\
  d^2 - cb - bc + a^2, \qquad 
  aba, \qquad 
  aca, \qquad 
  acb + abc - a^3, \qquad 
  bab, \qquad 
  bca - abc,
  \\
  bcb - ba^2 - a^2b, \qquad 
  cac, \qquad 
  cba + abc - a^3, \qquad 
  cbc - ca^2 - a^2c, \qquad 
  a^4 - a,
  \\ 
  ba^3 + a^3b - b, \qquad 
  ba^2b, \qquad 
  ca^3 + a^3c - c, \qquad 
  ca^2b + ba^2c - d, \qquad 
  ca^2c,
  \\
  cabc + a^3c - c. 
  \end{array}
  \]
  \caption{The Gr\"obner basis of the ideal $I$ for the Jordan tetrad}
  \label{JQS25}
  \end{table}  

\begin{exercise}
Use a computer algebra system to calculate the multiplication table of $U(Q)$.
Compute the Wedderburn decomposition of $U(Q)$. 
Prove or disprove that 
  \[
  U(T) \approx F \oplus M_2(F) \oplus M_2(F) \oplus M_2(F) \oplus M_2(F) \oplus M_2(F) \oplus M_2(F).
  \]
\end{exercise}

\begin{remark}
At present there is no general theory of the structures obtained from regarding
the Jordan tetrad as a quadrilinear operation on an associative algebra.
For the role played by tetrads in Jordan theory, 
see McCrimmon \cite{McCrimmon2004}.
\end{remark}

\begin{exercise}
Prove that if $J$ is a finite dimensional Jordan algebra then
its universal associative envelope is also finite dimensional.
\end{exercise}

\begin{exercise}
Prove that if $J$ is an $n$-dimensional Jordan algebra with zero product, 
then $U(J)$ is the exterior algebra of an $n$-dimensional vector space.
\end{exercise}

\begin{exercise}
Prove that if $J$ is the Jordan algebra of a symmetric bilinear form, 
then $U(J)$ is the corresponding Clifford algebra.
\end{exercise}


\section{Multilinear Operations} \label{sectionmultilinear}

We now consider generalizations of the two basic nonassociative bilinear operations,
the Lie bracket and the Jordan product, to $n$-linear operations for any integer $n \ge 2$.
This discussion is based on my papers with Peresi \cite{BremnerPeresi2007,BremnerPeresi2009}.

\subsection{Multilinear operations}

An $n$-linear operation $\omega( a_1, \dots, a_n )$ over a field $F$ 
is a linear combination of permutations of the monomial $a_1 \cdots a_n$.
We regard $\omega$ as a multilinear element of degree $n$ in the free associative algebra on $n$ generators:
  \[
  \omega( a_1, \dots, a_n ) = \sum_{\sigma \in S_n} x_\sigma a_{\sigma(1)} \cdots a_{\sigma(n)}
  \qquad
  ( x_\sigma \in F ),  
  \]
where the sum is over all permutations in the symmetric group $S_n$ acting on $\{ 1, \dots, n \}$.
We may also identify $\omega( a_1, \dots, a_n )$ with an element of $F S_n$, 
the group algebra of the symmetric group $S_n$:
  \[
  \omega( a_1, \dots, a_n ) = \sum_{\sigma \in S_n} x_\sigma \sigma
  \qquad
  ( x_\sigma \in F ).  
  \]
The group $S_n$ acts on $F S_n$ by permuting the subscripts of the generators:
  \[
  \sigma \cdot a_{\tau(1)} \cdots a_{\tau(n)} = a_{\sigma\tau(1)} \cdots a_{\sigma\tau(n)}.
  \]
Two $n$-linear operations are said to be equivalent if each is a linear combination of permutations of the other; 
that is, they generate the same left ideal in $F S_n$.

When discussing $n$-linear operations, we assume that the characteristic of $F$ is either 0 or a prime $p > n$;
this is a necessary and sufficient condition for the group algebra $F S_n$ to be semisimple.
In this case, $F S_n$ is the direct sum of simple two-sided ideals, each isomorphic to a matrix algebra
$M_d(F)$, and the projections of $S_n$ to these matrix algebras define the irreducible representations of $S_n$.

\subsection{The case $n = 2$}

Every bilinear operation is equivalent to one of the following:
the Lie bracket $[x,y] = xy - yx$, the Jordan product $x \circ y = \frac12 ( xy + yx )$, 
the original associative operation $xy$, and the zero operation.  
In other words, the only left ideals in the group algebra $F S_2 \approx F \oplus F$ are $\{0\} \oplus F$,
$F \oplus \{0\}$, $F \oplus F$, and $\{0\} \oplus \{0\}$.
The first summand $F$ corresponds to the unit representation of $S_2$, and a basis for this summand is the
idempotent $\frac12 ( xy + yx )$. 
The second summand corresponds to the sign representation, and a basis for this summand is 
the idempotent $\frac12 ( xy - yx )$.
These two idempotents are orthogonal in the sense that their product is zero.

\subsection{The case $n = 3$}

Faulkner \cite{Faulkner1985} classified the trilinear polynomial identities satisfied by a large class
of nearly simple triple systems.
Twenty years later, trilinear operations were classified up to equivalence in my work with Peresi \cite{BremnerPeresi2007};
we also determined the polynomial identities of degree 5 satisfied by these operations.
The structure of the group algebra in this case is 
  \[
  F S_3 \approx F \oplus M_2(F) \oplus F.
  \]
The first and last summands correspond to the unit and sign representations respectively;
bases for these summands are the following idempotents:
  \begin{align*}
  S &= \tfrac16 ( abc + acb + bac + bca + cab + cba ),
  \\
  A &= \tfrac16 ( abc - acb - bac + bca + cab - cba ).
  \end{align*}
The middle summand $M_2(F)$ corresponds to the irreducible 2-dimensional representation of $S_3$.
To find a basis for $M_2(F)$ corresponding to the matrix units $E_{ij}$ ($i, j = 1, 2$) 
we use the representation theory of the symmetric group developed by Young \cite{Young1977}
and simplified by Rutherford \cite{Rutherford1948} and Clifton \cite{Clifton1981}.
It follows that any trilinear operation can be represented as a triple of matrices:
  \[
  \bigg[ \, a, \, \begin{bmatrix} b_{11} & b_{12} \\ b_{21} & b_{22} \end{bmatrix}\!, \, c \bigg].
  \]
As representatives of the equivalence classes we may take the triples in which each matrix 
is in row canonical form.

Using computer algebra \cite{BremnerPeresi2007}, it can be shown that there are exactly 19 
trilinear operations satisfying polynomial identities in degree 5 which do not follow from
their identities in degree 3.
Simplified forms of these operations were later discovered by the author \cite{Bremner2012}
and Elgendy \cite{Elgendy2013}.
Together with these 19 operations, it is conventional to include the symmetric, alternating
and cyclic sums, even though for these operations, every identity in degree 5 follows
from those of degree 3; see my paper with Hentzel \cite{BremnerHentzel2000}.
These 22 trilinear operations are given in Table \ref{trilinearoperations}.
The first column gives the name of the operation;
the second column gives the row canonical forms of the
representation matrices of the corresponding element of the group algebra;
the third column gives the the simplest representative of the equivalence class 
as a linear combination of permutations.
(The parameter $q$ represents the $(1,2)$ entry of the $2 \times 2$ matrix.)

\begin{table}
\small
\begin{tabular}{lll}
operation 
&\quad 
$F \oplus M_2(F) \oplus F$ 
&\quad 
$F S_3$
\\ \midrule
symmetric sum 
&\quad
$\bigg[ \, 1, \, \begin{bmatrix} 0 & 0 \\ 0 & 0 \end{bmatrix}\!\!, \, 0 \, \bigg]$
&\quad
$abc+acb+bac+bca+cab+cba$
\\[8pt]
alternating sum
&\quad 
$\bigg[ \, 0, \, \begin{bmatrix} 0 & 0 \\ 0 & 0 \end{bmatrix}\!\!, \, 1 \, \bigg]$
&\quad
$abc-acb-bac+bca+cab-cba$
\\[8pt]
cyclic sum
&\quad 
$\bigg[ \, 1, \, \begin{bmatrix} 0 & 0 \\ 0 & 0 \end{bmatrix}\!\!, \, 1 \, \bigg]$
&\quad
$abc+bca+cab$
\\[8pt]
Lie $q = \infty$
&\quad 
$\bigg[ \, 0, \, \begin{bmatrix} 0 & 1 \\ 0 & 0 \end{bmatrix}\!\!, \, 0 \, \bigg]$
&\quad
$abc-acb-bca+cba$
\\[8pt]
Lie $q = \frac12$
&\quad 
$\bigg[ \, 0, \, \begin{bmatrix} 1 & \frac12 \\ 0 & 0 \end{bmatrix}\!\!, \, 0 \, \bigg]$
&\quad
$abc+acb-bca-cba$
\\[8pt]
Jordan $q = \infty$
&\quad 
$\bigg[ \, 1, \, \begin{bmatrix} 0 & 1 \\ 0 & 0 \end{bmatrix}\!\!, \, 0 \, \bigg]$
&\quad
$abc+cba$
\\[8pt]
Jordan $q = 0$
&\quad 
$\bigg[ \, 1, \, \begin{bmatrix} 1 & 0 \\ 0 & 0 \end{bmatrix}\!\!, \, 0 \, \bigg]$
&\quad
$abc+bac$
\\[8pt]
Jordan $q = 1$
&\quad 
$\bigg[ \, 1, \, \begin{bmatrix} 1 & 1 \\ 0 & 0 \end{bmatrix}\!\!, \, 0 \, \bigg]$
&\quad
$abc+acb$
\\[8pt]
Jordan $q = \frac12$
&\quad 
$\bigg[ \, 1, \, \begin{bmatrix} 1 & \frac12 \\ 0 & 0 \end{bmatrix}\!\!, \, 0 \, \bigg]$
&\quad
$abc+2acb+2cab+cba$
\\[8pt]
anti-Jordan $q = \infty$
&\quad 
$\bigg[ \, 0, \, \begin{bmatrix} 0 & 1 \\ 0 & 0 \end{bmatrix}\!\!, \, 1 \, \bigg]$
&\quad
$abc-2acb+2cab-cba$
\\[8pt]
anti-Jordan $q = -1$
&\quad 
$\bigg[ \, 0, \, \begin{bmatrix} 1 & -1 \\ 0 & 0 \end{bmatrix}\!\!, \, 1 \, \bigg]$
&\quad
$abc-acb$
\\[8pt]
anti-Jordan $q = \frac12$
&\quad 
$\bigg[ \, 0, \, \begin{bmatrix} 1 & \frac12 \\ 0 & 0 \end{bmatrix}\!\!, \, 1 \, \bigg]$
&\quad
$abc-cba$
\\[8pt]
anti-Jordan $q = 2$ 
&\quad 
$\bigg[ \, 0, \, \begin{bmatrix} 1 & 2 \\ 0 & 0 \end{bmatrix}\!\!, \, 1 \, \bigg]$
&\quad
$abc-bac$
\\[8pt]
fourth family $q = \infty$ 
&\quad 
$\bigg[ \, 1, \, \begin{bmatrix} 0 & 1 \\ 0 & 0 \end{bmatrix}\!\!, \, 1 \, \bigg]$
&\quad
$abc-acb-bac$
\\[8pt]
fourth family $q = 0$ 
&\quad 
$\bigg[ \, 1, \, \begin{bmatrix} 1 & 0 \\ 0 & 0 \end{bmatrix}\!\!, \, 1 \, \bigg]$
&\quad
$abc-acb+bca$
\\[8pt]
fourth family $q = 1$ 
&\quad 
$\bigg[ \, 1, \, \begin{bmatrix} 1 & 1 \\ 0 & 0 \end{bmatrix}\!\!, \, 1 \, \bigg]$
&\quad
$abc-bac+cab$
\\[8pt]
fourth family $q = -1$ 
&\quad 
$\bigg[ \, 1, \, \begin{bmatrix} 1 & -1 \\ 0 & 0 \end{bmatrix}\!\!, \, 1 \, \bigg]$
&\quad
$abc+bac+cab$
\\[8pt]
fourth family $q = 2$ 
&\quad 
$\bigg[ \, 1, \, \begin{bmatrix} 1 & 2 \\ 0 & 0 \end{bmatrix}\!\!, \, 1 \, \bigg]$
&\quad
$abc+acb+bca$
\\[8pt]
fourth family $q = \frac12$
&\quad 
$\bigg[ \, 1, \, \begin{bmatrix} 1 & \frac12 \\ 0 & 0 \end{bmatrix}\!\!, \, 1 \, \bigg]$
&\quad
$abc+acb+bac$
\\[8pt]
cyclic commutator 
&\quad 
$\bigg[ \, 0, \, \begin{bmatrix} 1 & 0 \\ 0 & 1 \end{bmatrix}\!\!, \, 0 \, \bigg]$
&\quad
$abc-bca$
\\[8pt]
weakly commutative
&\quad 
$\bigg[ \, 1, \, \begin{bmatrix} 1 & 0 \\ 0 & 1 \end{bmatrix}\!\!, \, 0 \, \bigg]$
&\quad
$abc+acb+bac-cba$
\\[8pt]
weakly anticommutative
&\quad 
$\bigg[ \, 0, \, \begin{bmatrix} 1 & 0 \\ 0 & 1 \end{bmatrix}\!\!, \, 1 \, \bigg]$
&\quad
$abc+acb-bca-cab$
\\ \midrule
\end{tabular}
\caption{The twenty-two trilinear operations}
\label{trilinearoperations}
\end{table}

\subsection{Associative $n$-ary algebras}

Simple associative triple systems were classified by Hestenes \cite{Hestenes1962}, 
Lister \cite{Lister1971} and Loos \cite{Loos1972};
their work was extended to simple associative $n$-ary systems by Carlsson \cite{Carlsson1980}.
The classification by Carlsson can be reformulated as follows.
Let $( d_1, \dots, d_{n-1} )$ be a sequence of $n{-}1$ positive integers;
two such sequences are regarded as equivalent if they differ only by a cyclic permutation.
For each $i = 1, \dots, n{-}1$, let $V_i$ be a vector space of dimension $d_i$ over $F$, 
and consider the direct sum $V = V_1 \oplus \cdots \oplus V_{n-1}$.
Let $A$ be the subspace of $\mathrm{End}_F(V)$ consisting of the linear operators $T\colon V \to V$
which satisfy the conditions 
  \[
  T(V_1) \subseteq V_2,
  \quad
  T(V_2) \subseteq V_3,
  \quad
  \dots,
  \quad
  T(V_{n{-}2}) \subseteq V_{n{-}1},
  \quad
  T(V_{n{-}1}) \subseteq V_1.
  \] 
Then $A$ is a simple associative $n$-ary system, and every such system has this form.
If we choose bases of the subspaces $V_1, \dots, V_{n-1}$ then we can represent the elements of $A$
as $D \times D$ block matrices where $D = d_1 + \cdots + d_{n-1}$.
The block in position $(i,j)$ where $1 \le i, j \le n{-}1$ has size $d_i \times d_j$; 
nonzero entries may appear only in blocks $(2,1)$, \dots, $(n{-}1,n)$, $(n,1)$.
To illustrate, for $n = 3, 4, 5$ we obtain the matrices of the following forms, 
where $T_{ij}$ is an arbitrary block of size $d_i \times d_j$:
  \[
  \left[
  \begin{array}{cc}
  0 & T_{12} \\
  T_{21} & 0
  \end{array}
  \right],
  \qquad
  \left[
  \begin{array}{ccc}
  0 & 0 & T_{13} \\
  T_{21} & 0 & 0 \\
  0 & T_{32} & 0
  \end{array}
  \right],
  \qquad
  \left[
  \begin{array}{cccc}
  0 & 0 & 0 & T_{14} \\
  T_{21} & 0 & 0 & 0 \\
  0 & T_{32} & 0 & 0 \\
  0 & 0 & T_{43} & 0
  \end{array}
  \right].
  \]

\subsection{Special nonassociative $n$-ary systems}

If $A$ is an associative $n$-ary system and $\omega(a_1, \dots, a_n)$ is an $n$-linear operation,
then we obtain a nonassociative $n$-ary system $A^\omega$ by interpreting each monomial in $\omega$
as the corresponding product in $A$.
Such a nonassociative $n$-ary system is called special (by analogy with special Jordan algebras) 
since it comes from a multilinear operation on an associative system.

In order to understand these nonassociative $n$-ary systems, we construct their universal associative envelopes
using the theory of noncommutative Gr\"obner bases.
The ultimate goal is to classify all the irreducible finite dimensional representations of these systems.
This generalizes the familiar construction of the universal enveloping algebras of Lie and Jordan algebras, 
where a dichotomy arises: 
a finite dimensional simple Lie algebra has an infinite dimensional universal envelope 
and infinitely many isomorphism classes of irreducible finite dimensional representations, 
but a finite dimensional simple Jordan algebra has a finite dimensional universal envelope 
and only finitely many irreducible representations.

\subsection{Universal associative envelopes}

This subsection gives the precise definition of the universal associative envelope of a nonassociative $n$-ary system
relative to an $n$-linear operation;
we consider only the case of a special nonassociative $n$-ary system.
The earliest discussion of this construction appears to be that of Birkhoff and Whitman
\cite[\S 2]{BirkhoffWhitman1949}; the presentation here follows my survey paper \cite[\S 7.2]{Bremner2012}.

Suppose that $B$ is a subspace, of an associative $n$-ary system $A$ over the field $F$, 
which is closed under the $n$-linear operation
  \[
  \omega( a_1, \dots, a_n ) = \sum_{\sigma \in S_n} x_\sigma a_{\sigma(1)} \cdots a_{\sigma(n)}
  \qquad
  ( x_\sigma \in F ). 
  \]
Set $d = \dim B$ and let $\{ b_1, \dots, b_d \}$ be a basis of $B$ over $F$; 
then we have the structure constants for the resulting nonassociative $n$-ary system $B^\omega$:
  \[
  \omega( b_{i_1}, \dots, b_{i_n} ) 
  =
  \sum_{j=1}^d
  c^j_{i_1 \cdots i_n} b_j
  \quad
  ( 1 \le i_1, \dots, i_n \le d ).
  \]
Let $F\langle X \rangle$ be the free associative algebra generated by the symbols $X = \{ b_1, \dots, b_d \}$
and consider the ideal $I \subseteq F\langle X \rangle$ generated by the following $d^n$ elements:
  \[
  \sum_{\sigma \in S_n} 
  x_\sigma b_{i_{\sigma(1)}} \cdots b_{i_{\sigma(n)}}
  -
  \sum_{j=1}^d
  c^j_{i_1 \cdots i_n} b_j
  \quad
  ( 1 \le i_1, \dots, i_n \le d ).
  \]
The quotient algebra $U(B^\omega) = F\langle X \rangle/I$ is the universal associative enveloping algebra of 
the nonassociative $n$-ary system $B^\omega$.
Since the $n$-ary structure on $B^\omega$ is special
(that is, defined in terms of the associative structure on $A$),
the natural map $B^\omega \to U(B^\omega)$ will necessarily be injective.

From this set of generators for the ideal $I$, we use the algorithm of 
Figure \ref{grobnerbasisalgorithm} to compute a Gr\"obner basis for $I$.
We then use this Gr\"obner basis to obtain a monomial basis for the universal associative envelope $U(B^\omega)$.
The multiplication table for $U(B^\omega)$ is then obtained by computing normal forms of products of basis monomials.
The next section is devoted to examples of this procedure.


\section{Special Nonassociative Triple Systems} \label{trilinearexamples}

In her Ph.D. thesis \cite{ElgendyThesis} and her forthcoming paper \cite{Elgendy2013}, 
Elgendy undertook a detailed study
using noncommutative Gr\"obner bases of the universal associative envelopes of the nonassociative triple systems 
obtained by applying the trilinear operations of Table \ref{trilinearoperations} to 
the 2-dimensional associative triple system $A_1$ of the form 
  \[
  \begin{bmatrix} 0 & \ast \\ \ast & 0 \end{bmatrix},
  \]
where $\ast$ represents an arbitrary scalar.
She distinguished two classes of operations: those of Lie type, for which the universal envelopes
are infinite dimensional; and those of Jordan type, for which the universal envelopes are finite dimensional.
For the operations of Lie type, she discovered that the universal envelopes are closely related to the down-up algebras
introduced by Benkart and Roby \cite{BenkartRoby1998}.
For the operations of Jordan type, she determined explicit Wedderburn decompositions of the universal
envelopes and classified the irreducible representations;
for these cases, she used the algorithms described in my survey paper \cite{Bremner2011}.

In this section, I consider the same problem for the 4- and 6-dimensional associative triple systems 
$A_2$ and $a_3$ consisting of all matrices of the forms
  \[
  \begin{bmatrix} 0 & \ast & \ast \\ \ast & 0 & 0 \\ \ast & 0 & 0 \end{bmatrix},
  \qquad
  \begin{bmatrix} 0 & \ast & \ast & \ast \\ \ast & 0 & 0 & 0 \\ \ast & 0 & 0 & 0 \\ \ast & 0 & 0 & 0 \end{bmatrix}.  
  \]
The resulting universal envelopes provide many examples of associative algebras,
both finite dimensional and infinite dimensional, that deserve further study.
It seems reasonable to expect that this will lead to generalizations of down-up algebras, and to 
nonassociative triple systems with many finite dimensional representations.

The computations are described in detail for $A_2$ and the results for $A_1$, $A_2$ and $A_3$ are summarized
in Table \ref{universalenvelopes}.
All calculations were done using Maple worksheets written by the author.

\subsection{Symmetric sum}

The original set of generators obtained from the structure constants consists of these 20 elements
which form a Gr\"obner basis for the ideal:
  \begin{align*}
  &
  a^3, \quad 
  ba^2 + aba + a^2b, \quad 
  b^2a + bab + ab^2, \quad 
  b^3, \quad 
  ca^2 + aca + a^2c - a,
  \\
  & 
  cba + cab + bca + bac + acb + abc - b, \quad 
  cb^2 + bcb + b^2c, \quad 
  c^2a + cac + ac^2 - c,
  \\
  & 
  c^2b + cbc + bc^2, \quad 
  c^3, \quad 
  da^2 + ada + a^2d, \quad 
  dba + dab + bda + bad + adb + abd - a,
  \\
  &
  db^2 + bdb + b^2d - b, \quad 
  dca + dac + cda + cad + adc + acd - d,
  \\
  &
  dcb + dbc + cdb + cbd + bdc + bcd - c, \quad 
  dc^2 + cdc + c^2d,
  \\
  &
  d^2a + dad + ad^2, \quad 
  d^2b + dbd + bd^2 - d, \quad 
  d^2c + dcd + cd^2, \quad 
  d^3. 
  \end{align*}
There are infinitely many monomials in $F\langle a,b,c,d \rangle$ which do not contain the leading monomial
of one of these generators as a subword, and so the universal envelope is infinite dimensional.
The first few dimensions of the homogeneous components of the associated graded algebra are as follows:
1, 4, 16, 44, 131, 344, 972, 2592, \dots.

\subsection{Alternating sum}

The original set of generators obtained from the structure constants consists of these 4 elements
which form a Gr\"obner basis for the ideal:
  \begin{align*}
  &
  cba - cab - bca + bac + acb - abc - b, 
  \quad 
  dba - dab - bda + bad + adb - abd + a,
  \\
  &
  dca - dac - cda + cad + adc - acd + d, 
  \quad 
  dcb - dbc - cdb + cbd + bdc - bcd - c.
  \end{align*} 
There are infinitely many monomials in $F\langle a,b,c,d \rangle$ which do not contain the leading monomial
of one of these generators as a subword, 
and so in this case again, the universal associative envelope is infinite dimensional.
The first few dimensions of the homogeneous components of the associated graded algebra are
1, 4, 16, 60, 225, 840, 3136, 11704, \dots.
The \emph{On-line Encyclopedia of Integer Sequences} 
\cite{OEIS}, sequence A072335, suggests that the generating function for these dimensions is
  \[
  \frac{1}{(1-x^2)(1-4x+x^2)}.  
  \]
Since the generating function has such a simple form, it seems reasonable to expect that the universal envelope 
has an interesting structure, and that the original 4-dimensional alternating triple system has a large class
of finite dimensional irreducible representations.

\begin{openproblem}
Prove the last claim about the generating function for the dimensions of the homogeneous components
of the associated graded algebra.
\end{openproblem}

\begin{openproblem}
Investigate the relationship between the universal envelopes for the symmetric and alternating sums 
and down-up algebras; see \cite{BenkartRoby1998} and \cite{Elgendy2013}.
\end{openproblem}

\subsection{Cyclic sum}

The original set of generators obtained from the structure constants consists of these 24 elements,
forming a self-reduced set:
  \begin{align*}
  &
  a^3, \quad 
  ba^2 + aba + a^2b, \quad 
  b^2a + bab + ab^2, \quad 
  b^3, \quad 
  ca^2 + aca + a^2c - a,
  \\
  &
  cab + bca + abc, \quad 
  cba + bac + acb - b, \quad 
  cb^2 + bcb + b^2c, \quad 
  c^2a + cac + ac^2 - c,
  \\
  & 
  c^2b + cbc + bc^2, \quad 
  c^3, \quad 
  da^2 + ada + a^2d, \quad 
  dab + bda + abd - a,
  \\
  &
  dac + cda + acd - d, \quad 
  dba + bad + adb, \quad 
  db^2 + bdb + b^2d - b, \quad 
  dbc + cdb + bcd,
  \\
  &
  dca + cad + adc, \quad 
  dcb + cbd + bdc - c, \quad 
  dc^2 + cdc + c^2d, \quad 
  d^2a + dad + ad^2,
  \\
  &
  d^2b + dbd + bd^2 - d, \quad 
  d^2c + dcd + cd^2, \quad 
  d^3. 
  \end{align*}
This is not a Gr\"obner basis; there are 40 distinct nontrivial compositions of these generators, 
the most complicated of which has this normal form:
  \begin{align*}
  &
  bacda + bacad + acbda - acadb - abcad + abadc -2 abacd + a^2cbd + a^2bdc 
  \\[-3pt]
  &
  -2 a^2bcd -2 bda - bad + adb - aca - abd -2 a^2c +2 a. 
  \end{align*}  
Combining the original 24 generators with the 40 compositions gives a set of 64 elements;
applying self-reduction to this set produces a new generating set of 59 elements.
This new generating set produces 724 distinct nontrivial compositions.
The combined set of 783 generators self-reduces to 62 elements, 
which form a Gr\"obner basis for the ideal:
  \begin{align*}
  &
  a^3, \quad 
  a^2b, \quad 
  a^2d, \quad
  aba, \quad 
  ab^2, \quad 
  abc, \quad 
  abd - a^2c, \quad 
  aca + a^2c - a, \quad 
  ada, \quad 
  adb,
  \\
  &
  adc, \quad 
  ad^2, \quad 
  ba^2, \quad 
  bab, \quad 
  bac + acb - b, \quad 
  bad, \quad 
  b^2a, \quad 
  b^3, \quad 
  b^2c, \quad 
  b^2d + acb - b,
  \\
  &
  bca, \quad 
  bcb, \quad 
  bc^2, \quad 
  bcd, \quad 
  bda + a^2c - a, \quad 
  bdb - acb, \quad 
  bdc - ac^2, \quad 
  bd^2 - acd,
  \\
  &
  ca^2, \quad 
  cab, \quad 
  cac + ac^2 - c, \quad 
  cad, \quad 
  cba, \quad 
  cb^2, \quad 
  cbc, \quad 
  cbd + ac^2 - c, \quad 
  c^2a, \quad 
  c^2b,
  \\
  &
  c^3, \quad 
  c^2d, \quad 
  cda, \quad 
  cdb, \quad 
  cdc, \quad 
  cd^2, \quad 
  da^2, \quad 
  dab, \quad 
  dac + acd - d, \quad 
  dad, \quad 
  dba,
  \\
  & 
  db^2, \quad 
  dbc, \quad 
  dbd + acd - d, \quad 
  dca, \quad 
  dcb, \quad 
  dc^2, \quad 
  dcd, \quad 
  d^2a, \quad 
  d^2b, \quad 
  d^2c, \quad 
  d^3,
  \\
  & 
  a^2cb - ab, \quad 
  a^2cd - ad. 
  \end{align*}
Only finitely many monomials in $F\langle a,b,c,d \rangle$ do not have a subword
equal to the leading monomial of an element of this Gr\"obner basis.
The universal associative envelope has a basis consisting of the cosets of these 26 monomials:
  \begin{align*}
  &
  1, \quad
  a, \quad
  b, \quad
  c, \quad
  d, \quad
  a^2, \quad
  ab, \quad
  ac, \quad
  ad, \quad
  ba, \quad
  b^2, \quad
  bc, \quad
  bd, \quad
  ca, \quad
  cb, \quad
  c^2,
  \\
  &
  cd, \quad
  da, \quad
  db, \quad
  dc, \quad
  d^2, \quad
  a^2c, \quad
  acb, \quad
  ac^2, \quad
  acd, \quad
  a^2c^2.
  \end{align*}  

\begin{exercise}
Determine the radical of the universal envelope in this case, and the decomposition of the semisimple
quotient into a direct sum of simple ideals.
\end{exercise}

\subsection{Lie $q = \infty$} 

In this case we are studying a simple Lie triple system; see Lister \cite{Lister1952}.
The original set of generators obtained from the structure constants consists of 24 elements;
after self-reduction, we are left with 20 elements:
  \begin{align*}  
  &
  ba^2 -2 aba + a^2b, \quad 
  b^2a -2 bab + ab^2, \quad 
  ca^2 -2 aca + a^2c +2 a,
  \\
  &
  cab - bca + bac - acb + b, \quad 
  cba - bca - acb + abc + b, \quad 
  cb^2 -2 bcb + b^2c,
  \\
  &
  c^2a -2 cac + ac^2 +2 c, \quad 
  c^2b -2 cbc + bc^2, \quad 
  da^2 -2 ada + a^2d,
  \\
  &
  dab - bda + bad - adb + a, \quad 
  dac - cda + cad - adc - d,
  \\
  &
  dba - bda - adb + abd + a, \quad 
  db^2 -2 bdb + b^2d +2 b, \quad 
  dbc - cdb + cbd - bdc - c,
  \\
  &
  dca - cda - adc + acd, \quad 
  dcb - cdb - bdc + bcd, \quad 
  dc^2 -2 cdc + c^2d,
  \\
  &
  d^2a -2 dad + ad^2, \quad 
  d^2b -2 dbd + bd^2 +2 d, \quad 
  d^2c -2 dcd + cd^2. 
  \end{align*}
There are 24 distinct compositions; the most complicated normal form is
  \begin{align*}
  &
  bdcda - bdadc - bcdad + badcd + adcdb - adbdc - acdbd + abdcd + bd^2 +3 acd.
  \end{align*} 
Combining the original 20 generators with the 24 compositions and
applying self-reduction gives a new generating set of 16 elements,
which is a Gr\"obner basis:
  \begin{align*}
  &
  ba - ab, \quad 
  dc - cd, \quad 
  ca^2 -2 aca + a^2c +2 a, \quad 
  cab - bca - acb + abc + b,
  \\
  &
  cb^2 -2 bcb + b^2c, \quad 
  c^2a -2 cac + ac^2 +2 c, \quad 
  c^2b -2 cbc + bc^2, \quad 
  da^2 -2 ada + a^2d,
  \\
  & 
  dab - bda - adb + abd + a, \quad 
  dac - cda + cad - acd - d, \quad 
  db^2 -2 bdb + b^2d +2 b,
  \\
  &
  dbc - cdb + cbd - bcd - c, \quad 
  d^2a -2 dad + ad^2, \quad 
  d^2b -2 dbd + bd^2 +2 d,
  \\
  &
  cbca - cacb - bcac + acbc + cb + bc, \quad 
  dbda - dadb - bdad + adbd - da - ad. 
  \end{align*}
There are infinitely many monomials in $F\langle a,b,c,d \rangle$ which do not contain the leading monomial
of one of these generators as a subword, 
and so in this case, the universal associative envelope is infinite dimensional.
The generating function for the dimensions seems to be as follows; see \cite{OEIS}, sequence A038164:
  \[
  \frac{1}{(1-x)^4(1-x^2)^4}.
  \] 
The first few terms are 1, 4, 14, 36, 85, 176, 344, 624, 1086, 1800, 2892, 4488, \dots.

\begin{openproblem}
Investigate the universal enveloping algebras of Lie triple systems and their representation theory.
Every finite dimensional Lie triple system can be embedded into a finite dimensional Lie algebra
as the odd subspace of a 2-grading on the Lie algebra.
For recent work, see Hodge and Parshall \cite{HodgeParshall2002}.
\end{openproblem}

\subsection{Lie $q = \frac12$} 

In this case we are studying a simple anti-Lie triple system.
The original set of generators obtained from the structure constants consists of 40 elements;
after self-reduction, we are left with 20 elements:
  \begin{align*}  
  &
  ba^2 - a^2b, \quad 
  b^2a - ab^2, \quad 
  ca^2 - a^2c, \quad 
  cab - bca - bac + acb - b,
  \\
  &
  cba + bca - acb - abc + b, \quad 
  cb^2 - b^2c, \quad 
  c^2a - ac^2, \quad 
  c^2b - bc^2, \quad 
  da^2 - a^2d,
  \\
  &
  dab - bda - bad + adb + a, \quad 
  dac - cda - cad + adc - d,
  \\
  &
  dba + bda - adb - abd - a, \quad 
  db^2 - b^2d, \quad 
  dbc - cdb - cbd + bdc + c,
  \\
  &
  dca + cda - adc - acd, \quad 
  dcb + cdb - bdc - bcd, \quad 
  dc^2 - c^2d, \quad 
  d^2a - ad^2,
  \\
  &
  d^2b - bd^2, \quad 
  d^2c - cd^2. 
  \end{align*} 
There are 26 distinct nontrivial compositions of these generators, 
the most complicated of which has normal form
  \begin{align*}
  &
  bdcda - bdadc - bcdad + badcd - adcdb + adbdc + acdbd - abdcd - bd^2 - acd.  
  \end{align*}
Combining the original 20 generators with the 26 compositions 
and applying self-reduction gives a new generating set of 12 elements,
which is a Gr\"obner basis:
  \begin{align*}  
  & 
  a^2, \quad 
  ba + ab, \quad 
  b^2, \quad 
  c^2, \quad 
  dc + cd, \quad 
  d^2, \quad
  cab - bca + acb + abc - b,
  \\
  & 
  dab - bda + adb + abd + a, \quad 
  dac - cda - cad - acd - d,
  \\
  &
  dbc - cdb - cbd - bcd + c, \quad 
  cbca - cacb - bcac + acbc + cb - bc,
  \\
  &
  dbda - dadb - bdad + adbd - da + ad.
  \end{align*} 
There are infinitely many monomials in $F\langle a,b,c,d \rangle$ which do not contain the leading monomial
of one of these generators as a subword, 
and so again in this case, the universal associative envelope is infinite dimensional.
The first few dimensions of the homogeneous components in the associated graded algebra are
1, 4, 10, 20, 35, 56, 84, 120, 165, 220, 286, 364, 455, 560, 680, 816, 969, \dots.
According to \cite{OEIS}, sequence A000292, these are the tetrahedral numbers;
the generating function is
  \[
  \sum_{n=1}^\infty \binom{n+2}{3} x^n.
  \]

\begin{openproblem}
Investigate the universal enveloping algebras of anti-Lie triple systems.
Every finite dimensional anti-Lie triple system can be embedded into a finite dimensional Lie superalgebra
as the odd subspace.
See the recent monograph by Musson \cite{Musson2012} on Lie superalgebras and their enveloping algebras.
\end{openproblem}

\subsection{Jordan $q = \infty$} 

In this case we are studying a simple Jordan triple system.
The original set of 40 generators is already self-reduced:
  \begin{align*}
  &
  a^3, \quad 
  aba, \quad 
  aca - a, \quad 
  ada, \quad 
  ba^2 + a^2b, \quad 
  bab, \quad 
  b^2a + ab^2, \quad 
  b^3, \quad 
  bca + acb - b,
  \\
  &
  bcb, \quad 
  bda + adb - a, \quad 
  bdb - b, \quad 
  ca^2 + a^2c, \quad 
  cab + bac, \quad 
  cac - c, \quad 
  cba + abc,
  \\
  & 
  cb^2 + b^2c, \quad 
  cbc, \quad 
  c^2a + ac^2, \quad 
  c^2b + bc^2, \quad 
  c^3, \quad 
  cda + adc, \quad 
  cdb + bdc, \quad 
  cdc,
  \\
  &
  da^2 + a^2d, \quad 
  dab + bad, \quad 
  dac + cad - d, \quad 
  dad, \quad 
  dba + abd, \quad 
  db^2 + b^2d,
  \\
  &
  dbc + cbd - c, \quad 
  dbd - d, \quad 
  dca + acd, \quad 
  dcb + bcd, \quad 
  dc^2 + c^2d, \quad 
  dcd,
  \\
  &
  d^2a + ad^2, \quad
  d^2b + bd^2, \quad 
  d^2c + cd^2, \quad 
  d^3.
  \end{align*} 
There are 32 distinct nontrivial compositions; their normal forms are
  \begin{align*}
  &
  a^2, \quad 
  ab, \quad 
  ba, \quad 
  b^2, \quad 
  c^2, \quad 
  cd, \quad 
  dc, \quad 
  d^2, \quad 
  a^2b, \quad 
  a^2c, \quad 
  a^2d, \quad 
  ab^2, \quad 
  abc,
  \\
  &
  abd, \quad 
  abd + a^2c, \quad 
  ac^2, \quad 
  acd, \quad 
  adc, \quad 
  ad^2, \quad 
  bac, \quad 
  bad, \quad 
  b^2c, \quad 
  b^2d, \quad 
  \\
  & 
  b^2d + bac, \quad
  bc^2, \quad 
  bcd, \quad 
  bdc + ac^2, \quad 
  bdc, \quad 
  bd^2 + acd, \quad 
  bd^2, \quad 
  c^2d, \quad 
  cd^2.
  \end{align*} 
The combined set of 72 elements self-reduces to 20, forming a Gr\"obner basis:
  \begin{align*}
  &
  a^2, \quad 
  ab, \quad 
  ba, \quad 
  b^2, \quad 
  c^2, \quad 
  cd, \quad 
  dc, \quad 
  d^2, \quad 
  aca - a, \quad 
  ada, \quad 
  bca + acb - b, 
  \\
  & 
  bcb, \quad
  bda + adb - a, \quad 
  bdb - b, \quad 
  cac - c, \quad 
  cbc, \quad 
  dac + cad - d, \quad 
  dad,
  \\
  &
  dbc + cbd - c, \quad
  dbd - d.
  \end{align*}
There are only 19 monomials which do not contain the leading monomial of an element of the Gr\"obner basis,
so the universal associative envelope is finite dimensional and has the
cosets of the following monomials as a basis:
  \begin{align*}
  &
  1, \quad
  a, \quad
  b, \quad
  c, \quad
  d, \quad
  ac, \quad
  ad, \quad
  bc, \quad
  bd, \quad
  ca, \quad
  cb, \quad
  da, \quad
  db,
  \\
  &
  acb, \quad
  adb, \quad
  cad, \quad
  cbd, \quad
  acbd, \quad
  cadb.
  \end{align*}  

\begin{exercise}
Compute the Wedderburn decomposition of the universal associative envelope of this Jordan triple system. 
In particular, prove or disprove that the envelope is isomorphic to $F \oplus M_3(F) \oplus M_3(F)$.
\end{exercise}

\begin{exercise}
Let $T$ be a finite dimensional Jordan triple system.  Prove that $U(T)$ is also finite dimensional.
\end{exercise}

\subsection{Jordan $q = 0$} 

The original self-reduced set of generators has 40 elements.
There are 20 distinct nontrivial compositions, 
and the combined set of 60 elements self-reduces to 27 elements.
These 27 generators have 4 distinct nontrivial compositions,
and the combined set of 31 elements self-reduces to 15 elements,
which form a Gr\"obner basis for the ideal:
  \begin{equation}
  \label{jordangrobner}
  \left\{ \;
  \begin{array}{l}
  a^2, \qquad 
  ab, \qquad 
  ad, \qquad 
  ba, \qquad 
  b^2, \qquad
  bc, \qquad 
  bd - ac, \qquad 
  c^2, \qquad 
  cd, 
  \\ 
  dc, \qquad
  d^2, \qquad 
  aca - a, \qquad 
  acb - b, \qquad 
  cac - c, \qquad 
  dac - d.
  \end{array}
  \right.
  \end{equation}
The universal associative envelope has dimension 10 with basis consisting of the cosets of the elements
$1$, $a$, $b$, $c$, $d$, $ac$, $ca$, $cb$, $da$, $db$.
See Exercise \ref{jordanexercise} below.

\subsection{Jordan $q = 1$} 

The original self-reduced set of generators has 40 elements.
There are 19 distinct nontrivial compositions, 
and the combined set of 59 elements self-reduces to 27 elements.
These 27 generators have 6 distinct nontrivial compositions,
and the combined set of 33 elements self-reduces to 15 elements,
which form a Gr\"obner basis for the ideal.
This Gr\"obner basis is the same as in the previous case,
and so the universal envelopes are isomorphic.

\subsection{Jordan $q = \frac12$} 

The original self-reduced set of 40 generators has 94 distinct nontrivial compositions, 
and the combined set of 134 elements self-reduces to 15 elements,
which form the Gr\"obner basis \eqref{jordangrobner}.

\begin{exercise} \label{jordanexercise}
Compute the Wedderburn decomposition of 
the universal envelope for the Gr\"obner basis \eqref{jordangrobner}. 
Prove or disprove that it is isomorphic to $F \oplus M_3(F)$.
\end{exercise}

\subsection{Anti-Jordan $q = \infty$} 

The original self-reduced set of 24 generators has 76 distinct nontrivial compositions, 
and the combined set of 100 generators self-reduces to 15 elements,
which is the Gr\"obner basis \eqref{jordangrobner}.

\subsection{Anti-Jordan $q = -1$} 

The original self-reduced set of 24 generators has 37 distinct nontrivial compositions.
The combined set of 61 elements self-reduces to 23 elements with 6 distinct nontrivial compositions.
The combined set of 29 elements self-reduces to 15 elements,
which is the Gr\"obner basis \eqref{jordangrobner}.

\subsection{Anti-Jordan $q = 2$} 

The original self-reduced set of 24 generators has 37 distinct nontrivial compositions.
The combined set of 61 elements self-reduces to 23 elements with 4 distinct nontrivial compositions.
The combined set of 27 elements self-reduces to 15 elements,
which is the Gr\"obner basis \eqref{jordangrobner}.

\subsection{Anti-Jordan $q = \frac12$} 

In this case we are studying a simple anti-Jordan triple system; 
see Faulkner and Ferrar \cite{FaulknerFerrar1980} and the Ph.D. thesis of Bashir \cite{Bashir2008}.
The original self-reduced set of 24 generators is as follows:
  \begin{align*}
  & 
  ba^2 - a^2b, \quad 
  b^2a - ab^2, \quad 
  bca - acb + b, \quad 
  bda - adb - a, \quad 
  ca^2 - a^2c, \quad 
  cab - bac,
  \\
  &
  cba - abc, \quad 
  cb^2 - b^2c, \quad 
  c^2a - ac^2, \quad 
  c^2b - bc^2, \quad 
  cda - adc, \quad 
  cdb - bdc,
  \\
  &
  da^2 - a^2d, \quad 
  dab - bad, \quad 
  dac - cad - d, \quad 
  dba - abd, \quad 
  db^2 - b^2d, \quad 
  dbc - cbd + c,
  \\
  & 
  dca - acd, \quad 
  dcb - bcd, \quad 
  dc^2 - c^2d, \quad 
  d^2a - ad^2, \quad 
  d^2b - bd^2, \quad 
  d^2c - cd^2.
  \end{align*}
There are 32 distinct nontrivial compositions:
  \begin{align*}
  &
  a^2, \quad 
  ab, \quad 
  ba, \quad 
  b^2, \quad 
  c^2, \quad 
  cd, \quad 
  dc, \quad 
  d^2, \quad 
  a^2b, \quad 
  a^2c, \quad 
  a^2d, \quad 
  ab^2, \quad 
  abc,
  \\
  & 
  abd, \quad 
  abd + a^2c, \quad 
  ac^2, \quad 
  acd, \quad 
  adc, \quad 
  ad^2, \quad 
  bac, \quad 
  bad, \quad 
  b^2c, \quad 
  b^2d,
  \\
  &
  b^2d + bac, \quad
  bc^2, \quad 
  bcd, \quad 
  bdc, \quad 
  bdc + ac^2, \quad 
  bd^2 + acd, \quad 
  bd^2, \quad 
  c^2d, \quad 
  cd^2.
  \end{align*}
The combined set of 56 elements self-reduces to a Gr\"obner basis of 12 elements:
  \begin{align*}
  &
  a^2, \quad 
  ab, \quad 
  ba, \quad 
  b^2, \quad 
  c^2, \quad 
  cd, \quad 
  dc, \quad 
  d^2, \quad 
  bca - acb + b, \quad 
  bda - adb - a,
  \\
  &
  dac - cad - d, \quad 
  dbc - cbd + c. 
  \end{align*}
The universal associative envelope is infinite dimensional;
the dimensions of the homogeneous components of 
the associated graded algebra appear to be
  \[
  \tfrac12 (n+1) (n+3) \quad \text{($n$ odd)},
  \qquad
  \tfrac12 (n+2)^2 \quad \text{($n$ even)}.
  \]
  
\begin{openproblem}
Prove that the universal envelope is infinite dimensional, 
and that the dimensions of the homogeneous components are as stated.
\end{openproblem}

\subsection{Fourth family $q = \infty$} 

The original set of 52 generators self-reduces to 44 elements,
which have 140 distinct nontrivial compositions.
Self-reducing the combined set of 184 generators produces the Gr\"obner basis \eqref{jordangrobner}.

\subsection{Fourth family $q = 0$} 

The original set of 52 generators self-reduces to 44 elements,
which have 88 distinct nontrivial compositions.
Self-reducing the combined set of 132 generators produces the Gr\"obner basis \eqref{jordangrobner}.

\subsection{Fourth family $q = 1$} 

The original set of 52 generators self-reduces to 44 elements,
which have 76 distinct nontrivial compositions.
Self-reducing the combined set of 120 generators produces the Gr\"obner basis \eqref{jordangrobner}. 

\subsection{Fourth family $q = -1$} 

The original set of 64 generators self-reduces to 44 elements,
which have 209 distinct nontrivial compositions.
Self-reducing the combined set of 253 generators produces the Gr\"obner basis \eqref{jordangrobner}. 

\subsection{Fourth family $q = 2$} 

The original set of 64 generators self-reduces to 44 elements,
which have 227 distinct nontrivial compositions.
Self-reducing the combined set of 271 generators produces the Gr\"obner basis \eqref{jordangrobner}.

\subsection{Fourth family $q = \frac12$} 

The original set of 64 generators self-reduces to 44 elements,
which have 184 distinct nontrivial compositions.
Self-reducing the combined set of 228 generators produces the Gr\"obner basis \eqref{jordangrobner}.

\subsection{Cyclic commutator} 

The original set of 60 generators self-reduces to 40 elements,
which have 86 distinct nontrivial compositions.
Self-reducing the combined set of 126 generators produces the Gr\"obner basis \eqref{jordangrobner}.

\subsection{Weakly commutative operation} 

The original set of 64 generators self-reduces to 60 elements,
which have 15 distinct nontrivial compositions.
Self-reducing the combined set of 75 generators produces the Gr\"obner basis \eqref{jordangrobner}. 

\subsection{Weakly anticommutative operation} 

The original set of 60 generators self-reduces to 44 elements,
which have 41 distinct nontrivial compositions.
Self-reducing the combined set of 85 generators produces the Gr\"obner basis \eqref{jordangrobner}. 

\begin{table}
\begin{tabular}{llll}
operation 
&\!\!\!\!\!\!
$U(A_1^\omega)$
&\!\!\!\!\!\!
$U(A_2^\omega)$
&\!\!\!\!\!\!
$U(A_3^\omega)$
\\ 
\midrule
Sym sum
&\!\!\!\!\!\!
$\Big\{\!\! \begin{array}{l} 4 \\ 1{,}2{,}4{,}\textbf{4{,}5}{,}\dots \end{array}$
&\!\!\!\!\!\!
$\Big\{\!\! \begin{array}{l} 20 \\ 1{,}4{,}16{,}44{,}131{,}344{,}\dots \end{array}$
&\!\!\!\!\!\!
$\Big\{\!\! \begin{array}{l} 56 \\ 1{,}6{,}36{,}160{,}750{,}3240{,}\dots \end{array}$
\\[6pt]
Alt sum
&\!\!\!\!\!\!
$\Big\{\!\! \begin{array}{l} 0 \\ 1{,}2{,}4{,}8{,}16{,}32{,}\dots \end{array}$
&\!\!\!\!\!\!
$\Big\{\!\! \begin{array}{l} 4 \\ 1{,}4{,}16{,}60{,}225{,}840{,}\dots \end{array}$
&\!\!\!\!\!\!
$\Big\{\!\! \begin{array}{l} 20 \\ 1{,}6{,}36{,}196{,}1071{,}5796{,}\dots \end{array}$
\\[6pt]
Cyc sum
&\!\!\!\!\!\!
$\Big\{\!\! \begin{array}{l} 4 \\ 1{,}2{,}4{,}\textbf{4{,}5}{,}\dots \end{array}$
&\!\!\!\!\!\!
$\Big\{\!\! \begin{array}{l} 24{,}40 \,|\, 59{,}724 \,|\, 62 \\ 26 \end{array}$
&\!\!\!\!\!\!
$\Big\{\!\! \begin{array}{l} \text{unable to complete} \end{array}$
\\[6pt]
Lie $q = \infty$
&\!\!\!\!\!\!
$\Big\{\!\! \begin{array}{l} 2 \\ 1{,}2{,}4{,}6{,}9{,}12{,}\dots \end{array}$
&\!\!\!\!\!\!
$\Big\{\!\! \begin{array}{l} 20{,}24 \,|\, 16 \\ 1{,}4{,}14{,}36{,}85{,}176{,}\dots \end{array}$
&\!\!\!\!\!\!
$\Big\{\!\! \begin{array}{l} 70{,}140 \,|\, 51 \\ 1{,}6{,}30{,}110{,}360{,}1026{,}\dots \end{array}$
\\[6pt]
Lie $q = \frac12$
&\!\!\!\!\!\!
$\Big\{\!\! \begin{array}{l} 2 \\ 1{,}2{,}4{,}6{,}9{,}12{,}\dots \end{array}$
&\!\!\!\!\!\!
$\Big\{\!\! \begin{array}{l} 20{,}26 \,|\, 12 \\ 1{,}4{,}10{,}20{,}35{,}56{,}\dots \end{array}$
&\!\!\!\!\!\!
$\Big\{\!\! \begin{array}{l} 70{,}147 \,|\, 39 \\ 1{,}6{,}24{,}74{,}195{,}456{,}\dots \end{array}$
\\[6pt]
Jor $q = \infty$
&\!\!\!\!\!\!
$\Big\{\!\! \begin{array}{l} 6{,}4 \,|\, 4 \\ 5 \end{array}$
&\!\!\!\!\!\!
$\Big\{\!\! \begin{array}{l} 40{,}32 \,|\, 20 \\ 19 \end{array}$
&\!\!\!\!\!\!
$\Big\{\!\! \begin{array}{l} 126{,}107 \,|\, 54 \\ 69 \end{array}$
\\[6pt]
Jor $q = 0$
&\!\!\!\!\!\!
$\Big\{\!\! \begin{array}{l} 6 \\ 9 \end{array}$
&\!\!\!\!\!\!
$\Big\{\!\! \begin{array}{l} 40{,}20 \,|\, 27{,}4 \,|\, 15 \\ 10 \end{array}$
&\!\!\!\!\!\!
$\Big\{\!\! \begin{array}{l} 126{,}97 \,|\, 71{,}9 \,|\, 32 \\ 17 \end{array}$
\\[6pt]
Jor $q = 1$
&\!\!\!\!\!\!
$\Big\{\!\! \begin{array}{l} 6 \\ 9 \end{array}$
&\!\!\!\!\!\!
$\Big\{\!\! \begin{array}{l} 40{,}19 \,|\, 27{,}6 \,|\, 15 \\ 10 \end{array}$
&\!\!\!\!\!\!
$\Big\{\!\! \begin{array}{l} 126{,}93 \,|\, 71{,}18 \,|\, 32 \\ 17 \end{array}$
\\[6pt]
Jor $q = \frac12$
&\!\!\!\!\!\!
$\Big\{\!\! \begin{array}{l} 6{,}4 \,|\, 4 \\ 5 \end{array}$
&\!\!\!\!\!\!
$\Big\{\!\! \begin{array}{l} 40{,}94 \,|\, 15 \\ 10 \end{array}$
&\!\!\!\!\!\!
$\Big\{\!\! \begin{array}{l} 126{,}542 \,|\, 32 \\ 17 \end{array}$
\\[6pt]
AJ $q = \infty$
&\!\!\!\!\!\!
$\Big\{\!\! \begin{array}{l} 2 \\ 1{,}2{,}4{,}6{,}9{,}12{,}\dots \end{array}$
&\!\!\!\!\!\!
$\Big\{\!\! \begin{array}{l} 24{,}76 \,|\, 15 \\ 10 \end{array}$
&\!\!\!\!\!\!
$\Big\{\!\! \begin{array}{l} 90{,}513 \,|\, 32 \\ 17 \end{array}$
\\[6pt]
AJ $q = -1$
&\!\!\!\!\!\!
$\Big\{\!\! \begin{array}{l} 2{,}2 \,|\, 4{,}2 \,|\, 4 \\ 5 \end{array}$
&\!\!\!\!\!\!
$\Big\{\!\! \begin{array}{l} 24{,}37 \,|\, 23{,}6 \,|\, 15 \\ 10 \end{array}$
&\!\!\!\!\!\!
$\Big\{\!\! \begin{array}{l} 90{,}135 \,|\, 62{,}18 \,|\, 32 \\ 17 \end{array}$
\\[6pt]
AJ $q = \frac12$
&\!\!\!\!\!\!
$\Big\{\!\! \begin{array}{l} 2 \\ 1{,}2{,}4{,}6{,}9{,}12{,}\dots \end{array}$
&\!\!\!\!\!\!
$\Big\{\!\! \begin{array}{l} 24{,}32 \,|\, 12 \\ 1{,}4{,}8{,}12{,}18{,}24{,}\dots \end{array}$
&\!\!\!\!\!\!
$\Big\{\!\! \begin{array}{l} 90{,}107 \,|\, 36 \\ 1{,}6{,}18{,}36{,}72{,}120{,}\dots \end{array}$
\\[6pt]
AJ $q = 2$ 
&\!\!\!\!\!\!
$\Big\{\!\! \begin{array}{l} 2{,}2 \,|\, 4{,}2 \,|\, 4 \\ 5 \end{array}$
&\!\!\!\!\!\!
$\Big\{\!\! \begin{array}{l} 24{,}37 \,|\, 23{,}4 \,|\, 15 \\ 10 \end{array}$
&\!\!\!\!\!\!
$\Big\{\!\! \begin{array}{l} 90{,}137 \,|\, 62{,}9 \,|\, 32 \\ 17 \end{array}$
\\[6pt]
4th $q = \infty$ 
&\!\!\!\!\!\!
$\Big\{\!\! \begin{array}{l} 6{,}4 \,|\, 4 \\ 5 \end{array}$
&\!\!\!\!\!\!
$\Big\{\!\! \begin{array}{l} 40{,}140 \,|\, 15 \\ 10 \end{array}$
&\!\!\!\!\!\!
$\Big\{\!\! \begin{array}{l} 146{,}1065 \,|\, 32 \\ 17 \end{array}$
\\[6pt]
4th $q = 0$ 
&\!\!\!\!\!\!
$\Big\{\!\! \begin{array}{l} 6 \\ 9 \end{array}$
&\!\!\!\!\!\!
$\Big\{\!\! \begin{array}{l} 44{,}88 \,|\, 15 \\ 10 \end{array}$
&\!\!\!\!\!\!
$\Big\{\!\! \begin{array}{l} 146{,}737 \,|\, 32 \\ 17 \end{array}$
\\[6pt]
4th $q = 1$ 
&\!\!\!\!\!\!
$\Big\{\!\! \begin{array}{l} 6 \\ 9 \end{array}$
&\!\!\!\!\!\!
$\Big\{\!\! \begin{array}{l} 44{,}76 \,|\, 15 \\ 10 \end{array}$
&\!\!\!\!\!\!
$\Big\{\!\! \begin{array}{l} 146{,}618 \,|\, 32 \\ 17 \end{array}$
\\[6pt]
4th $q = -1$ 
&\!\!\!\!\!\!
$\Big\{\!\! \begin{array}{l} 6{,}5 \,|\, 4 \\ 5 \end{array}$
&\!\!\!\!\!\!
$\Big\{\!\! \begin{array}{l} 44{,}209 \,|\, 15 \\ 10 \end{array}$
&\!\!\!\!\!\!
$\Big\{\!\! \begin{array}{l} 146{,}1432 \,|\, 32 \\ 17 \end{array}$
\\[6pt]
4th $q = 2$ 
&\!\!\!\!\!\!
$\Big\{\!\! \begin{array}{l} 6{,}5 \,|\, 4 \\ 5 \end{array}$
&\!\!\!\!\!\!
$\Big\{\!\! \begin{array}{l} 44{,}227 \,|\, 15 \\ 10 \end{array}$
&\!\!\!\!\!\!
$\Big\{\!\! \begin{array}{l} 146{,}1601 \,|\, 32 \\ 17 \end{array}$
\\[6pt]
4th $q = \frac12$
&\!\!\!\!\!\!
$\Big\{\!\! \begin{array}{l} 6{,}4 \,|\, 4 \\ 5 \end{array}$
&\!\!\!\!\!\!
$\Big\{\!\! \begin{array}{l} 44{,}184 \,|\, 15 \\ 10 \end{array}$
&\!\!\!\!\!\!
$\Big\{\!\! \begin{array}{l} 146{,}1347 \,|\, 32 \\ 17 \end{array}$
\\[6pt]
Cyc com
&\!\!\!\!\!\!
$\Big\{\!\! \begin{array}{l} 4{,}4 \,|\, 4 \\ 5 \end{array}$
&\!\!\!\!\!\!
$\Big\{\!\! \begin{array}{l} 40{,}86 \,|\, 15 \\ 10 \end{array}$
&\!\!\!\!\!\!
$\Big\{\!\! \begin{array}{l} 140{,}396 \,|\, 32 \\ 17 \end{array}$
\\[6pt]
Weak C
&\!\!\!\!\!\!
$\Big\{\!\! \begin{array}{l} 8{,}2 \,|\, 4 \\ 5 \end{array}$
&\!\!\!\!\!\!
$\Big\{\!\! \begin{array}{l} 60{,}15 \,|\, 15 \\ 10 \end{array}$
&\!\!\!\!\!\!
$\Big\{\!\! \begin{array}{l} 196{,}58 \,|\, 32 \\ 17 \end{array}$
\\[6pt]
Weak AC
&\!\!\!\!\!\!
$\Big\{\!\! \begin{array}{l} 4{,}4 \,|\, 4 \\ 5 \end{array}$
&\!\!\!\!\!\!
$\Big\{\!\! \begin{array}{l} 44{,}41 \,|\, 15 \\ 10 \end{array}$
&\!\!\!\!\!\!
$\Big\{\!\! \begin{array}{l} 160{,}124 \,|\, 32 \\ 17 \end{array}$
\\ \midrule
\end{tabular}
\caption{Universal associative envelopes of nonassociative triple systems}
\label{universalenvelopes}
\end{table}

\subsection{Summary}

Table \ref{universalenvelopes} summarizes the results of Elgendy \cite{ElgendyThesis,Elgendy2013} for $U(A_1^\omega)$,
the results of this section for $U(A_2^\omega)$, and further computations for $U(A_3^\omega)$.
Each entry in the table has the form 
  \[
  \Big\{ \begin{array}{l} \text{algorithm} \\ \text{dimension} \end{array}
  \]
where ``algorithm'' describes the performance of the Gr\"obner basis algorithm,
and ``dimension'' gives the dimension of the universal associative envelope.
The ``algorithm'' data consists of a sequence of pairs $x,y$ corresponding to the iterations of the algorithm;
$x$ is the size of the self-reduced set 
of generators at the start of the iteration, and $y$ is the number of distinct nontrivial compositions
in normal form at the end of the iteration (if $y = 0$ it is omitted). 
The ``dimension'' data consists either of a single number (in the case where the universal envelope is
finite dimensional), or a sequence of numbers giving the first few dimensions of the homogeneous components
of the associated graded algebra (in the case where the universal envelope is infinite dimensional).
Dimensions in boldface indicate values that repeat indefinitely.

For example, consider the entry in row ``Lie $q = \infty$'' and column ``$U(A_3^\omega)$'',
the Lie triple product on the 6-dimensional simple associative triple system:
  \[
  \Big\{\!\! \begin{array}{l} 70,140 \,|\, 51 \\ 1,6,30,110,360,1026,\dots \end{array}
  \]
This means:
  \begin{enumerate}
  \item[(a)]
The algorithm terminated after two iterations:
the original self-reduced set of 70 generators produced 140 nontrivial compositions;
the combined set of 210 generators self-reduced to a Gr\"obner basis of 51 elements.
  \item[(b)]
The universal associative envelope is infinite dimensional, and the generating function for the dimensions of the
homogeneous components of the associated graded algebra begins with the terms
  \[
  1 + 6z + 30z^2 + 110z^3 + 360z^4 + 1026z^5 + \cdots
  \]
  \end{enumerate}
In one case, $U(A_3^\omega)$ for the cyclic sum, the computations were so complicated that Maple 14
on my MacBook Pro was unable to complete them in a reasonable time.  
This may be related to the fact that the polynomial identities satisfied by this operation are extremely complicated;
see my paper with Peresi \cite{BremnerPeresi2009b}.

\subsection{Conclusions}

The results of Table \ref{universalenvelopes} suggest a slightly different classification 
of operations into ``Lie type'' and ``Jordan type'' from that of Elgendy \cite{ElgendyThesis,Elgendy2013}.
Two operations, the cyclic sum and the anti-Jordan $q = \infty$ operation,
produce infinite dimensional envelopes for $A_1^\omega$ but finite dimensional envelopes
for $A_2^\omega$.  
It seems likely that $A_1^\omega$ is exceptional, owing to its small dimension, 
and that the universal associative envelopes will be finite dimensional when either of these operations
is applied to a simple associative triple system of dimension $> 2$.
If this is correct, then these two operations should be reclassified as having ``Jordan type''.  

Four operations produced a non-semisimple envelope for $A_1^\omega$:
Jordan $q = 0, 1$ and fourth family $q = 0, 1$.
In these cases, the 9-dimensional envelope has a 4-dimensional radical and a 5-dimensional semisimple quotient
which is isomorphic to $F \oplus M_2(F)$.
For these operations it seems very likely that $U(A_n^\omega)$ ($n = 2, 3$) is semisimple
and is isomorphic to $F \oplus M_{n+1}(F)$.
The reason is that the dimension of the envelope (10 for $n = 2$ and 17 for $n = 3$) 
is the sum of the squares of the dimensions of the 1-dimensional trivial representation and the 
$(n{+}1)$-dimensional natural representation.
This seems to hold for most of the operations: the universal envelopes are finite dimensional
and the only irreducible representations are the trivial representation and the natural representation.

\begin{conjecture}
Let $A_{p,q}$ ($p \le q$) be the simple associative triple system consisting of $(p{+}q) \times (p{+}q)$ block matrices 
of the form
  \[
  \begin{bmatrix}
  0 & p \times q \\
  q \times p & 0
  \end{bmatrix}.
  \]
Let $\omega$ be one of the following trilinear operations from Table \ref{trilinearoperations}:
Jordan ($q = 0, 1, \frac12$), 
anti-Jordan ($q = \infty, -1, 2$),
fourth family (all cases),
cyclic commutator, weakly commutative, weakly anticommutative.
Then, with finitely many exceptions, $U(A^\omega)$ is finite dimensional 
and semisimple and is isomorphic to $F \oplus M_{p+q}(F)$.
\end{conjecture}

The operations not included in this conjecture are the first three operations 
(the symmetric, alternating, and cyclic sums), 
together with the four classical operations
(the Lie, anti-Lie, Jordan, and anti-Jordan triple products).
These seem likely to be the operations producing nonassociative triple systems with 
the most interesting representation theory.
This is well-known for the four classical operations, owing to their close connection
with Lie and Jordan algebras and superalgebras.
On the other hand, very little is known about the representation theory of nonassociative triple systems
arising from the first three operations.

\begin{openproblem}
Study the structure of the universal associative envelopes,
and classify the finite dimensional irreducible representations, 
for the nonassociative triple systems $A_{p,q}^\omega$ where
$\omega$ is the symmetric, alternating, and cyclic sum.
\end{openproblem}


\section{Bibliographical Remarks} \label{bibliographicalremarks}

The historical origins of the theory of Gr\"obner bases are complex, with similar ideas discovered 
in different contexts at different times by different people. 

\subsection{The commutative case}

The most famous branch of the theory, owing to its close connections with algebraic geometry, 
is that of commutative Gr\"obner bases.
Many of these ideas can be traced back to the work of Macaulay; his 1916 monograph on 
\emph{The Algebraic Theory of Modular Systems}
is available online \cite{Macaulay1916}.
The original work of Gr\"obner most often cited as the origin of the theory of commutative Gr\"obner bases
is his 1939 paper on linear differential equations \cite{Grobner1939};
this has appeared in English translation \cite{Grobner2009} with commentary by the translator \cite{Abramson2009}.
The modern form of the theory which emphasizes the algorithmic aspects originated in the 1965 Ph.D. thesis of
Buchberger which has been translated into English \cite{BuchbergerThesis} with commentary by the author
\cite{Buchberger2006}; see also his 1970 paper \cite{Buchberger1970}.
There are many textbooks on the theory of commutative Gr\"obner bases and their applications; see
Adams and Loustaunau \cite{Adams1994},
Becker and Weispfennig \cite{Becker1993},
Cox et al. \cite{Cox1992}, 
Ene and Herzog \cite{Ene2012}, and
Fr\"oberg \cite{Froberg1997}.

\subsection{The noncommutative case}

The theory of noncommutative Gr\"obner bases seems to have originated with the Russian school of nonassociative algebra;
see the papers of Zhukov \cite{Zhukov1950} and especially Shirshov \cite{Shirshov1962a,Shirshov1962}.
The first systematic statements of the Composition (Diamond) Lemma in the noncommutative case,
and its application to the proof of the PBW theorem, were published almost simultaneously by
Bokut \cite{Bokut1976} and Bergman \cite{Bergman1978}.
The latter paper traces the origins of the theory to earlier work of Newman \cite{Newman1942}.
The computational complexity of algorithms for constructing noncommutative Gr\"obner bases 
has been studied by F. Mora \cite{FMora1986}.
The Ph.D. thesis of Keller \cite{Keller1997,Keller1998} on noncommutative Gr\"obner bases 
led to the software package Opal \cite{Green1997}.
A more recent software package, with extensive online documentation, 
has been developed by Cohen and Gijsbers \cite{CohenGijsbers2007}.
For some important papers on theory and algorithms for noncommutative Gr\"obner bases, see
Borges-Trenard et al. \cite{Borges2000},
Gerritzen \cite{Gerritzen1998},
Green et al. \cite{GreenMoraUfnarovski2000}, and
Kang et al. \cite{Kang2007}.
For a connection between commutative and noncommutative Gr\"obner bases, see Eisenbud et al. \cite{Eisenbud1998}.
For an extension to noncommutative power series, see Gerritzen and Holtkamp  \cite{GerritzenHoltkamp1998}.
For textbooks on noncommutative Gr\"obner bases, see
Bokut and Kukin \cite{BokutKukin1994},
Bueso et al. \cite{Bueso2003}, and
Li \cite{Li2012}.

\subsection{The nonassociative case}

The most important branch of the nonassociative theory deals with Gr\"obner-Shirshov bases for free Lie algebras;
see Bokut and Chibrikov \cite{BokutChibrikov2006} and Bokut and Chen \cite{BokutChen2007}.
A theory of Gr\"obner-Shirshov bases in free nonassociative algebras has been developed by
Gerritzen \cite{Gerritzen2000,Gerritzen2006} and Rajaee \cite{Rajaee2006}.
For related work on Sabinin algebras, see 
Shestakov and Umirbaev \cite{ShestakovUmirbaev2002},
P\'erez-Izquierdo \cite{Perez2007},
and
Chibrikov \cite{Chibrikov2011}.

\subsection{Loday algebras}

An active area of current research is extending the Composition (Diamond) Lemma
from associative algebras to the dialgebras and dendriform algebras introduced by Loday 
\cite{Loday1993,Loday1995,Loday2002}.
For associative dialgebras, see  Bokut et al. \cite{BokutChenLiu2010}.
For dendriform algebras, see Bokut et al. \cite{BokutChenHuang2013}, Chen and Wang \cite{ChenWang2010},
as well as the papers on Rota-Baxter algebras by Bokut et al. \cite{BokutChenDeng2010,BokutChenQiu2010},
Chen and Mo \cite{ChenMo2011}, Qiu \cite{Qiu2013}, and Guo et al \cite{GuoSitZhang2013}.
It is an open problem to extend these results further to the quadri-algebras of Aguiar and Loday
\cite{AguiarLoday2004}, and to the Koszul dual of quadri-algebras introduced by Vallette 
\cite[\S 5.6]{Vallette2008}.
For Leibniz algebras, which are the analogues of Lie algebras in the setting of dialgebras, 
see Loday and Pirashvili \cite{LodayPirashvili1993}, Aymon and Grivel \cite{AymonGrivel2003},
Casas et al. \cite{CasasInsuaLadra2007}, Insua and Ladra \cite{InsuaLadra2009}.
For pre-Lie algebras, which are the analogue of Lie algebras in the setting of dendriform algebras,
see Bokut et al. \cite{BokutChenLi2008}.
For L-dendriform algebras, which are the analogue of Lie algebras in the setting of quadri-algebras, 
see Bai et al. \cite{BaiLiuNi2010}.
(For corresponding generalizations of Jordan algebras, see Hou et al. \cite{HouBai2012,HouNiBai2013}.)

\subsection{Survey papers}

A survey of commutative and noncommutative Gr\"obner bases from the point of view of theoretical computer science 
has been written by T. Mora \cite{TMora1994}.
For an introduction to noncommutative Gr\"obner bases from the point of view of computer algebra,
see Green \cite{Green1994,Green1999} and Ufnarovski \cite{Ufnarovski1998}.
A number of introductory surveys of Gr\"obner-Shirshov bases in associative and 
nonassociative algebras have been written by Bokut and his co-authors: see
Bokut \cite{Bokut1999}, Bokut and Kolesnikov \cite{BokutKolesnikov2003}, and Bokut and Shum \cite{BokutShum2005}.


\section*{Acknowledgements}

I thank NSERC (Natural Sciences and Engineering Research Council of Canada) 
for financial support through a Discovery Grant, and the faculty and staff of CIMAT 
(Centro de Investigaci\'on en Matem\'aticas, Guanajuato, Mexico) 
for their hospitality during the Research School on 
\emph{Associative and Nonassociative Algebras and Dialgebras: Theory and Algorithms - 
In Honour of Jean-Louis Loday (1946--2012)} 
from February 17 to March 2, 2013 which was sponsored by CIMPA 
(Centre International de Math\'ematiques Pures et Appliqu\'ees).
I thank my former Ph.D. student Hader Elgendy for pointing out some errors
in an earlier version of these notes.



\begin{thebibliography}{999}

\bibitem{Abramson2009}
\textsc{M. P. Abramson}:
Historical background to Gr\"obner's paper.
\textit{ACM Commun. Comput. Algebra} 
43 (2009) no.~1-2, 22--23.

\bibitem{Adams1994}
\textsc{W. W. Adams, P. Loustaunau}:
\textit{An Introduction to Gr\"obner Bases}.
Graduate Studies in Mathematics, 3. 
American Mathematical Society, Providence, RI, 1994. 

\bibitem{AguiarLoday2004}
\textsc{M. Aguiar, J.-L. Loday}:
Quadri-algebras.
\textit{J. Pure Appl. Algebra} 
191 (2004) no.~3, 205--221.

\bibitem{AymonGrivel2003}
\textsc{M. Aymon, P.-P. Grivel}:
Un th\'eor\`eme de Poincar\'e-Birkhoff-Witt pour les alg\`ebres de Leibniz.
\textit{Comm. Algebra} 
31 (2003) no.~2, 527--544. 

\bibitem{BaiLiuNi2010}
\textsc{C. Bai, L. Liu, X. Ni}:
Some results on L-dendriform algebras.
\textit{J. Geom. Phys.} 
60 (2010) no.~6-8, 940--950.

\bibitem{Bashir2008}
\textsc{S. Bashir}:
\textit{Automorphisms of Simple Anti-Jordan Pairs}.
Ph.D. Thesis, University of Ottawa, Canada, 2008.

\bibitem{Becker1993}
\textsc{T. Becker, V. Weispfenning}:
\textit{Gr\"obner Bases: A Computational Approach to Commutative Algebra}. 
Graduate Texts in Mathematics, 141. 
Springer-Verlag, New York, 1993. 

\bibitem{BenkartRoby1998}
\textsc{G. Benkart, T. Roby}:
Down-up algebras.
\textit{J. Algebra} 
209 (1998) no.~1, 305--344.
Addendum: ``Down-up algebras''.
\textit{J. Algebra} 
213 (1999), no.~1, 378. 

\bibitem{Bergman1978}
\textsc{G. M. Bergman}:
The diamond lemma for ring theory.
\textit{Adv. in Math.} 
29 (1978) no.~2, 178--218. 

\bibitem{BirkhoffWhitman1949}
\textsc{G. Birkhoff, P. M. Whitman}:
Representation of Jordan and Lie algebras.
\textit{Trans. Amer. Math. Soc.} 
65 (1949) 116--136. 

\bibitem{Bokut1976}
\textsc{L. A. Bokut}:
Imbeddings into simple associative algebras.
\textit{Algebra i Logika} 
15 (1976) no.~2, 117--142.

\bibitem{Bokut1999}
\textsc{L. A. Bokut}:
The method of Gr\"obner-Shirshov bases.
\textit{Siberian Adv. Math.} 
9 (1999) no.~3, 1--16. 

\bibitem{BokutChen2007}
\textsc{L. A. Bokut, Y. Chen}:
Gr\"obner-Shirshov bases for Lie algebras: after A. I. Shirshov.
\textit{Southeast Asian Bull. Math.} 
31 (2007) no.~6, 1057--1076.

\bibitem{BokutChenDeng2010}
\textsc{L. A. Bokut, Y. Chen, X. Deng}:
Gr\"obner-Shirshov bases for Rota-Baxter algebras. 
\textit{Sibirsk. Mat. Zh.} 
51 (2010) no.~6, 1237--1250.

\bibitem{BokutChenHuang2013}
\textsc{L. A. Bokut, Y. Chen, J. Huang}:
Gr\"obner-Shirshov bases for $L$-algebras.
\textit{Internat. J. Algebra Comput.} 
(to appear).
\texttt{arXiv:1005.0118 [math.RA]}

\bibitem{BokutChenLi2008}
\textsc{L. A. Bokut, Y. Chen, Y. Li}:
Gr\"obner-Shirshov bases for Vinberg-Koszul-Gerstenhaber right-symmetric algebras. 
\textit{Fundam. Prikl. Mat.} 
14 (2008) no.~8, 55--67.

\bibitem{BokutChenLiu2010}
\textsc{L. A. Bokut, Y. Chen, C. Liu}:
Gr\"obner-Shirshov bases for dialgebras.
\textit{Internat. J. Algebra Comput.} 
20 (2010) no.~3, 391--415. 

\bibitem{BokutChenQiu2010}
\textsc{L. A. Bokut, Y. Chen, X. Deng}:
Gr\"obner-Shirshov bases for associative algebras with multiple operators and free Rota-Baxter algebras.
\textit{J. Pure Appl. Algebra} 
214 (2010) no.~1, 89--100. 

\bibitem{BokutChibrikov2006}
\textsc{L. A. Bokut, E. S. Chibrikov}:
Lyndon-Shirshov words, Gr\"obner-Shirshov bases, and free Lie algebras.
\textit{Non-associative Algebra and Its Applications}, pages 17--39.
Lect. Notes Pure Appl. Math., 246, Chapman \& Hall/CRC, Boca Raton, 2006. 

\bibitem{BokutKolesnikov2003}
\textsc{L. A. Bokut, P. S. Kolesnikov}:
Gr\"obner-Shirshov bases: from their incipiency to the present.
\textit{J. Math. Sci.}
Vol. 116, No. 1, 2003.

\bibitem{BokutKukin1994}
\textsc{L. A. Bokut, G. P. Kukin}:
\textit{Algorithmic and Combinatorial Algebra}.
Mathematics and its Applications, 255. 
Kluwer Academic Publishers Group, Dordrecht, 1994. 

\bibitem{BokutShum2005}
\textsc{L. A. Bokut, K. P. Shum}:
Gr\"obner and Gr\"obner-Shirshov bases in algebra: an elementary approach.
\textit{Southeast Asian Bull. Math.} 
29 (2005) no.~2, 227--252. 

\bibitem{Borges2000}
\textsc{M. A. Borges-Trenard, M. Borges-Quintana, T. Mora}:
Computing Gr\"obner bases by FGLM techniques in a non-commutative setting.
\textit{J. Symbolic Comput.} 
30 (2000) no.~4, 429--449. 

\bibitem{Bremner2011}
\textsc{M. R. Bremner}:
How to compute the Wedderburn decomposition of a finite-dimensional associative algebra.
\textit{Groups Complex. Cryptol.} 
3 (2011) no.~1, 47--66. 

\bibitem{Bremner2012}
\textsc{M. R. Bremner}:
Algebras, dialgebras, and polynomial identities.
\textit{Serdica Math. J.} 
38 (2012) 91--136. 

\bibitem{BremnerHentzel2000}
\textsc{M. R. Bremner, I. R. Hentzel}:
Identities for generalized Lie and Jordan products on totally associative triple systems.
\textit{J. Algebra} 
231 (2000) no.~1, 387--405. 

\bibitem{BremnerMadariaga2013}
\textsc{M. R. Bremner, S. Madariaga}:
Polynomial identities for tangent algebras of monoassociative loops.
\textit{Comm. Algebra}
(to appear)

\bibitem{BremnerPeresi2007}
\textsc{M. R. Bremner, L. A. Peresi}:
Classification of trilinear operations.
\textit{Comm. Algebra} 
35 (2007) no.~9, 2932--2959. 

\bibitem{BremnerPeresi2009}
\textsc{M. R. Bremner, L. A. Peresi}:
An application of lattice basis reduction to polynomial identities for algebraic structures.
\textit{Linear Algebra Appl.} 
430 (2009) no.~2-3, 642--659.

\bibitem{BremnerPeresi2009b}
\textsc{M. R. Bremner, L. A. Peresi}:
Polynomial identities for the ternary cyclic sum.
\emph{Linear Multilinear Algebra} 
57 (2009) no.~6, 595--608. 

\bibitem{BuchbergerThesis}
\textsc{B. Buchberger}:
\textit{An Algorithm for Finding the Basis Elements of the Residue Class Ring of a Zero Dimensional Polynomial Ideal}.
Translated from the 1965 German original by Michael P. Abramson.
\textit{J. Symbolic Comput.} 
41 (2006) no.~3-4, 475--511. 

\bibitem{Buchberger1970}
\textsc{B. Buchberger}:
Ein algorithmisches Kriterium f\"ur die L\"osbarkeit eines algebraischen Gleichungssystems.
\textit{Aequationes Math.} 
4 (1970) 374--383. 

\bibitem{Buchberger2006}
\textsc{B. Buchberger}:
Comments on the translation of my PhD thesis: 
``An Algorithm for Finding the Basis Elements of the Residue Class Ring of a Zero Dimensional Polynomial Ideal''
\textit{J. Symbolic Comput.} 
41 (2006) no.~3-4, 471--474. 

\bibitem{Bueso2003}
\textsc{J. Bueso, J. G\'omez-Torrecillas, A. Verschoren}:
\textit{Algorithmic Methods in Noncommutative Algebra: Applications to Quantum Groups}. 
Mathematical Modelling: Theory and Applications, 17. 
Kluwer Academic Publishers, Dordrecht, 2003. 

\bibitem{Carlsson1980}
\textsc{R. Carlsson}:
$n$-ary algebras.
\textit{Nagoya Math. J.} 
78 (1980) 45--56. 

\bibitem{CasasInsuaLadra2007}
\textsc{J. M. Casas, M. A. Insua, M. Ladra}:
Poincar\'e-Birkhoff-Witt theorem for Leibniz $n$-algebras.
\textit{J. Symbolic Comput.} 
42 (2007) no.~11-12, 1052--1065. 

\bibitem{ChenMo2011}
\textsc{Y. Chen, Q. Mo}:
Embedding dendriform algebra into its universal enveloping Rota-Baxter algebra.
\textit{Proc. Amer. Math. Soc.} 
139 (2011) no.~12, 4207--4216.

\bibitem{ChenWang2010}
\textsc{Y. Chen, B. Wang}:
Gr\"obner-Shirshov bases and Hilbert series of free dendriform algebras.
\textit{Southeast Asian Bull. Math.} 
34 (2010) no.~4, 639--650. 

\bibitem{Chibrikov2011}
\textsc{E. S. Chibrikov}:
On free Sabinin algebras.
\textit{Comm. Algebra} 
39 (2011) no.~11, 4014--4035. 

\bibitem{Clifton1981}
\textsc{J. M. Clifton}:
A simplification of the computation of the natural representation of the symmetric group $S_n$.
\textit{Proc. Amer. Math. Soc.} 
83 (1981) no.~2, 248--250. 

\bibitem{CohenGijsbers2007}
\textsc{A. M. Cohen, D. A. H. Gijsbers}:
\textit{Documentation on the GBNP Package}. Available at: \hfill \break
\texttt{http://www.win.tue.nl/{\textasciitilde}amc/pub/grobner/doc.html}
(accessed 19 January 2013)

\bibitem{Cox1992}
\textsc{D. Cox, J. Little, D. O'Shea}:
\textit{Ideals, Varieties, and Algorithms: An Introduction to Computational Algebraic Geometry and Commutative Algebra}. Undergraduate Texts in Mathematics. 
Springer-Verlag, New York, 1992. 

\bibitem{deGraaf2000}
\textsc{W. A. de Graaf}:
\textit{Lie Algebras: Theory and Algorithms}.
North-Holland Mathematical Library, 56. 
North-Holland Publishing Co., Amsterdam, 2000. 

\bibitem{Eisenbud1998}
\textsc{D. Eisenbud, I. Peeva, B. Sturmfels}:
Non-commutative Gr\"obner bases for commutative algebras.
\textit{Proc. Amer. Math. Soc.} 
126 (1998) no.~3, 687--691. 

\bibitem{ElgendyThesis}
\textsc{H. A. Elgendy}:
\textit{Polynomial Identities and Enveloping Algebras for $n$-ary Structures}.
Ph.D. thesis, University of Saskatchewan, Canada, 2012.

\bibitem{Elgendy2013}
\textsc{H. A. Elgendy}:
Universal associative envelopes of nonassociative triple systems.
\textit{Comm. Algebra} (to appear).
\texttt{arXiv:1211.4243 [math.RA]}

\bibitem{ElgendyBremner2012}
\textsc{H. A. Elgendy, M. R. Bremner}:
Universal associative envelopes of $(n{+}1)$-dimensional $n$-Lie algebras.
\textit{Comm. Algebra} 
40 (2012) no.~5, 1827--1842. 

\bibitem{Ene2012}
\textsc{V. Ene, J. Herzog}:
\textit{Gr\"obner Bases in Commutative Algebra}.
Graduate Studies in Mathematics, 130. 
American Mathematical Society, Providence, RI, 2012. 

\bibitem{Faulkner1985}
\textsc{J. R. Faulkner}:
Identity classification in triple systems.
\textit{J. Algebra} 
94 (1985) no.~2, 352--363. 

\bibitem{FaulknerFerrar1980}
\textsc{J. R. Faulkner, J. C. Ferrar}:
Simple anti-Jordan pairs.
\textit{Comm. Algebra} 
8 (1980) no.~11, 993--1013.

\bibitem{Froberg1997}
\textsc{R. Fr\"oberg}:
\textit{An Introduction to Gr\"obner Bases}.
Pure and Applied Mathematics (New York). 
John Wiley \& Sons, Ltd., Chichester, 1997. 

\bibitem{Gerritzen1998}
\textsc{L. Gerritzen}:
On infinite Gr\"obner bases in free algebras.
\textit{Indag. Math. (N.S.)} 
9 (1998) no.~4, 491--501. 

\bibitem{Gerritzen2000}
\textsc{L. Gerritzen}:
Hilbert series and non-associative Gr\"obner bases.
\textit{Manuscripta Math.} 
103 (2000) no.~2, 161--167.

\bibitem{Gerritzen2006}
\textsc{L. Gerritzen}:
Tree polynomials and non-associative Gr\"obner bases.
\textit{J. Symbolic Comput.} 
41 (2006) no.~3-4, 297--316.

\bibitem{GerritzenHoltkamp1998}
\textsc{L. Gerritzen, R. Holtkamp}:
On Gr\"obner bases of noncommutative power series.
\textit{Indag. Math. (N.S.)} 
9 (1998) no.~4, 503--519. 

\bibitem{Glennie1966}
\textsc{C. M. Glennie}:
Some identities valid in special Jordan algebras but not valid in all Jordan algebras.
\textit{Pacific J. Math.} 
16 (1966) 47--59. 

\bibitem{Green1994}
\textsc{E. L. Green}:
An introduction to noncommutative Gr\"obner bases. 
\textit{Computational Algebra}, pp. 167--190,
Lecture Notes in Pure and Appl. Math., 151. 
Dekker, New York, 1994. 

\bibitem{Green1999}
\textsc{E. L. Green}:
Noncommutative Gr\"obner bases, and projective resolutions. 
\textit{Computational Methods for Representations of Groups and Algebras}, pp.~29--60.
Progr. Math., 173. 
Birkh\"auser, Basel, 1999. 

\bibitem{Green1997}
\textsc{E. L. Green, L. S. Heath, B. J. Keller}:
Opal: a system for computing noncommutative Gr\"obner bases.
\textit{Rewriting Techniques and Applications}, pp.~331--334.
Lecture Notes in Computer Science, 1232.
Springer, 1997.

\bibitem{GreenMoraUfnarovski2000}
\textsc{E. L. Green, T. Mora, V. Ufnarovski}:
The non-commutative Gr\"obner freaks. 
\textit{Symbolic Rewriting Techniques (Ascona, 1995)}, 
93Ð104,
Progr. Comput. Sci. Appl. Logic, 15, Birkh\"auser, Basel, 1998. 

\bibitem{Grobner1939}
\textsc{W. Gr\"obner}:
\"Uber die algebraischen Eigenschaften der Integrale von linearen Differentialgleichungen mit konstanten Koeffizienten.
\textit{Monatsh. Math. Phys.} 
47 (1939) no.~1, 247--284. 

\bibitem{Grobner2009}
\textsc{W. Gr\"obner}:
On the algebraic properties of integrals of linear differential equations with constant coefficients.
Translated from the German by Michael Abramson.
\textit{ACM Commun. Comput. Algebra} 
43 (2009) no.~1-2, 24--46. 

\bibitem{GuoSitZhang2013}
\textsc{L. Guo, W. Sit, R. Zhang}:
Differential type operators and Gr\"obner-Shirshov bases.
\textit{J. Symbolic Comput.} 
52 (2013) 97--123. 

\bibitem{Hestenes1962}
\textsc{M. R. Hestenes}:
A ternary algebra with applications to matrices and linear transformations.
\textit{Arch. Rational Mech. Anal.} 
11 (1962) 138--194. 

\bibitem{HodgeParshall2002}
\textsc{T. L. Hodge, B. J, Parshall}:
On the representation theory of Lie triple systems.
\textit{Trans. Amer. Math. Soc.} 
354 (2002) no.~11, 4359Ð--4391.

\bibitem{HouBai2012}
\textsc{D. Hou, C. Bai}:
J-dendriform algebras.
\textit{Front. Math. China} 
7 (2012) no.~1, 29--49. 

\bibitem{HouNiBai2013}
\textsc{D. Hou, X. Ni, C. Bai}:
Pre-Jordan algebras. 
\textit{Math. Scand.} 
(to appear)

\bibitem{InsuaLadra2009}
\textsc{M. A. Insua, M. Ladra}:
Gr\"obner bases in universal enveloping algebras of Leibniz algebras.
\textit{J. Symbolic Comput.} 
44 (2009) no.~5, 517--526. 

\bibitem{Jacobson}
\textsc{N. Jacobson}:
\textit{Structure and Representations of Jordan Algebras}.
American Mathematical Society Colloquium Publications, Vol.~XXXIX. 
American Mathematical Society, Providence, R.I., 1968.

\bibitem{Kang2007}
\textsc{S.-J. Kang, D.-I. Lee, K.-H. Lee, H. Park}:
Linear algebraic approach to Gr\"obner-Shirshov basis theory.
\textit{J. Algebra} 
313 (2007) no.~2, 988--1004. 

\bibitem{Keller1997}
\textsc{B. J. Keller}:
\textit{Algorithms and Orders for Finding Noncommutative Gr\"obner Bases}.
Ph.D. Thesis, Virginia Polytechnic Institute and State University, 1997.

\bibitem{Keller1998}
\textsc{B. J. Keller}:
Alternatives in implementing noncommutative Gr\"obner basis systems.
\textit{Symbolic Rewriting Techniques}, pp.~ 105--126.
Progr. Comput. Sci. Appl. Logic, 15.
Birkh\"user, Basel, 1998. 

\bibitem{Li2012}
\textsc{H. Li}:
\textit{Gr\"obner Bases in Ring Theory}. 
World Scientific Publishing Co. Pte. Ltd., Hackensack, NJ, 2012. 

\bibitem{Lister1952}
\textsc{W. G. Lister}:
A structure theory of Lie triple systems.
\textit{Trans. Amer. Math. Soc.} 
72 (1952) 217--242. 

\bibitem{Lister1971}
\textsc{W. G. Lister}:
Ternary rings.
\textit{Trans. Amer. Math. Soc.} 
154 (1971) 37--55. 

\bibitem{Loday1993}
\textsc{J.-L. Loday}:
Une version non commutative des alg\`ebres de Lie: les alg\`ebres de Leibniz.
\textit{Enseign. Math.} 
(2) 39 (1993) no.~3-4, 269--293. 

\bibitem{Loday1995}
\textsc{J.-L. Loday}:
Alg\`ebres ayant deux op\'erations associatives (dig\`ebres).
\textit{C. R. Acad. Sci. Paris S\'er. I Math.} 
321 (1995) no.~2, 141--146. 

\bibitem{Loday2002}
\textsc{J.-L. Loday}:
Dialgebras. 
\textit{Dialgebras and Related Operads}, pp. 7--66.
Lecture Notes in Math., 1763, Springer, Berlin, 2001.

\bibitem{LodayPirashvili1993}
\textsc{J.-L. Loday, T. Pirashvili}:
Universal enveloping algebras of Leibniz algebras and (co)homology.
\textit{Math. Ann.} 
296 (1993) no.~1, 139--158. 

\bibitem{Loos1971}
\textsc{O. Loos}:
\textit{Lectures on Jordan Triples}. 
The University of British Columbia, Vancouver, Canada, 1971.

\bibitem{Loos1972}
\textsc{O. Loos}:
Assoziative Tripelsysteme.
\textit{Manuscripta Math.}
7 (1972) 103--112. 

\bibitem{Macaulay1916}
\textsc{F. S. Macaulay}:
\emph{The Algebraic Theory of Modular Systems}.
Revised reprint of the 1916 original. 
With an introduction by Paul Roberts. 
Cambridge Mathematical Library. 
Cambridge University Press, Cambridge, 1994. 
\texttt{http://archive.org/details/algebraictheoryo00macauoft}

\bibitem{McCrimmon2004}
\textsc{K. McCrimmon}:
\textit{A Taste of Jordan Algebras}.
Universitext. Springer-Verlag, New York, 2004.

\bibitem{Meyberg1972}
\textsc{K. Meyberg}:
\textit{Lectures on Algebras and Triple Systems}.
Notes on a course of lectures given during the academic year 1971-1972. 
The University of Virginia, Charlottesville, 1972.

\bibitem{Mikhalev1998}
\textsc{A. A. Mikhalev, A. A. Zolotykh}:
Standard Gr\"obner-Shirshov bases of free algebras over rings. I. Free associative algebras.
\textit{Internat. J. Algebra Comput.} 
8 (1998) no.~6, 689--726. 

\bibitem{FMora1986}
\textsc{F. Mora}:
Groebner bases for noncommutative polynomial rings.
\textit{Algebraic Algorithms and Error-Correcting Codes},
Lecture Notes in Computer Science, 229, pp.~353--362.
Springer, Berlin, 1986.

\bibitem{TMora1994}
\textsc{T. Mora}:
An introduction	to commutative and noncommutative Gr\"obner bases.
\textit{Theoret. Comput. Sci.} 
134 (1994) 131--173.

\bibitem{Musson2012}
\textsc{I. M. Musson}:
\textit{Lie Superalgebras and Enveloping Algebras}.
Graduate Studies in Mathematics, 131. 
American Mathematical Society, Providence, 2012. 

\bibitem{Newman1942}
\textsc{M. H. A. Newman}:
On theories with a combinatorial definition of ``equivalence''.
\textit{Annals of Math.}
43, 2 (1942) 223--243.

\bibitem{OEIS}
\textit{On-line Encyclopedia of Integer Sequences}:
\texttt{http://oeis.org/}

\bibitem{Perez2005}
\textsc{J.-M. P\'erez-Izquierdo}:
An envelope for Bol algebras.
\textit{J. Algebra} 
284 (2005) no.~2, 480--493. 

\bibitem{Perez2007}
\textsc{J.-M. P\'erez-Izquierdo}:
Algebras, hyperalgebras, nonassociative bialgebras and loops.
\textit{Adv. Math.} 
208 (2007) no.~2, 834--876. 

\bibitem{PerezShestakov2004}
\textsc{J.-M. P\'erez-Izquierdo, I. P. Shestakov}:
An envelope for Malcev algebras.
\textit{J. Algebra} 
272 (2004) no.~1, 379--393. 

\bibitem{Qiu2013}
\textsc{J. Qiu}:
Gr\"obner-Shirshov bases for commutative algebras with multiple operators and free commutative Rota-Baxter algebras.
\texttt{arXiv:1301.5018}

\bibitem{Rajaee2006}
\textsc{S. Rajaee}:
Non-associative Gr\"obner bases.
\textit{J. Symbolic Comput.} 
41 (2006) no.~8, 887--904.

\bibitem{Rutherford1948}
\textsc{D. E. Rutherford}:
\textit{Substitutional Analysis}. 
Edinburgh, at the University Press, 1948.

\bibitem{ShestakovUmirbaev2002}
\textsc{I. P. Shestakov, U. U, Umirbaev}:
Free Akivis algebras, primitive elements, and hyperalgebras.
\textit{J. Algebra} 
250 (2002) no.~2, 533--548. 

\bibitem{Shirshov1962a}
\textsc{A. I. Shirshov}:
Some algorithmic problems for $\epsilon$-algebras.
\textit{Sibirsk. Mat. Zh.}
3 (1962) 132--137. 

\bibitem{Shirshov1962}
\textsc{A. I. Shirshov}:
On a hypothesis in the theory of Lie algebras.
\textit{Sibirsk. Mat. Zh.}
3 (1962) 297--301 (1962).

\bibitem{Shirshov2009}
\textsc{A. I. Shirshov}:
\textit{Selected Works of A. I. Shirshov}. 
Translated from the Russian by M. R. Bremner and M. V. Kotchetov. 
Edited by L. A. Bokut, V. N. Latyshev, I. P. Shestakov and E. Zelmanov. 
Contemporary Mathematicians. 
Birkh\"auser Verlag, Basel, 2009. 

\bibitem{Ufnarovski1998}
\textsc{V. S. Ufnarovski}:
Introduction to noncommutative Gr\"obner bases theory.
\textit{Gr\"obner Bases and Applications}, pp.~259--280.
London Math. Soc. Lecture Note Ser., 251.
Cambridge Univ. Press, Cambridge, 1998. 

\bibitem{Vallette2008}
\textsc{B. Vallette}:
Manin products, Koszul duality, Loday algebras and Deligne conjecture.
\textit{J. Reine Angew. Math.} 
620 (2008) 105--164. 

\bibitem{Young1977}
\textsc{A. Young}:
\textit{The Collected Papers of Alfred Young (1873--1940)}.
With a foreword by G. de B. Robinson and a biography by H. W. Turnbull. 
Mathematical Expositions, No.~21. University of Toronto Press, 1977.

\bibitem{Zhukov1950}
\textsc{A. I. Zhukov}:
Reduced systems of defining relations in non-associative algebras.
\textit{Mat. Sbornik N.S.} 
27(69) (1950) 267--280. 

\end{thebibliography}
\end{document}